\documentclass[12,reqno]{amsart}
\usepackage{amsfonts}
\usepackage{amssymb}
\usepackage{a4wide}
\usepackage{pgf,pgfarrows,pgfautomata,pgfheaps,pgfnodes,pgfshade}
\usepackage{url}
\usepackage{bbm}

\usepackage{graphicx}

\usepackage[numbers]{natbib}

\usepackage{hyperref}
\numberwithin{equation}{section}

\newcommand{\diff}{\,\mathrm{d}}
\newcommand{\diffns}{\mathrm{d}}
\newcommand{\hata}{\hat\alpha}

\newtheorem{defi}{Definition}[section]
\newtheorem{thm}[defi]{Theorem}
\newtheorem{lemm}[defi]{Lemma}
\newtheorem{remark}[defi]{Remark}
\newtheorem{cor}[defi]{Corollary}
\newtheorem{assum}[defi]{Assumption}
\newtheorem{prop}[defi]{Proposition}
\newtheorem{prob}[defi]{Problem}

\newcommand{\Q}{\mathbb{Q}}

\newcommand{\R}{\mathbb{R}}

\newcommand{\E}{\mathbb{E}}
\newcommand{\PP}{\mathbb{P}}

\DeclareMathOperator{\sgn}{sgn}

\newcommand{\norm}[1]{\left\Vert#1\right\Vert}
\newcommand{\abs}[1]{\left\vert#1\right\vert}

\begin{document}

	\title[ SMP for systems with B-V drifts]{Stochastic Optimal Control for Systems with Drifts of Bounded Variation: A Maximum Principle Approach}
	\author[A.-M. Bogso]{Antoine-Marie Bogso}
	\address{Faculty of Sciences, Department of Mathematics, University of Yaounde I,
		P.O. Box 812, Yaounde, Cameroon}
	\email{antoine.bogso@facsciences-uy1.cm} 
	\author[R. Likibi Pellat]{Rhoss Likibi Pellat}
	\address{Laboartoire J. A Dieudonn\'e, CNRS, Universit\'e C\^ote d'Azur, Nice, France}
	\email{rhoss.likibi-pellat@univ-cotedazur.fr}
	
	\author[W. Kuissi Kamdem]{Wilfried Kuissi Kamdem}
	\address{African Institute for Mathematical Sciences, Ghana}
	\address{Department of Mathematics, University of Rwanda, Ghana}
	\email{donatien@aims.edu.gh}
	\author[O. Menoukeu Pamen]{Olivier Menoukeu Pamen}
	\address{IFAM, Department of Mathematical Sciences, University of Liverpool, L69 7ZL, United Kingdom}
	\email{menoukeu@liverpool.ac.uk}

	\date{}

	\keywords{Stochastic Maximum Principle; Sobolev differentiable flow; Ekeland's variational principle; Local time; Malliavin calculus.}
	
	\subjclass[2010]{Primary: 60H10, 60H07, 49J55, Secondary: 60J55, ,60H30} 
	\maketitle
	\begin{abstract}
		We study a stochastic control problem for nonlinear systems governed by
		stochastic differential equations with irregular drift. The drift
		coefficient is assumed to decompose as
		$b(t,x,a)=b_1(t,x)+b_2(x)b_3(t,a)$, where $b_1$ is bounded and Borel
		measurable, $b_2$ has bounded variation, and $b_3$ is bounded and
		smooth. Under these minimal regularity assumptions, we establish a
		Pontryagin--type stochastic maximum principle. The analysis relies on
		new results for SDEs with random drift of bounded variation, including
		existence, uniqueness, and Malliavin--Sobolev differentiability of the
		state process. A key ingredient is an explicit representation of the
		first variation process obtained via integration with respect to the
		space--time local time of bounded variation processes. By combining a
		suitable approximation scheme with Ekeland’s variational principle, and
		using a Garcia--Rodemich--Rumsey inequality to obtain a uniform control
		of the first variation, we derive the maximum principle. As an application, we derive an optimal corridor-type capital adjustment policy for an insurance surplus model.
	\end{abstract}
	\section{Introduction}
	In this paper, we consider stochastic control problems of nonlinear systems,
	where the control domain is convex and the controlled state is given by a stochastic differential equation of the following type:
	\begin{equation}
		\label{eqconpb1}
		X^\alpha_t = x + \int_0^tb(u,X^{\alpha}_u,\alpha_u)\mathrm{d}u +\sigma B_t,
	\end{equation}
	where $b$ is a deterministic function and $\sigma$ is a constant vector, $B$ is a $d$-dimensional Brownian motion on a probability space $(\Omega, {\mathcal F}, P)$ equipped with the complete filtration $({\mathcal F}_t)_{t\in[0,T]}$. The control variable $\alpha:[0,T]\times \Omega \mapsto \mathbb{A}\subset \mathbb{R}$ is $\mathcal{B}([0,T])\otimes \mathcal{F}$-measurable, and $({\mathcal F}_t)_{t\in[0,T]}$-adapted.
	
	The criterion to be minimised over the set of admissible controls $\mathcal{A}$ (see Definition \ref{def:admissible}) is given by the following form.
	\begin{align} 
		J(\alpha):= \mathbb{E}\Big[ \int_0^T  f(s,X^\alpha_s,\alpha_s)\diff s + g(X^\alpha_T)\Big].\label{perfunct1}
	\end{align}
	An admissible control is considered optimal if it provides a solution to the above problem. We assume that an optimal control exists. Our primary objective is to establish necessary and sufficient Pontryagin-type conditions for optimality in this problem when the drift coefficient is bounded and measurable.
	
	The stochastic maximum principle (SMP), which provides necessary conditions for optimality in stochastic control problems, is a powerful mathematical tool that employs a probabilistic approach. Since its introduction by Kushner (\cite{Kush72}), under the assumption that the coefficients are sufficiently regular and the set of control is convex, it has been extensively studied (see for example \cite{Ben81, Bis78, Haus86}). 
	
	The key idea behind SMP is to introduce adjoint processes and derive a Hamiltonian function that characterises the optimal control strategy. The convexity assumption was relaxed in \cite{Pen90}, leading to the derivation of an SMP in a global form, expressed in terms of first- and second-order adjoint equations. 
	Under additional convexity conditions, this necessary condition is sufficient. This method is particularly useful for solving problems in finance, economics, engineering, and other fields where uncertainty plays a crucial role (see \cite{YZ99}).

	The SMP is typically derived under the assumption that the coefficients of the state process are differentiable, making it inapplicable to optimal control problems with non-smooth coefficients. To the best of our knowledge, the first work to address the stochastic maximum principle with non-smooth coefficients was by Mezerdi \cite{Mer88}, where the author established a necessary condition of optimality for a problem with a Lipschitz continuous drift—though not necessarily differentiable in the state and control variables. This result has been extended in several directions (see \cite{BCDM09, BDM07}), with all extensions still requiring some form of smoothness in the coefficients.
	
	In a recent study \cite{MenTan22}, the authors established a necessary and sufficient stochastic maximum principle for the optimal control of systems governed by stochastic differential equations with irregular drift coefficients. More precisely, they assumed $b$ satisfies $b(t,x,a):= b_1(t,x) + b_2(t,x,a)$ where $b_1$ is a bounded, Borel measurable function and $b_2$ is bounded measurable, and continuously differentiable in its second and third variables with bounded derivatives. To address the challenges posed by irregularity, they first derived an explicit representation of the first variation process (in the Sobolev sense) of the controlled process. They then constructed a sequence of optimal control problems with smooth coefficients using an approximation argument and applied Ekeland's variational principle to obtain an approximating adjoint process. Finally, by passing to the limit, they derived the stochastic maximum principle. Similar approach was also used to study SMP for McKean-Vlasov SDEs with singular drift in \cite{TBMP24}.

	Our work is closely related to \cite{MenTan22}, with the key distinction that we assume the drift function takes the form 
	$b(t,x,a):= b_1(t,x) +b_2(x) \times b_3(t,a)$, where $b_1$ is bounded and measurable, $b_2$
	is of bounded variation and $b_3$ is bounded and continuously differentiable in the second variable. Our motivation stems from optimal capital adjustment or dividend withdrawal for an insurance firm  in which the dynamics of its surplus process evolves under a \emph{corridor-type control} mechanism.
	(see Section \ref{motivexam} for further details). The primary challenge in this study lies in handling the irregularity of $b_2$. We address this by first deriving an existence,uniqueness and smoothness results (in Malliavin and Sobolev sense) for SDEs with random drift of the form $b(t,x,\omega):= b_1(t,x)+b_2(x) \times b_3(t,\omega)$. We then employ integration with respect to space-time local time for bounded variation random processes (see Proposition \ref{prop:EisenbaumBV}) to obtain an explicit representation of the first variation process (in the Sobolev sense) of the controlled process. The SMP is then obtained by an approximation argument. 
	
	Moreover, despite the lack of smoothness in $b_1, b_2$ one can leverage their boundedness, the mean-value theorem, and the Girsanov theorem to establish the $\mathrm{d}t\times \PP$ a.e convergence of the derivative of the approximating Hamiltonian with respect to the control (see the proof of Theorem \ref{thm:necc}). Furthermore, providing a uniform control of the supremum norm of the difference between the approximated solution of the SDE and the actual solution is a crucial component of our analysis (see Lemma \ref{lem:conv.Xnn}). To achieve this, we use a Garcia-Rodemich-Rumsey-type result.	Overall, our approach provides a novel framework for studying stochastic control problems with irregular drift.

	SDEs with random coefficients have been investigated in the literature, often under stronger assumptions. In particular, when both the initial condition and the drift coefficient are allowed to anticipate the future of the driving Brownian motion, the authors in \cite{OcPa89} employ a generalised Itô–Ventzel formula for anticipating integrands to study Stratonovich-type SDEs. They show that if the initial condition is Malliavin differentiable, the drift is Malliavin smooth, and both its spatial and Malliavin derivatives exhibit polynomial growth, then the SDE admits a unique, non-exploding, Malliavin differentiable solution. In \cite{HLN97}, the authors establish the existence and uniqueness of solutions to SDEs with random initial conditions and random coefficients, provided the stochastic integral is interpreted in the generalised Stratonovich sense and the coefficients satisfy Lipschitz continuity. More recently, in \cite{MenTan19}, under the assumption that the drift takes the form $b(t,x,\omega)=b_1(t,x)+b_2(t,x,\omega)$, where $b_1$ has spatial linear growth and $b_2$ is smooth in its second variable and satisfies some integrability conditions, the existence, uniqueness, and Malliavin smoothness of a strong solution to SDE \eqref{eq:SDErvmain1} are established. This is achieved using a purely probabilistic approach based on Malliavin calculus and a compactness criterion.
	
	Another contribution of this work is the establishment of existence, uniqueness, and smoothness of solutions to the SDE \eqref{eq:SDErvmain1}, under the assumption that the drift coefficient takes the form $	b(t,\omega,x)= b_1(t,x)+ b_2(x)b_3(t,\omega)$, where $b_1$ is Borel measurable and bounded, $b_2$ is of bounded variation and $b_3$ satisfies certain integrability conditions. The primary challenge arises from the fact that $b_2$ is of bounded variation. To address this, we first define an integral with respect to space-time local time for such integrands. Then, by applying the estimate in Proposition \ref{propSha1}, we show the relative compactness of a sequence of approximated solutions—a crucial step in our proof. As a byproduct of our approach, we also establish the Malliavin differentiability of the solution. Furthermore, we show that the solution is Sobolev differentiable with respect to its initial condition. In addition, we derive a uniform bound on the supremum norm of the derivative with respect to the initial condition (see Proposition \ref{theosup1} and Corollary \ref{corsupclasder}). Despite the roughness of the drift introducing, we overcome the difficulty by leveraging a Garcia-Rodemich-Rumsey-type inequality to derive the required bound. The above uniform bound is key for proving our stochastic maximum principle.


	The remainder of the paper is organised as follows. 
	Section~\ref{motivexam} presents a motivating example together with the general formulation of the control problem. 
	Section~\ref{secmaxprinres} establishes the necessary and sufficient maximum principles. 
	Section~\ref{SectRanSDE} provides results on the existence, uniqueness, and regularity of stochastic differential equations with random drift. 
	Section~\ref{seclocaltime} addresses space–time integration with respect to the local time of random processes of bounded variation. 
	Section~\ref{appli} applies the preceding results to derive the optimal capital adjustment or dividend withdrawal strategy for an insurance firm. 
	Section~\ref{proffmainr} contains the detailed proofs of the main theoretical results, while the Appendix collects supplementary proofs and a technical result.

	\section{A motivating example and main results}\label{motivexam}
	\subsection{Probabilistic setting and notation}
	Let $T \in (0,\infty)$ and $d \in \mathbb{N}$ be fixed and consider a probability space $(\Omega, {\mathcal F}, \PP)$ equipped with the completed filtration $({\mathcal F}_t)_{t\in[0,T]}$ of a $d$-dimensional Brownian motion $B$.
	Throughout the paper, the product $\Omega \times [0,T]$ is endowed with the predictable $\sigma$-algebra. Subsets of $\mathbb{R}^k$, $k\in\mathbb{N}$, are always endowed with the Borel $\sigma$-algebra induced by the Euclidean norm $|\cdot|$. The interval $[0,T]$ is equipped with the Lebesgue measure. Unless otherwise stated, all equalities and inequalities between random variables and processes will be understood in the $\PP$-almost sure and $\PP\otimes \diffns t$-almost sure sense, respectively.
	For $p \in [1, \infty]$ and $k\in\mathbb{N}$, denote by ${\mathcal S}^p(\mathbb{R}^k)$
	the space of all adapted continuous processes $X$
	with values in $\mathbb{R}^k$ such that
	$\norm{X}_{{\mathcal{S}}^p(\mathbb{R}^k)}^p :=  \E[(\sup\nolimits_{t \in [0,T]}\abs{X_t})^p] < \infty$,
	and by ${\mathcal H}^p(\mathbb{R}^{k})$ the space of all predictable processes $Z$ with values in $\mathbb{R}^{k}$ such that $\norm{Z}_{{\mathcal H}^p(\mathbb{R}^{k})}^p := \E[(\int_0^T\abs{Z_u}^2\diffns u)^{p/2}] < \infty$.
	\subsection{Motivating example} \label{subsec:motivating}
	
	This subsection presents a motivating example that illustrates the type of stochastic control problems considered in this paper. Specifically, we describe the dynamics of a regulated surplus process evolving under a \emph{corridor-type control} mechanism. The model captures the objective of maintaining the surplus within a prescribed target interval while accounting for stochastic fluctuations arising from claims, operational uncertainty, or external market shocks.
	
	\medskip
	\noindent
	Let $X_t$ denote the surplus (or reserve) level at time $t \in [0,T]$. Its evolution is governed by the stochastic differential equation
	\begin{equation}\label{eqmot1}
		\mathrm{d}X_t
		= \big( b_1(X_t) - \alpha_t\,\mathrm{sign}(X_t)\mathbf{1}_{\{|X_t|>H\}} \big)\mathrm{d}t
		+ \sigma\,\mathrm{d}B_t,
		\qquad X_0 = x_0,
	\end{equation}
	where:
	\begin{itemize}
		\item $b_1(x)$ represents the \emph{baseline drift} (or \emph{liability rate}) of the surplus, describing the net deterministic inflow or outflow of cash (e.g., premiums minus claims);
		\item $\alpha_t$ is a progressively measurable control variable representing managerial interventions such as capital adjustments or dividend withdrawals, activated when the surplus exits the corridor  $[-\rho,\rho],\,\,\rho>0$;
		\item $B_t$ is a standard Brownian motion;
		\item $\sigma > 0$ denotes the volatility coefficient associated with random fluctuations.
	\end{itemize}

	\medskip
	In a more general formulation, one may include a proportional term $\delta_t X_t$ in the drift to capture exponential accumulation due to interest rates or investment returns.  
	However, in many practical settings this term is negligible or intentionally excluded, leading to $\delta_t \equiv 0$.  
	This assumption is justified by the following considerations:
	\begin{enumerate}
		\item \textsl{Additive cash-flow dynamics:} The system’s evolution depends primarily on inflows and outflows rather than on the absolute surplus level, resulting in an additive (rather than multiplicative) drift structure.
		\item \textsl{Non–interest-bearing reserves:} The reserves are not assumed to accumulate interest, removing the need for a proportional growth component.
		\item \textsl{Stability-oriented control:} The primary management objective is to maintain stability around a target level rather than to induce surplus growth.
	\end{enumerate}

	\medskip
	We adopt a bounded and smooth drift of the form
	\begin{equation}\label{eq:boundedb}
		b_1(x) = \mu\,\tanh(x/M), 
		\qquad \mu, M > 0.
	\end{equation}
	This functional form captures a \emph{saturation effect}: for large surplus or deficit values, the rate of change stabilizes at $\pm \mu$, reflecting limited sensitivity to extreme states.  
	From an analytical standpoint, $b$ satisfies:
	\begin{itemize}
		\item $b_1(-x) = -b_1(x)$ (\textsl{oddness}), ensuring symmetric response to positive and negative deviations;
		\item $|b_1(x)| \le \mu$ (\textsl{boundedness}), providing realistic growth limits and simplifying the stability analysis.
	\end{itemize}

	\medskip
	
	\noindent\textbf{Control objective.}
	Let $\mathcal{A}$ denote the set of admissible control processes taking values in a closed convex set $\mathbb{A}=[-1,1]\subset\mathbb{R}$.  
	The decision maker aims to select a control $\alpha \in \mathcal{A}$ that minimizes the expected deviation of the surplus from its target level.  
	In its simplest form, this can be expressed as the minimization of a terminal quadratic cost:
	\begin{equation}\label{eq:quadcost_example}
		V^0(x)
		= \inf_{\alpha \in \mathcal{A}} 
		\mathbb{E}\big[(X_T^{\alpha,x})^2\big]
		= \mathbb{E}\big[(X_T^{\hat \alpha, x})^2\big],
	\end{equation}
	where $\hat{\alpha}$ denotes an optimal control (if it exists).  
	More generally, one may consider an instantaneous running cost $f$ and a terminal cost $g$, as formulated in the subsequent section.
	
	Although the term $b_2(x)$ in \eqref{eqmot1} is discontinuous at the corridor boundaries $\pm \rho$, it can still be seen as a bounded nondecreasing function, and thus is of bounded variation.  
	However, optimal control problems involving \emph{bounded variation drifts} fall outside the scope of standard stochastic control theory. Furthermore, recent work by \cite{MenTan22}, which addresses bounded and measurable drift coefficients, does not directly apply in this case due to the involvement of a term $\alpha_t\times b_2(x)$, where $b_2$ is non-Lipschitz. Hence, it motivates the study of optimal control problems with bounded drift, as developed in the next section.

	\subsection{General Problem Formulation}

		Motivated by the preceding example, we consider the following controlled diffusion system:
	\begin{equation}
		\label{eqconpb111}
		X^{\alpha,x}_t = x + \int_0^tb(u,X^{\alpha,x}_u,\alpha_u)\mathrm{d}u +\sigma B_t.
	\end{equation} 
	$B$ is a Brownian motion on a probability space $(\Omega,\mathcal{F},\mathbb{P})$ with filtration $(\mathcal{F}_t)_{t\in[0,T]}$. 
	
	The objective functional is
	\begin{equation}
		J(\alpha)
		= \mathbb{E}\Big[ \int_0^T f(s,X^{\alpha,x}_s,\alpha_s)\,\mathrm{d}s 
		+ g(X_T^{\alpha,x}) \Big].
		\label{eq:performance}
	\end{equation}
	The optimal control problem consists of finding
	\begin{equation}
		V(x) = \inf_{\alpha\in\mathcal{A}} J(\alpha)
		= J(\hat\alpha),
		\label{eq:value}
	\end{equation}
	where $\hat\alpha$ denotes an optimal control.
	
	\begin{defi}[Admissible controls]
		\label{def:admissible}
		Let $\mathbb{A}\subseteq\R$ be a closed convex set. 
		The set of admissible controls is
		\begin{multline*}
			\mathcal{A} := \Big\{\alpha:[0,T]\times \Omega\to \mathbb{A}, \text{ progressive, \eqref{eqconpb1}}\text{ has a unique strong solution and }\\ \E\sup_{t\in[0,T]}|\alpha(t)|^{4}< C \text{ for some } C>0 \Big\}.
		\end{multline*}
		
	\end{defi}

	We assume the following conditions on $f$ and $g$
	\begin{assum}
		\label{mainassum2}
		The functions $f:[0,T]\times\mathbb{R}\times\mathbb{R}\to\mathbb{R}$ and $g:\mathbb{R}\to\mathbb{R}$ satisfy:
		\begin{enumerate}
			\item $f$ and $g$ are continuously differentiable in their arguments;
			\item there exists $C>0$ such that
			\[
			|f(t,x,a)| + |\partial_x f(t,x,a)| + |\partial_a f(t,x,a)|
			\le C(1+|x|^2+|a|^2), \quad \text{for all } (t,x,a)
			\]
			and
			\begin{equation*}
				|g(x)| + |\partial_xg(x)| \le C(1 + |x|^2).
			\end{equation*}
		\end{enumerate}
	\end{assum}
	
	In the next section we state necessary and sufficient stochastic maximum principle for the above problem, allowing the drift coefficient to be only of bounded variation in the state variable.
	
	\section{Stochastic Maximum Principle with Irregular Drift}\label{secmaxprinres}
	
	In this section, we derive necessary and sufficient optimality conditions 
	for the control problem introduced in Section~2. 
	Recall that the controlled state process 
	\(X^{\alpha,x}=\{X_t^{\alpha,x}\}_{t\in[0,T]}\) satisfies
	\[
	\mathrm{d}X_t^{\alpha,x}
	= b(t,X_t^{\alpha,x},\alpha_t)\,\mathrm{d}t + \sigma\,\mathrm{d}B_t,
	\qquad X_0^{\alpha,x}=x,
	\]
	where the drift coefficient \(b\) satisfies the following structural 
	assumptions
	\begin{assum}
		\label{mainassum20}
		$$b(t,x,a):= b_1(t,x) +b_2(x)b_3(t,a),$$ 
		where
		\begin{enumerate}
			\item $b_1$ bounded, Borel measurable function;
			\item $b_2$ of bounded variation;
			\item $b_3$ is bounded measurable, and continuously differentiable in its second variable with bounded derivatives.
		\end{enumerate}	
	\end{assum}
	Our goal is to characterise the optimal control 
	\(\hat\alpha\in\mathcal A\) that minimises \(J(\alpha)\).
	We first state a necessary maximum principle, followed by a sufficient one.
	\begin{thm}[Necessary SMP]
		\label{thm:necc}
		Suppose Assumptions \ref{mainassum2} and \ref{mainassum20} hold.
	Let $\hat\alpha \in \mathcal{A}$ be an optimal control of \eqref{eq:value} and let $X^{\hat\alpha,x}$ be the associated optimal trajectory. Then the first variation $\Phi^{\hat\alpha}$ of $X^{\hat\alpha,x}$ (i.e. $\Phi^{\hat\alpha}=\partial_xX^{\hat\alpha,x}$) is well-defined and it holds 
	\begin{equation}
		\label{eq:nec.cond}
		\partial_{a}H(t, X^{\hat\alpha,x}_t,Y^{\hat\alpha}_t,\hat\alpha_t )\cdot(\beta - \hat\alpha_t) \ge 0 \quad \mathbb{P}\otimes \diff t\text{-a.s. for all } \beta \in \mathcal{A},
	\end{equation} 
	where $H$ is the Hamiltonian defined by $	H(t,x,y,a) := f(t, x,a) + b(t,x,a)y$ and $Y^{\hat\alpha}$ satisfies
	\begin{equation}
		\label{eq:adj.proc}
		Y^{\hat\alpha}_t := \mathbb{E}\big[\Phi^{\hat\alpha}_{t,T} \partial_xg( X^{\hat\alpha,x}_T) + \int_t^T\Phi^{\hat\alpha}_{t,s} \partial_xf(s, X^{\hat\alpha,x}_s, \hat\alpha_s))\mathrm{d}s\mid \mathcal{F}_t \big],
	\end{equation}	
	where 
	\begin{equation*} 
		\Phi^\alpha_{t,s}  =e^{-\int_t^s\int_{\mathbb{R}} b(u,z,\alpha_u) L^{X^{\alpha,x}}(\diffns u,\diffns z)} , \,\,\,0\leq t\leq s\leq T,
	\end{equation*}
	and the double integral above is taken over both time and space with respect to the local time of the process $X^{\alpha,x}$
	\end{thm}
	\begin{proof}[Idea of the proof]
		We prove the theorem in several steps. First, we construct a sequence of optimal control problems with smooth coefficients using an approximation argument and apply \emph{Ekeland’s variational principle} to obtain an approximating adjoint process. Then, by passing to the limit, we derive the \emph{stochastic maximum principle}. 
		
		One of the main challenges in this approach lies in establishing the almost everywhere convergence of the derivative of the approximating Hamiltonian with respect to the control variable. This difficulty is overcome by obtaining a uniform bound on the supremum norm of the difference between the approximated solution of the SDE and the actual solution, using a \emph{Garcia--Rodemich--Rumsey-type inequality} (see Lemma~\ref{lem:conv.Xnn}). The complete proof is provided in Subsection \ref{proofteonecmaxp} 
	\end{proof}
	
	\begin{thm}[Sufficient SMP]
		\label{thm:suff}
			Suppose Assumptions \ref{mainassum2} and \ref{mainassum20} hold. Let $\alpha \in \mathcal{A}$ and let $X^{\hat \alpha}$ be the associated trajectory and $Y^{\hat \alpha}$ be the corresponding adjoint given by \eqref{eq:adj.proc}. In addition, assume that $g$ and $(x,a)\mapsto H(t,x,Y^{\hat \alpha}_t,a)$ are concave for all $t\in [0,T]$. Further assume that 
		\begin{equation}
			\label{eq:suff.con}
			\partial_{a} H(t, X^{\hat\alpha,x}_t, Y^{\hat\alpha}_t, \hat\alpha_t)=0, \quad \mathbb{P}\otimes \diff t\text{-a.s.,}
		\end{equation}
		Then, $\hat\alpha$ is an optimal control.
	\end{thm}
	\begin{proof}
		See Subsection \ref{proofteosufmaxp}.
	\end{proof}
	
	In the next section, we justify that the SDEs involved indeed possess strong solutions and differentiable flows under bounded variation drift assumptions

	
	\section{SDEs with Random Irregular Drift: Existence, Uniqueness, and Regularity}\label{SectRanSDE}
	
	In this section, we consider the SDE
	\begin{equation}
		\label{eq:SDE sumprod}
		X^x_t = x + \int_0^t\big( b_1(u,X^x_u) +b_{2}(X^x_u)b_3(u,\omega)\big)\diffns u +\sigma B_t.
	\end{equation}
	To ease the notation, we introduce the function
	\begin{equation}\label{eqdriftassu}
		b(t,\omega,x):= b_1(t,x)+ b_2(x)b_3(t,\omega),
	\end{equation}
	Then, equation \eqref{eq:SDE sumprod} now takes the form
	\begin{equation}	\label{eq:SDErvmain1}
		X^x_t = x + \int_0^t b(u,\omega, X^x_u)\,du + \sigma B_t.
	\end{equation}
	We suppose the following:
	\begin{assum}\label{mainassum}
		\leavevmode
		\begin{enumerate}
			\item[(AX1)] \begin{itemize}
				\item The function $ b_1:[0,T]\times \mathbb{R}\to \mathbb{R}$ is Borel measurable and bounded.
				\item The function $b_2: \mathbb{R}\to \mathbb{R}$ is Borel measurable and of bounded variation. 
				
				\item  $b_3:[0,T]\times \Omega\to \mathbb{R}$ is an adapted bounded variation process, that is Malliavin differentiable for every $t$, and $b_3$ as well as its Malliavin derivative are bounded. That is, there is a positive random variable $M(\omega)$ such that 
				\begin{equation*}
					|\frac{\partial b_3}{\partial t}(t,\omega)|+	|b_3(t,\omega)| + |D^j_tb_3(t,\omega)| \le M(\omega),\,\,\,j=1,\ldots,d,
				\end{equation*}
				where $D_t^j$ is the $j$-th component of the Malliavin derivative of $b_3(t,\omega)$				and $M(\omega)$ satisfies 
				$$
				\mathbb E[e^{qM^2(\omega)}]<\infty \text{ for all }  q>0.
				$$
				\item  $\sigma \in \mathbb{R}^d$ and $|\sigma|^2>0$.		
			\end{itemize}
			\item[(AX2)]  There exist constants $C, \theta  > 0 $ such that
			$$\mathbb{E}[| D_tb_3(s,  \omega )  -  D_{t^\prime} b_3(s, \omega )|^4] \leq  C| t^\prime- t|^\theta .$$
			\label{a1}
		\end{enumerate}
	\end{assum}
	We start by recalling the following result from \cite[Proposition 2.1]{Shaposhnikov_2016}
	\begin{prop}\label{propSha1}
		Let $b\in C([0,T],C^1_b(\mathbb{R},\mathbb{R}))$. There exists a constant $C$ only dependent on $\|b\|_\infty$ and a constant $k$ that does not depend on $b$ such that the following inequality holds:
		$$
		\mathbb{E} \exp\big\{k\big|\int_0^T \partial_xb(t,B_t)\diffns t\big|^2\big|\big\}\leq C,
		$$
		where $\partial_xb$ denotes the spacial derivative of $b$. 
	\end{prop}
	\subsection{Existence, uniqueness and Malliavin differentiability}\label{secmaldif}
	
	In this section we discuss existence, uniqueness and Malliavin differentiability of the solution to the SDE \eqref{eq:SDE sumprod}. Since Malliavin calculus will play an important role in our arguments, we briefly introduce the spaces of Malliavin differentiable random variables and stochastic processes
	${\mathcal D}^{1,p}(\mathbb{R}^l)$ and ${\mathcal L}^{1,p}_a(\mathbb{R}^l)$, $p,l,\ge 1$. For a thorough treatment of the theory of Malliavin calculus we refer to \cite{Nua06}.
	Let ${\mathcal M}$ be the class of smooth random variables $\xi=(\xi^1,\dots,\xi^l)$ of the form
	\begin{equation*}
		\xi^i = \varphi^i\big(\int_0^T h^{i1}_s\diffns B_s, \dots, \int_0^T h^{in}_s\diffns B_s \big),
	\end{equation*}
	where $\varphi^i$ is in the space $C_{\text{poly}}^\infty(\mathbb{R}^{n};\mathbb{R})$  of infinitely continuously differentiable functions whose partial derivatives have polynomial growth, $h^{i1}, \dots, h^{in} \in L^2([0,T]; \mathbb{R}^d)$ and $n\ge 1$.
	For every $\xi$ in ${\mathcal M}$ let the operator $D = (D^1, \dots, D^d):{\mathcal M}\to L^2(\Omega\times [0,T];\mathbb{R}^d)$ be  given by
	\begin{equation*}
		D_t\xi^i := \sum_{j=1}^n\frac{\partial \varphi^i}{\partial x_{j}}\big(\int_0^T h^{i1}_s\diffns B_s, \dots, \int_0^T h^{in}_s\diffns B_s \big)h^{ij}_t, \quad 0\le t\le T,\,\, 1\le i\le l,
	\end{equation*}
	and the norm $	\norm{\xi}_{1,p} := \big(\E[\abs{\xi}^p + \int_0^T\abs{D_t\xi}^p\diffns t  ] \big)^{1/p}$.
	As shown in \cite{Nua06}, the operator $D$ extends to the closure ${\mathcal D}^{1,p}(\mathbb{R}^l)$ of the set ${\mathcal M}$ with respect to the norm $\norm{\cdot}_{1,p}$.
	A random variable $\xi$ is Malliavin differentiable if $\xi \in {\mathcal D}^{1,p}(\mathbb{R}^l)$ and we denote by $D_t\xi$ its Malliavin derivative.
	Denote by ${\mathcal L}^{1,p}_a(\mathbb{R}^{l})$ the space of processes $Y \in {\mathcal H}^2(\mathbb{R}^{l})$ such that
	$Y_t \in {\mathcal D}^{1,p}(\mathbb{R}^{l})$ for all $t \in [0,T]$, the process $DY_t$ admits a square integrable progressively measurable version and
	\begin{equation*}
		\norm{Y}_{{\mathcal L}^{1,p}_a(\mathbb{R}^l)}^p := \norm{Y}_{{\mathcal H}^p(\mathbb{R}^l)} + \E\big[\int_0^T\int_0^T\abs{D_r Y_t}^p\diffns r\diffns t \big] < \infty.
	\end{equation*}
	
	The main result of this subsection is the following:
	\begin{thm}
		\label{thm:sumprod}
		Suppose Assumption \ref{mainassum} is in force. Then the SDE \eqref{eq:SDE sumprod} admits a unique strong Malliavin differentiable solution. 
	\end{thm}
	
	The above result can be extended to the multidimensional case under an additional commutativity condition on the spatial derivative of the approximating drift (See for example \cite{MBP10}). We formulate this extension in the following remark.
	\begin{remark}
		Let $X^{x,n}_t$ be the strong solution to  \eqref{eq:SDE sumprod} associated with the drift $b_n$, for $n \ge 1$, and let $b_n'(\cdot, X^{x,n}_{\cdot})$ denote the derivative of $b_n$ with respect to the spatial variable. Then $b_n'(\cdot, X^{x,n}_{\cdot})$ is a continuous matrix-valued function. Assume that for all $n \ge 1$ and all $0 \le s \le t \le 1$,
		\[
		b_n'(t, X^{x,n}_{t})
		\big( \int_s^t b_n'(u, X^{x,n}_{u})\,\mathrm{d}u \big)
		=
		\big( \int_s^t b_n'(u, X^{x,n}_{u})\,\mathrm{d}u \big)
		b_n'(t, X^{x,n}_{t}).
		\]
		Under this commutativity condition, the Malliavin derivative admits an explicit representation. Consequently, by arguments similar to those in \cite{MBP10}, one can establish the existence and uniqueness of a strong solution.
	\end{remark}

	\begin{proof}[Idea of the Proof]
		The proof of the above result proceeds in three main steps:
		
		\begin{enumerate}
			\item \textsl{Existence of a weak solution.}  
			We first establish that the given stochastic differential equation admits a weak solution.
			
			\item \textsl{Approximation by strong solutions.}  
			Let $(X^{x,n}_{t})_{n \in \mathbb{N}}$ be a sequence of strong solutions to the approximating equations \eqref{eq:SDE sumprod}, where the drift coefficients 
			$b_n= b_{1,n}+b_{2,n}b_3$ replace $b=b_1+b_2b_3$, respectively. More precisely, for each $n \ge 1$, define
			\[
			b_n := b_{1,n} + b_{2,n} \, b_3,
			\]
			where $\{b_{1,n}\}_{n \ge 1}$ and $\{b_{2,n}\}_{n \ge 1}$ are sequences of smooth functions with compact support such that $\{b_{2,n}\}_{n \ge 1}\in \text{BV}(\mathbb R)$ and 
			\[
			b_{1,n}(t,x) \to b_1(t,x), \quad (t,x)\text{-a.e. as } n \to \infty, 
			\quad \text{and } \sup_{n \ge 1} \|b_{1,n}\|_{\infty} < \infty,
			\]
			and
			\[
			b_{2,n}(x) \to b_2(x), \quad (t,x)\text{-a.e. as } n \to \infty, 
			\quad \text{and } \sup_{n \ge 1} \|b_{2,n}\|_{\infty} < \infty.
			\]
			In this step, we show that the sequence \((X^{x,n}_{t})_{n \in \mathbb{N}}\) converges strongly in \(L^2(\Omega,\mathbb{P})\) to \(\E[X^x_{t}|\mathcal{F}_t]\).
			
			\item \textsl{Identification and uniqueness.}  
			We then prove that \(\E[X^x_{t}|\mathcal{F}_t] = X^x_t\), implying that \(X^x\) is indeed a strong solution.  
			Finally, we establish pathwise uniqueness of this solution.
		\end{enumerate}
		A complete proof of this result is provided in Subsection \ref{proofteomaldif}.
	\end{proof}
	
		Note that the existence of $X^{x,n}$ is ensured by \cite{HLN97}.
	By applying the chain rule formula (see, e.g., \cite{Nua06}) and using the Malliavin differentiability of $b_2$ we deduce that $X^{x,n}$ is also Malliavin differentiable, and we have for $i=1,\ldots,d,$
	\begin{align}\label{malldifeqr}
		D^i_tX^{x,n}_s =
		& \sigma+ \int_t^sb_{2,n}(X^{x,n}_u) D^i_tb_3(u,\omega)\diffns u\notag\\
		&+ \int_t^s\big\{\partial_xb_{1,n}(u,X^{x,n}_u)+ b^\prime_{2,n}(X^{x,n}_u) b_3(u,\omega)\big\}D^i_t X^{x,n}_u \diffns u.
	\end{align}

	The following proposition provides a key estimate for the Malliavin derivative of $X^{n,x}$ and plays an important role in the proof of Theorem \ref{thm:sumprod}.

	\begin{prop} \label{lemmainres1r}
		The strong solution $X^{x,n}$ of the SDE \eqref{eq:SDE sumprod} with $b_n$ replacing $b$ satisfies
		$$
		\E \left[ | D_t X^{x,n}_s - D_{t'} X^{x,n}_s |^2 \right] \leq {\mathcal C}(\|b\|_{\infty}) |t -t'|^{\alpha}
		$$ 
		for all $s\in [0,T]$; $0 \leq t' \leq t \leq T$ and some $\alpha = \alpha(s) > 0$.
		Moreover, 
		$$
		\sup_{0 \leq t \leq T} \E \left[ | D_t X^{x,n}_s |^2 \right] \leq {\mathcal C}(\|b\|_{\infty}),
		$$
		where the function ${\mathcal C}(\cdot, \cdot, \cdot): [0, \infty)^3 \rightarrow [0, \infty)$ is continuous and increasing in each components 
		and $\|b_1\|_{\infty}$
	\end{prop}
	The proof of the above proposition will need the following lemma:
	\begin{lemm}
		\label{lem:auxlemma}
		Let $c\ge1$ be a fixed constant.
		Then, it holds
		\begin{equation}
			\E\Big[\Big\{\exp\Big(
			c\int_{t'}^t\big\{\partial_xb_{1,n}(u,X^{x,n}_u)+b^\prime_{2,n}(X^{x,n}_u)\cdot b_3(u,\omega)\big\}\diffns u\Big)-1\Big\}^p\Big]\le  {\mathcal C}(p,\|b\|_{\infty})|t-t'|^{p/2}.
		\end{equation}
	\end{lemm}
	\begin{proof}
		See Appendix \ref{profauxlem}.
	\end{proof}
	\begin{proof}[Proof Proposition \ref{lemmainres1r}]
		See Subsection \ref{proofteomaldif}.
	\end{proof}
	The main consequence of this proposition is the following compactness result:
	\begin{cor} \label{maincor}
		For each $0 \leq t \leq T$, the sequence $(X^{x,n}_{t})_{n \geq 1}$ is relatively compact in $L^2(\Omega, \mathcal{F}, \PP)$.
	\end{cor}
	\begin{proof}
		This follows by Proposition \ref{lemmainres1r} and the compactness criteria of \cite{DPMN92}.
	\end{proof}
	
	The next Theorem gives the explicit representation of the Malliavin derivative of the solution to the SDE.
	\begin{thm}\label{Thmexplimalder}
		Assume that the conditions of Theorem \ref{thm:sumprod} are satisfied. Suppose that
		\begin{equation}
			\label{eq: power moment b21}
			\sup_{0\le s\le T}\E\big[\big(\int_0^T|D^i_sb_3(t,\omega+\varphi) |^2\diffns t \big)^2 \big]<\infty,\quad i=1,\ldots,d
		\end{equation}
		for every $\varphi \in C^1_b([0,T],\mathbb{R}^d)$. 
		For every $0\leq s\leq t\leq T$, the $i$-th component of the Malliavin derivative of the unique strong solution to the SDE \eqref{eq:SDE sumprod} admits the following representation:
		\begin{align}\label{ExpliciMalder}
			D^i_tX_s^x=& e^{-\int_t^s\int_{\mathbb{R}} b(u,z,\omega) L^{X^x}(\diffns u,\diffns z)}\notag\\
			&\times\Big(\int_t^sb_{2}(X_u)\cdot D^i_tb_3(u,\omega)e^{-\int_t^u \int_{\mathbb{R}}b(r,z,\omega) L^{X^x}(\diffns r,\diffns z)}\diffns u + \sigma_i \Big),\, i =1,\dots,d.
		\end{align}

	\end{thm}
	\begin{proof}
		See Appendix \ref{proofthmexplmal}.
	\end{proof}
	\subsection{Sobolev differentiability}\label{secsobdif}
	Let $X^{s,x}_\cdot$ be the solution to the SDE \eqref{eq:sde into} with initial condition $X_s=x$.
	As is well-known, the solution $X^{s,x}$ may not belong to the Sobolev space $W^{1,p}(\mathbb{R}, dx)$, $p>1$.
	Thus, following the intuition of \cite{MNP2015} we will show that $X^{s,x}$ belongs to a weighted Sobolev space. Let $w:\mathbb{R}\to (0,\infty)$ be a Borel-measurable (weight) function such that
	\begin{equation*}
		\int_\mathbb{R}e^{c|x|^2}w(x)\diffns x<\infty
	\end{equation*}
	for every $c\ge0$. Let $W^{1,p}(\mathbb{R}, w)$ be the weighted Sobolev space of functions $u:\mathbb{R}\to\mathbb{R}$ such that, it holds
	\begin{equation*}
		||u||_{1, p, w} := \big(\int_\mathbb{R}|u(x)|^pw(x)\diffns x \big)^{1/p} + \big(\int_\mathbb{R}|u'(x)|^pw(x)\diffns x \big)^{1/p}<\infty,
	\end{equation*}
	where $u'$ is the weak derivative of $u$.

	The main result in this section is the following Sobolev differentiability of the solution to \eqref{eq:SDE sumprod}.  
	\begin{thm}
		\label{thm:sumprod1}
		Suppose Assumption \ref{mainassum} is in force. 
		Let $X^{x}$ be the unique solution of \eqref{eq:SDE sumprod}.
		Then for every $p
		\ge2$, the map $x \mapsto X^{x}_t$ belongs to $L^2(\Omega, W^{1,p}(\mathbb{R}, w))$. In addition, we have the following explicit representation for its first variation in the Sobolev sense
		\begin{equation} \label{eqfirstvar1}
			\Phi_{s,t}:=\partial_x X^{s,x}_t  =e^{-\int_s^t\int_{\mathbb{R}} b(u,z,\omega) L^{X^x}(\diffns u,\diffns z)} .
		\end{equation} 
	\end{thm}
	\begin{proof}The proof of the first part is provided in Subsection \ref{proofteosobdif}, whereas the explicit representation \eqref{eqfirstvar1} follows by arguments similar to those used for \eqref{ExpliciMalder}, with the main difference that $\Phi^{x,n}_{s,t}$ converges to $\Phi^{x}_{s,t}$ only weakly in $L^2(U\times\Omega)$.
	\end{proof}
	As in the previous section, we denote $b_n: = b_{1,n} + b_{2,n}b_3$ where $b_{1,n}:[0,T]\times \mathbb{R}\to \mathbb{R}$ (resp. $b_{2,n}:\mathbb{R}\to \mathbb{R}$), is the sequence of smooth functions with compact support approximating $b_1$ (resp. $b_2$) as introduced previously. We further denote by $X^{s,x,n}$ the unique solution of the SDE associated to $b_n$ with initial condition $X^{s,x,n}_s = x$.
	That is,
	\begin{equation}\label{eq:sde into}
		X^{s,x,n}_t = x + \int_s^tb_{n}(u, X^{s,x,n}_u,\omega)\diffns u + \sigma\cdot B_t.
	\end{equation}
	We first derive a uniform (in $n$) bound on the derivative of the solution $X^{s,x,n}$ to the SDE \eqref{eq:sde into} which will play a key role in the proof of Theorem \ref{thm:sumprod1}.
	\begin{prop}
		\label{pro:bound derivative}
		Let $p\ge1$ and $(s,x) \in [0,T]\times \mathbb{R}$ be fixed.
		Then for every $t \in [0,T]$, almost every trajectories of $x \mapsto X^{s,x,n}_t$ is differentiable and it holds
		\begin{align}\label{eqapprosobder1}
			\underset{x\in \mathbb{R}}{\sup}\,\E\left[|\partial_x X^{s,x,n}_t|^p \right] &\le \tilde{C}(p,\sigma,\|b\|_\infty)
		\end{align}
		for some positive constant $\tilde{C}(p,\sigma,\|b\|_\infty)$ depending on $p,T,\sigma$ and $\|b\|_\infty$.
	\end{prop}
	\begin{proof}
		See Subsection \ref{proofteosobdif}.
	\end{proof}
	
	\begin{cor}\label{corholdx}
		Let $p\ge 2$, then it holds
		\begin{equation*}
			\E\left[ |X^{s, x_1}_{t_1} - X^{s, x_2}_{t_2} |^p\right]\le \mathcal{C}_p(|| b||_\infty, T)\big(|t_1 - t_2|^{p/2} + |x_1 - x_2|^{p} \big)
		\end{equation*}
		for every $t_1, t_2, s \in [0,T]$, $x_1, x_2 \in \mathbb{R}$ and for some continuous function $\mathcal{ C}_p$ increasing in each component. In particular, for every $s\in [0,T]$ almost every trajectories\footnote{We use the convention $X^{s,x}_t = x$ whenever $s\le t$.} of $(t,x)\mapsto X^{s,x}_t$ is $\alpha$-H\"older continuous with $\alpha <1/2$ in $s$ and $\alpha <1$ in $x$, locally.
	\end{cor}
	\begin{proof}
		See Subsection \ref{proofteosobdif}.
	\end{proof}

	\begin{prop}\label{theosup1}
		Let $X^{x}$ be the strong solution to the SDE \eqref{eqconpb1}. Then there exists a nondecreasing function $C$ such that the following inequality holds:
		\begin{align}\label{eqsup1}
			\mathbb{E}\big[\sup_{t\in [0,T]}|X^x_t-X^y_t|^p\big]\leq C(\|b\|_\infty,p)|x-y|^p.
		\end{align}
	\end{prop}
	\begin{proof}
		The proof of the above Proposition is postpone to  Subsection \ref{proofteosobdif} and it relies on Theorem \ref{LemGRR}.
	\end{proof}
	We also have the following result regarding the bound of the supremum of the first variation. 
	\begin{cor}\label{corsupclasder}
		Let $X^{x}$  be the strong solution to the SDE \eqref{eqconpb1} and let $p\geq 1$. Then there exists a nondecreasing function $C$ such that the following inequality holds:
		\begin{align}\label{eqsup2}
			\mathbb{E}\big[\sup_{t\in [0,T]}|\partial_xX^x_t|^p\big]\leq C(\|b\|_\infty,p).
		\end{align}
	\end{cor}
	
	In the next section, we discuss results related to the space-time integration with respect to local time of random processes of bounded variation, which play a key role in the explicit representation of both the Malliavin derivative and the first variation.

	\section{Integration with respect to local time of random processes}\label{seclocaltime}
	Let $X^x$ be the (strong) solution to the SDE \eqref{eq:SDErvmain1}.  In this section, we aim to define integration over time and space with respect to the local time $L^{X^x}(t,y)$ of $X^x$ when the integrand is a random process. Such a definition was introduced in \cite{Ein2000} for the case where $X^x$ is a Brownian motion. To extend this framework to the local time of the solution to \eqref{eq:SDErvmain1}, we will employ a change of measure under which $X^x$ is a Brownian motion. For reference, we recall the following definition from \cite[Definition 5.1]{Ein2000}.

	\begin{defi}\label{defltran}
		Let $A=\{A(t,y);\,y\in\R,\,t\geq0\}$ be a double indexed random process. For $a<b$, consider $(y_i)_{0\leq i\leq n}$ a subdivision of $[a,b]$ and $(s_j)_{0\leq j\leq m}$ a subdivision of $[0,t]$. Note $\Delta$ the grid $\{(y_i,s_j),\,0\leq i\leq n,\,0\leq j\leq m\}$. When the expression
		\begin{align*}
			\sum\limits_{0\leq i\leq n,0\leq j\leq m}A(s_j,y_i)\Big(L^{X^x}(s_{j+1},y_{i+1})-L^{X^x}(s_{j},y_{i+1})-L^{X^x}(s_{j+1},y_{i})+L^{X^x}(s_{j},y_{i})\Big)
		\end{align*}
		has a limit in probability as $\Delta\to0$, we denote this limit $\int_0^t\int_a^bA(s,y)L^{X^x}(\mathrm{d}s,\mathrm{d}y)$.
	\end{defi}
	
	\begin{prop}\label{prop:EisenbaumBV}
		Let $(A(t,y)=A_1(y)A_2(t);\,y\in\R,\,t\geq0)$ be a double indexed random process such that $A_1$ is a bounded variation random process and $A_2$ is a random process with continuous paths $\mathbb{P}$-a.s. Then for any triple $(t,a,b)\in[0,\infty)\times\R^2$ with $t\geq0$, $a<b$, and for $\mathbb{P}$-a.e. $\omega\in\Omega$, the limit
		\begin{align*}
			&\int_0^t\int_a^bA_1(y)A_2(s)L^{X^x}(\mathrm{d}s,\mathrm{d}y)\\=&\lim\limits_{\vert\Delta\vert\to0}\sum\limits_{\begin{subarray}{c}
					0\leq i\leq n\\0\leq j\leq m
			\end{subarray}}A_1(y_i)A_2(s_j)\Big(L^{X^x}(s_{j+1},y_{i+1})-L^{X^x}(s_{j},y_{i+1})-L^{X^x}(s_{j+1},y_{i})+L^{X^x}(s_{j},y_{i})\Big)
		\end{align*}
		exists, where $(y_i)_{0\leq i\leq n}$ is a subdivision of $[a,b]$, $(t_j)_{0\leq j\leq m}$ and $\Delta$ is the grid $\{(y_i,s_j);\,0\leq i\leq n,\,0\leq j\leq m\}$. In particular, one has
		\begin{align}\label{EqEisenBV00}
			\int_0^t\int_{\R}A_1(y)A_2(s)\,\mathrm{d}L^{X^x}(s,y)=-\int_{\R}\Big(\int_0^tA_2(s)\mathrm{d}_sL^{X^x}(s,y)\Big)\mathrm{d}A_1(y).
		\end{align}
	\end{prop}	
	\begin{proof}
		See Appendix \ref{proofeisen1}.
	\end{proof}
	
	\begin{remark}\label{remRNder}
		The Radon-Nikodym density defined by 
		$$
		\frac{\diffns \Q}{\diffns \PP}:={\mathcal E}\big(-\int_0^T u(s,\omega,X_s^x)\diffns B_s\big),
		$$
		where the function $u$ is such that $|u(t,\omega,x)|\leq CM(\omega)$ for all $t\in [0,T], \,\, x\in \mathbb{R}$ with $M(\omega)$ satisfying the Novikov condition defines a probability measure $\Q$ equivalent to $\PP$. Thanks to the Girsanov Theorem, under $\Q$, the solution $X^x$ to the SDE \eqref{eq:SDErvmain1} is a Brownian motion.
	\end{remark}
	
	We have the following
	
	\begin{lemm}\label{lemconsLTI1}
		Let $X^x$ be the solution to the SDE \eqref{eq:SDErvmain1}. Suppose that $A$ is as in Proposition \ref{prop:EisenbaumBV}. Then the limit in Definition \ref{defltran} exists. 
	\end{lemm}
	
	\begin{proof}
		Let $\Q$ be the probability measure from Remark \ref{remRNder} under which $X^x$ is a Brownian motion. Let $F=\int_0^T\int_a^bA(y,s)L^{X^x}(\diffns y,\diffns s)$ be the integration over the time and space with respect to the local time of the Brownian motion $X^x$ under $\Q$. Such process exists as a limit in probability under $\Q$ (see Definition \ref{defltran} and Proposition \ref{prop:EisenbaumBV}.). Thus the convergence in probability also holds under $\PP$ since $\PP$ and $\Q$ are equivalent.
	\end{proof}

	\begin{lemm}\label{lemeqra1}
		Suppose that $A$ is as in Proposition \ref{prop:EisenbaumBV}. Suppose in addition that $A(\cdot, t)$ admits a weak spatial derivative. Then for all $t\in [0,T]$, we have
		\begin{align}\label{eqLTf1}
			\int_0^T\int_\mathbb{R}A(y,s,\omega)L^{X^x}(\diffns y,\diffns s)=-\int_0^T\frac{\partial A}{\partial y}(X_s^x,s,\omega)\diffns s.
		\end{align}

	\end{lemm}
	\begin{proof}
		Suppose first that $ X=B$. Then the Equality \eqref{eqLTf1} is given by \cite[Theorem 5.1]{Ein2000}. The generalisation to the case of Brownian motion starting at an arbitrarily $x\in \mathbb{R}$ follows easily. Suppose now that $X^x$ is the solution to \eqref{eq:SDErvmain1}. Then under the equivalent martingale measure $\Q$ given in Remark \ref{remRNder} the equality\eqref{eqLTf1} holds.  Then using similar steps as in the proof of Lemma \ref{lemconsLTI1} we have that the integrals in \eqref{eqLTf1} are with respect to the Brownian motion $X^x$, and by \cite[Theorem 5.1]{Ein2000}, the identity is valid. The proof is completed.
	\end{proof}
	
	\section{Application to optimal fund rate}\label{appli}
	
	In this section, we apply the stochastic maximum principle developed in Section \ref{secmaxprinres} to study the problem introduced in Subsection \ref{subsec:motivating}. More precisely the problem is the following:
	
	\begin{prob}\label{probex1}
		Find $\hat \alpha \in \mathcal{A}$ such that 
		\begin{equation}\label{pracex21}
			J(\hat \alpha(\cdot)) = \inf_{\alpha \in \mathcal{A}}\mathbb{E}\big[(X^{\alpha, x}_T)^2\big]
		\end{equation} 
		subject to
		\begin{equation}\label{pracex11}
			\mathrm{d}X^{\alpha, x}_t
			= \big( \mu\,\tanh(X^{\alpha, x}_t/M)  - \alpha_t\,\sgn(X^{\alpha, x}_t)\mathbf{1}_{\{|X^{\alpha, x}_t|>\rho\}} \big)\mathrm{d}t
			+ \sigma\mathrm{d}B_t,
			\qquad X_0 = x,
		\end{equation}
	\end{prob}

	The main result of this section is the following
	\begin{prop}
		The optimal control for the Problem \ref{probex1} is given as
		\begin{equation}
			\label{eq:max}
			\hat\alpha_t =-\sgn(-\hat X_t \sgn(\hat X_t)\mathbf{1}_{\{|\hat X_t|>\rho\}})=\mathbf{1}_{\{|\hat X_t|>\rho\}}.
		\end{equation}
	\end{prop}

	\begin{proof}
		Let the state process be given by \eqref{pracex11}. We decompose the drift as follows:
		$$
		b_1(t,x) = \mu\,\tanh(x/M), \qquad 
		b_2(x) = -\sgn(x)\mathbf{1}_{\{|x|>\rho\}}, \qquad 
		b_3(t,\alpha) = \alpha.
		$$
		Since the terminal cost $g:x \mapsto g(x)=x^2$ is continuously differentiable, even, nonnegative, and monotone increasing on $(0,\infty)$, it follows that the optimal control $\hat{\alpha}$ should minimize the absolute value $|X_T^{\alpha,x}|$. 
		
		Applying It\^o–Tanaka's formula to $|X_t^x|$ yields
		$$
		\diff |X_t^{x}| = \sgn(X_t^{x}) \Big\{ \big( \mu\,\tanh(X_t^{x}/M) 
		- \alpha_t\,\sgn(X_t^{x})\mathbf{1}_{\{|X_t^{x}|>\rho\}} \big)\mathrm{d}t
		+ \sigma\,\mathrm{d}B_t \Big\} + \mathrm{d}_tL^{X^x}(t,0),
		$$
		where $L^{X^x}(\cdot,0)$ denotes the local time of $X^x$ at zero.  
		Thus to minimize $|X_t^{x}|$, we choose
		$$
		\hat{\alpha}_t = -\sgn\big(-\hat{X}_t\,\sgn(\hat{X}_t)\mathbf{1}_{\{|\hat{X}_t|>\rho\}}\big)=\mathbf{1}_{\{|\hat{X}_t|>\rho\}}.
		$$
		The Hamiltonian 
		$$
		H(t,\hat X_t,\alpha_t,\hat Y_t) = \big(\mu\,\tanh(\hat X_t/M) - \sgn(\hat X_t)\mathbf{1}_{\{|\hat X_t|>\rho\}}\alpha\big)\hat Y_t.
		$$
		is maximised by
		$$
		\hat{\alpha}_t = \sgn\big(-\hat Y_t\,\sgn(\hat{X}_t)\mathbf{1}_{\{|\hat{X}_t|>\rho\}}\big).
		$$
		From Theorem~\ref{thm:necc}, the adjoint process satisfies
		$$
		\hat Y_t = -2\,\mathbb{E}\big[\Phi^{\hat{\alpha}}_{t,T}\,X^{\hat{\alpha}}_T\mid \mathcal{F}_t\big], 
		$$
		where
		\begin{align}\label{abc12}
			\Phi^{\hat{\alpha}}_{t,T} 
			=& \exp\Big\{
			\int_t^T \frac{\mu}{M\cosh^2(\hat{X}_s/M)}\,\mathrm{d}s
			+ \int_{\mathbb{R}}\mathbf{1}_{\{|z|>\rho\}}\sgn(z)\big(L^{\hat{X}}(T,\mathrm{d}z)-L^{\hat{X}}(t,\mathrm{d}z)\big)
			\Big\} \notag \\
			=& \exp\Big\{
			\int_t^T \frac{\mu}{M\cosh^2(\hat{X}_s/M)}\,\mathrm{d}s
			+ 2\Big( 
			\bar{b}_2(\hat{X}_T) - \bar{b}_2(\hat{X}_t) \notag \\
			&\qquad 	+ \int_t^T \mathbf{1}_{\{|\hat{X}_s|>\rho\}}\sgn(\hat{X}_s)\mu\,\tanh(\hat{X}_s/M)\,\mathrm{d}s \notag \\
			&\qquad
			-\int_t^T \mathbf{1}_{\{|\hat{X}_s|>\rho\}}\,\mathrm{d}s
			+ \int_t^T \mathbf{1}_{\{|\hat{X}_s|>\rho\}}\sgn(\hat{X}_s)\sigma\,\mathrm{d}B_s
			\Big)
			\Big\}.
		\end{align}
		
		To derive this expression, define
		$$
		\bar{b}_2(x) := -\int_{-\infty}^{x}\mathbf{1}_{\{|z|>H\}}\sgn(z)\,\mathrm{d}z
		= \int_{-\infty}^{x} b_2(z)\,\mathrm{d}z.
		$$
		The function $\bar{b}_2$ is even. By the Bouleau–Yor formula (see, e.g., \cite[Theorem 77, p.~227]{cmr78}), we have
		$$
		\bar{b}_2(\hat{X}_T)
		= \bar{b}_2(\hat{X}_t)
		+ \int_t^T b_2(\hat{X}_s)\,\mathrm{d}\hat{X}_s
		- \frac{1}{2}\int_{\mathbb{R}} b_2(z)\big(L^{\hat{X}}(T,\mathrm{d}z)-L^{\hat{X}}(t,\mathrm{d}z)\big),
		$$
		which leads directly to the equality in \eqref{abc12}.
		
		Next, we show that $\sgn(\hat Y_t) = -\sgn(\hat{X}_t)$.  
		Substituting $\hat{\alpha}$ into the state equation gives
		$$
		\mathrm{d}\hat{X}_t
		= \big( \mu\,\tanh(\hat{X}_t/M) 
		- \sgn(\hat{X}_t)\mathbf{1}_{\{|\hat{X}_t|>\rho\}} \big)\mathrm{d}t
		+ \sigma\,\mathrm{d}B_t, 
		\qquad \hat{X}_0 = x.
		$$
		The drift of this SDE is an odd function. Define $\tilde{B}_t := -B_t$, which is a Brownian motion with the same law as $B$. Then the process $(-\hat{X}_t)$ satisfies
		$$
		\mathrm{d}(-\hat{X}_t)
		= \big( \mu\,\tanh(-\hat{X}_t/M)
		- \sgn(-\hat{X}_t)\mathbf{1}_{\{|\hat{X}_t|>\rho\}} \big)\mathrm{d}t
		+ \sigma\,\mathrm{d}(-B_t),
		$$
		By weak uniqueness, given the same initial distribution, it holds $(-\hat{X}, \tilde B)$ and $(\hat{X},B)$ have the same law. In particular, $-\hat{X}$ and $\hat{X}$ have the same distribution for all $s\ge \tau$, with $\tau:=\inf\{s\le t:\hat{X}_s=0\}$.
		
		Let us consider the following $\sigma$-algebra $\mathcal{G}_t=\sigma(\hat X_s, t\leq s\leq \tau \wedge T)$. Then we can then express $\hat Y_t $ as		
		\begin{align}
			\hat Y_t=&	\mathbb{E}\big[\Phi^{\hat{\alpha}}_{t,T}\hat X_T\,\big|\mathcal{F}_t \big]=\mathbb{E}\big[\Phi^{\hat{\alpha}}_{t,T}\hat X_T\,\big|\hat X_t \big]=\mathbb{E}\big[\Phi^{\hat{\alpha}}_{t,\tau\wedge T}\mathbb{E}\big[\Phi^{\hat{\alpha}}_{\tau\wedge T,T}\hat X_T\,\big|\mathcal{G}_t \big]\,\big|\hat X_t \big]\notag\\
			=&\mathbb{E}\big[\Phi^{\hat{\alpha}}_{t,\tau}\mathbb{E}\big[\Phi^{\hat{\alpha}}_{\tau,T}\hat X_T\,\big|\mathcal{G}_t \big]\mathbbm{1}_{\{\tau\le T\}}\,\big|\hat X_t \big]+\mathbb{E}\big[\Phi^{\hat{\alpha}}_{t, T}\mathbb{E}\big[\Phi^{\hat{\alpha}}_{T,T}\hat X_T\,\big|\mathcal{G}_t \big]\mathbbm{1}_{\{\tau> T\}}\,\big|\hat X_t \big]\notag\\
			=&\mathbb{E}\big[\Phi^{\hat{\alpha}}_{t,\tau}\mathbb{E}\big[\Phi^{\hat{\alpha}}_{\tau,T}\hat X_T\,\big|\mathcal{G}_\tau \big]\mathbbm{1}_{\{\tau\le T\}}\,\big|\hat X_t \big]+\mathbb{E}\big[\Phi^{\hat{\alpha}}_{t, T}\mathbb{E}\big[\hat X_T\,\big|\mathcal{G}_t \big]\mathbbm{1}_{\{\tau> T\}}\,\big|\hat X_t \big]\notag\\
			=&\mathbb{E}\big[\Phi^{\hat{\alpha}}_{t,\tau}\mathbb{E}\big[\Phi^{\hat{\alpha}}_{\tau,T}\hat X_T\,\big|\hat X_\tau \big]\mathbbm{1}_{\{\tau\le T\}}\,\big|\hat X_t \big]+\mathbb{E}\big[\Phi^{\hat{\alpha}}_{t, T}\mathbb{E}\big[\hat X_T\,\big|\mathcal{G}_t \big]\mathbbm{1}_{\{\tau> T\}}\,\big|\hat X_t \big]\notag\\
			=&\mathbb{E}\big[\Phi^{\hat{\alpha}}_{t,\tau}\mathbb{E}\big[\Phi^{\hat{\alpha}}_{\tau,T}\hat X_T\,\big|\hat X_\tau \big]\mathbbm{1}_{\{\tau\le T\}}\,\big|\hat X_t \big]+\mathbb{E}\big[\Phi^{\hat{\alpha}}_{t, T}\hat X_T\mathbbm{1}_{\{\tau> T\}}\,\big|\hat X_t \big],
		\end{align}	
		where we have used the Markov (or strong Markov) property together with 
		$\mathcal{G}_\tau \subset \mathcal{G}_t \subset \mathcal{F}_\tau$ to obtain the first term on the right-hand side in the fifth equality. 
		In the sixth equality we use that $\mathcal{G}_\tau = \sigma(\hat X_\tau)$, and in the final equality we use that $\hat X_T$ is $\mathcal{G}_t$-measurable.
		
		Thus, the proof is complete once we show that
		\begin{equation}\label{abc13}
			I_1 := \mathbb{E}\big[\Phi^{\hat{\alpha}}_{\tau,T}\,\hat X_T \,\big|\, X_\tau\big] 
			\mathbbm{1}_{\{\tau \le T\}} = 0 .
		\end{equation}
		Indeed, when $\tau > T$, we have $\sgn(\hat{X}_t) = \sgn(\hat{X}_s)$ for all $t \le s \le T$. Thus, the term inside the expectation either vanishes or has the same sign as $\hat{X}_t$. Consequently, $\sgn(\hat Y_t) = -\sgn(\hat{X}_t)$, and $\hat{\alpha}_t$ is an optimal control.
		
		To verify \eqref{abc13}, note that by weak uniqueness, $(\hat{X},B)$ and $(-\hat{X},\tilde{B})$ have the same law under $\mathbb{P}$ when they share the same initial distribution. Since $\bar{b}_2$, $\mathbf{1}_{\{|x|>\rho\}}$, $\cosh^2(x/M)$, and $\mathbf{1}_{\{|x|>\rho\}}\sgn(x)\tanh(x/M)$ are even functions, while $\sgn(x)\mathbf{1}_{\{|x|>\rho\}}$ is odd, we have		
		\begin{align*}
			I_1 
			=& \mathbbm{1}_{\{\tau\le T\}}\mathbb{E}\big[ 
			\exp\Big\{ \int_\tau^T \frac{\mu}{M\cosh^2(-\hat{X}_s/M)}\,\mathrm{d}s
			+ 2\Big( \bar{b}_2(-\hat{X}_T) - \bar{b}_2(-\hat{X}_t) \\
			&\quad + \int_\tau^T \mathbf{1}_{\{|-\hat{X}_s|>\rho\}}\sgn(-\hat{X}_s)\mu\tanh(-\hat{X}_s/M)\,\mathrm{d}s
			\\
			&\qquad - \int_\tau^T \mathbf{1}_{\{|-\hat{X}_s|>\rho\}}\,\mathrm{d}s 
			+ \int_\tau^T \mathbf{1}_{\{|-\hat{X}_s|>\rho\}}\sgn(-\hat{X}_s)\sigma\,\mathrm{d}\tilde{B}_s
			\Big)\Big\}\times (-\hat{X}_T) \,\big|\, X_\tau \big] \\
			=& -\mathbbm{1}_{\{\tau\le T\}}\mathbb{E}\big[ 
			\Phi^{\hat{\alpha}}_{\tau,T} \hat{X}_T \,\big|X_\tau \big]
			= -I_1,
		\end{align*}
		which implies $I_1=0$. The proof is complete.

	\end{proof}

	\subsection*{Numerical illustrations}
	Figure~\ref{fig:image1} illustrates the performance criteria of the controlled system for the parameter values 
	\(\mu = 0.5\), \(M = 4\), \(\rho = 2\), \(\sigma = 1.0\), \(x_0 = 0.0\), and \(T = 5\), under several choices of the control function:\\
	$
	\alpha_t \in 
	\Big\{
	\mathbf{1}_{\{|X_t|>\rho\}},
	\frac{X_t}{1+X_t^2},
	\frac{1}{1+X_t^2},
	\mathrm{sign}(X_t),
	-\mathrm{sign}(X_t),\;
	-\mathbf{1}_{\{|X_t|>\rho\}},
	\int_0^t e^{-s}\frac{1}{1+B_s^2}\,\mathrm{d}s
	\Big\}.
	$
	
	For each of these choices, the stochastic differential equation admits a unique strong solution, and the controls satisfy the standard admissibility conditions (progressive measurability, boundedness and integrability). 
	
	Among all controls, \(\alpha_t = \mathbf{1}_{\{|X_t|>\rho\}}\) yields the lowest cost, indicating its strong stabilising effect on the state process.  The integral-type control provides intermediate performance, whereas sign-based and negative controls lead to larger expected costs. 
	
	Figure~\ref{fig:image2} also shows the dependence of the value function on the exit corridor $\rho$. As $\rho$ increases, the value function grows monotonically, reflecting weaker control intensity and larger state excursions.

	In addition, Figure~\ref{fig:image2} displays the effect of varying the exit corridor \(H\) on the value function. We observe that the value function is an increasing function of \(H\), reflecting that wider corridors reduce the strength of the control, thereby allowing the state process to fluctuate more freely.  This monotonic relationship is consistent with theoretical expectations in stochastic control, where weaker control generally corresponds to higher performance costs.

	Figure \ref{fig:image1} depicts the performance criteria for function for Next we plotted some graphs for $\mu = 0.5, M = 4, H = 2, \sigma = 1.0, x_0 = 0.0, T = 5$ and different form of $\alpha_t=\mathbf{1}_{\{|X_t|>H\}},  \frac{X_t}{1+X_t^2}, \frac{1}{1+X_t^2}, \mathrm{sign}(X_t), -\mathrm{sign}(X_t), -\mathbf{1}_{\{|X_t|>H\}},  \int_0^t e^{-s}\frac{1}{1+B_s^2}\,\mathrm{d}s$. Observe that for the different form of $\alpha$ the corresponding SDE has a unique strong solution. In addition the control are admissible. As seem in the graph the value of the performance functional corresponding to $\alpha_t=\mathbf{1}_{\{|X_t|>H\}}$ is the smallest one. In Figure \ref{fig:image1}, we show the impact of the exit corridor $H$ on the value function. As seen, the value function is an increasing function of $H$.
	
	\begin{figure}[h!]
		\centering
		\begin{minipage}{0.45\textwidth}
			\centering
			\includegraphics[width=\linewidth]{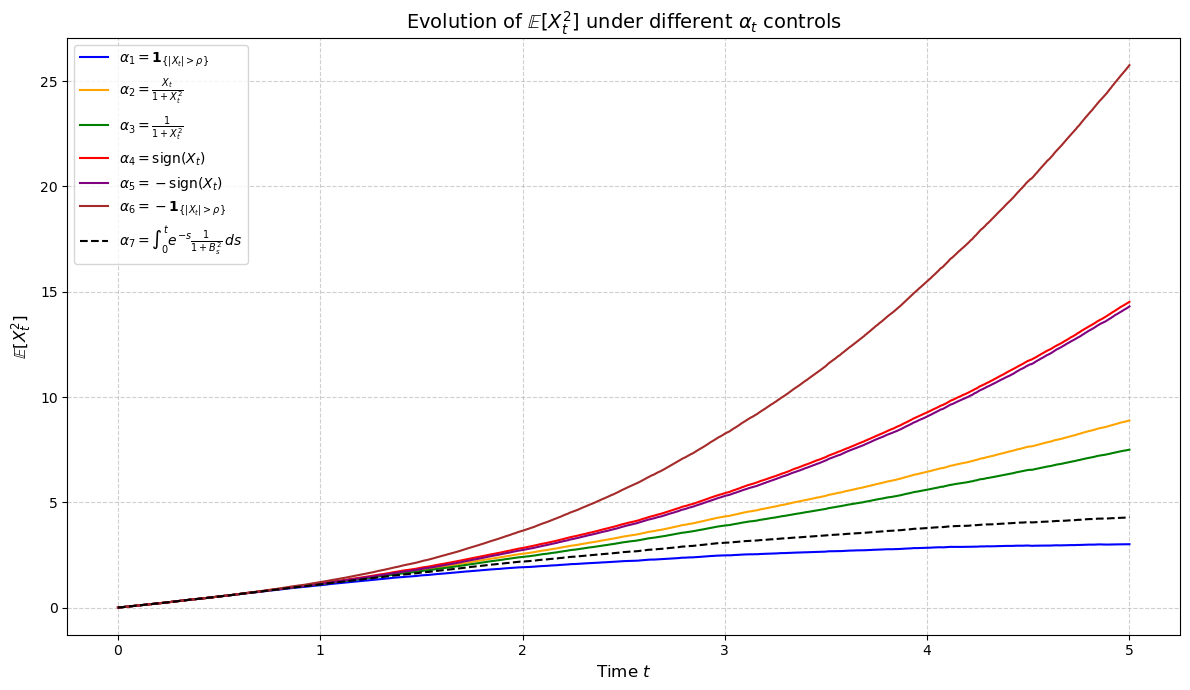}
			\caption{Terminal cost under six different control strategies}
			\label{fig:image1}
		\end{minipage}
		\hfill
		\begin{minipage}{0.45\textwidth}
			\centering
			\includegraphics[width=\linewidth]{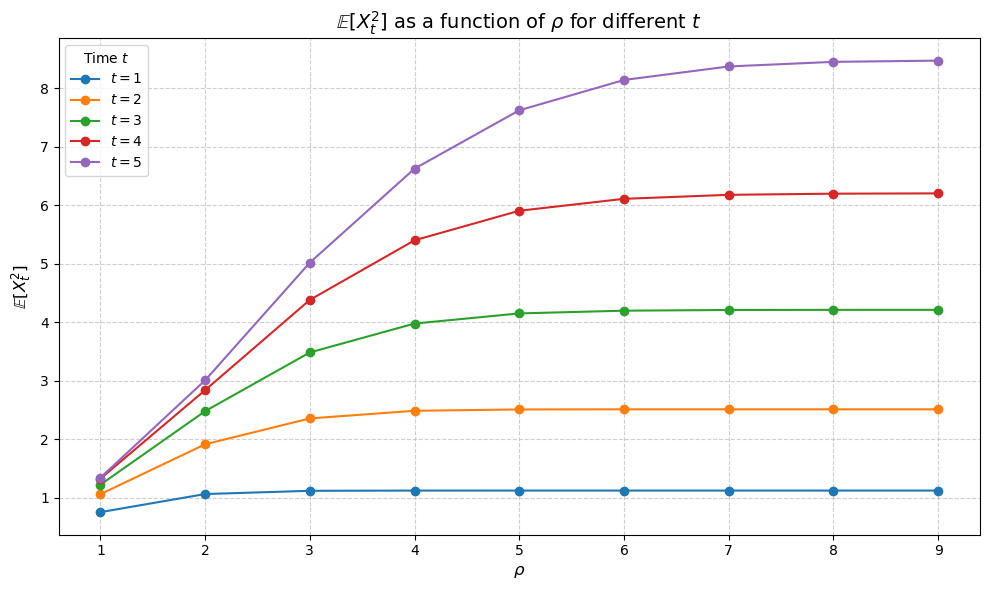}
			\caption{Value function as function of $\rho$}
			\label{fig:image2}
		\end{minipage}
		\label{fig:two_images_minipage}
	\end{figure}

	\section{Proofs of the main results}\label{proffmainr}
	
	In this section, we prove the main results. 
	\subsection{Proof of results in Section \ref{SectRanSDE}}\label{sectprorsde}
	We first start by giving he proof of the results in Section \ref{secmaldif}. 
	
	\subsubsection{Proof of of Theorem \ref{thm:sumprod}}\label{proofteomaldif} 
	In the whole of this section, we suppose Assumption \ref{mainassum} is satisfied. 
	The proof of the Theorem \ref{thm:sumprod} mimics that of \cite[Theorem 1.1]{MenTan19} with some simplified arguments. 
	Throughout the proof, we fix $b_{1,n}: [0,T] \times \mathbb{R} \rightarrow \mathbb{R}$ and $b_{2,n}: \mathbb{R} \rightarrow \mathbb{R}$, $n\ge 1$ to be smooth coefficients with compact support
	and converging a.e. to $b_1$ and $b_2$, respectively (see for example Subsection \ref{secmaldif}). We denote by ${\mathcal E}(\int q\diffns B)$ the Dol\'eans-Dade exponential
	\begin{equation*}
		{\mathcal E}\big(\int q\diffns B\big)_t := \exp\big(\int_0^t q_u\diffns B_u - \frac 12 \int_0^t |q_u|^2\diffns u \big).
	\end{equation*}

	The following result will be useful
	\begin{lemm}\label{bengen}
			The process $Z:=\mathcal{E}\left( \int b(r,\omega,\sigma\cdot B_r)\diffns B_r\right)$ is a martingale.
		\end{lemm}
		\begin{proof}
			This follows from the Girsanov theorem since by (AX1) in Assumption \ref{mainassum}, Novikov condition is satisfied (see also Remark \ref{remRNder}).
		\end{proof}
		
		\begin{proof}[Proof of Proposition \ref{lemmainres1r}]

			Equation \eqref{malldifeqr} admits the explicit solution
			\begin{align} \label{MalliavinDerivativeEquationr}
				D^i_tX^{n}_s &= e^{\int_t^s\big\{\partial_xb_{1,n}(u,X^n_u)+ b_{2,n}^\prime(X^{x,n}_u) b_3(u,\omega)\big\} \diffns u}\notag\\
				&\times\Big(\int_t^sb_{2,n}(X^n_u) D_tb_3(u,\omega)e^{-\int_t^u\big\{\partial_xb_{1,n}(r,X^n_r)+ b_{2,n}'(X^{x,n}_r)\cdot b_3(\omega,r) \big\}\diffns r}\diffns u + \sigma_i \Big).
			\end{align}
			Therefore, for every $ 0 \leq t' \leq t \leq s\leq T$, $D_{t'} X^{x,n}_s - D_t X^{x,n}_s$ can be written as.
			\begin{align}\label{eqmalderpro111r}
				&D^i_{t'} X^{x,n}_s - D^i_t X^{x,n}_s \notag\\
				=& \sigma_ie^{\int_{t}^s\big\{\partial_xb_{1,n}(u,X^n_u)+ b_{2,n}^\prime(X^{x,n}_u)\cdot b_3(u,\omega)\big\} \diffns u}\big(e^{\int_{t'}^t\big\{\partial_xb_{1,n}(u,X^n_u)+b_{2,n}^\prime(X^{x,n}_u)\cdot b_3(u,\omega) \big\}\diffns u}-1\big)\notag\\
				&+\int_{t'}^tb_{2,n}(X^n_u)\cdot D_tb_3(u,\omega)e^{-\int_{s}^u\big\{\partial_xb_{1,n}(r,X^n_r)+b_{2,n}^\prime(X^{x,n}_r) b_3(\omega,r)  \big\}\diffns r}\diffns u\notag\\
				&  +\int_{t}^sb_{2,n}(X^n_u)\big(D_{t'}b_3(u,\omega)- D^i_tb_3(u,\omega)\big)e^{-\int_{s}^u\big\{\partial_xb_{1,n}(r,X^n_r)+ b_{2,n}^\prime(X^{x,n}_r)\cdot b_3(\omega,r) \big\}\diffns r}\diffns u  \notag\\
				&=  I_1+I_2+I_3.
			\end{align}
			In what follows, $C$ is a constant that may depend on $\|b\|_{\infty}$ and $M(\omega)$.  	Using H\"older inequality, and the fact that ${b}_{1,n}$ is bounded and $b_{2,n}$ is the difference of two increasing functions, we have as in the proof of Lemma \ref{lem:auxlemma}
			\begin{align}\label{eq:bound I1}
				\mathbb{E}[I_1^2]=&
				\sigma_i^2 \mathbb{E}\Big[ e^{2\int_{t}^s \big\{\partial_xb_{1,n}(u,X^n_u)+b_{2,n}^\prime(X^{n}_u) b_3(u,\omega)\big\} \diffns u}\notag\\
				&\times\big(e^{\int_{t'}^t \big\{\partial_xb_{1,n}(u,X^n_u)+b_{2,n}^\prime(X^{n}_u) b_3(u,\omega)\big\} \diffns u}-1\big)^2\Big]
				\leq  C|t-t'.
			\end{align}

				
				As for $I_2$, using H\"older inequality, the boundedness of $b_{2,n}$ and once more $D^i_tb_3$ and similar reasoning as in Lemma \ref{lem:auxlemma}, we have 
				\begin{align}	\label{eq:boundI2}
					\nonumber
					\mathbb{E}[I^2_2] \leq& C\int_{t'}^t\mathbb{E}\big[ \big(|b_{2,n}(X^n_u) D^i_tb_3(u,\omega)\big)^4\big]^{1/2}\\
					&\times \mathbb{E}\big[ e^{-4\int_{s}^u\big\{\partial_xb_{1,n}(r,X_r)+b_{2,n}^\prime(X^{x,n}_r) b_3(\omega,r)  \big\}\diffns r}\big]^{1/2}\diffns u 
					\leq  C|t-t'|.
				\end{align}
				
				Once more repeated use of H\"older inequality, the boundedness of $b_1$ and $b_3$, and (AX2) in Assumption \ref{mainassum} give
				\begin{align}
					\nonumber
					\mathbb{E}[I^2_3] \leq& \int_{t}^s\big\{\mathbb{E}\big[ \big(b_{2,n}(X^n_u)\big( D^i_{t'}b_3(u,\omega)-D^i_tb_3(u,\omega)\big)\big)^4\big]^{1/2}\\
					&\times\mathbb{E}\big[e^{-4\int_{s}^u\big\{\partial_xb_{1,n}(r,X^n_r)+ b_{2,n}^\prime(X^{x,n}_r)\cdot b_3(\omega,r) \big\}\diffns r} \big]^{1/2}\big\}\diffns u\le C|t-t'|.
					\label{eq:boundI3}
				\end{align}

				Combining \eqref{eq:bound I1}, \eqref{eq:boundI2} and \eqref{eq:boundI3}, yields 
				$$
				\mathbb{E} \left[ | D_t X^{x,n}_s - D_{t'} X^{x,n}_s |^2 \right] \leq C|t -t'|.
				$$
				Thus the first part of the Lemma is shown.
				Taking $t'>s$ above, $D_{t'}X_s^{x,n}=0$, which implies 
				$$
				\sup_{0 \leq t \leq T} \mathbb{E} \left[ | D_t X^{x,n}_s |^2 \right] \leq C.
				$$
				This proves the proposition.
			\end{proof}
			\textbf{Step 1:} We start by showing that the SDE \eqref{eq:SDE sumprod} has a weak solution.
			\begin{lemm}[Weak Existence]
				\label{lem:weak existence}
				The SDE \eqref{eq:SDE sumprod} admits a weak solution $(X_t^x)_{t\geq 0}$. 
			\end{lemm}
			\begin{proof}
				Let $(\Omega, \mathcal{F}, \mathbb{Q})$ be a probability space supporting a Brownian motion $\hat{B}$ taking values in $\mathbb{R}^d$, and define 
				$
				X_t^x := x + \sigma \cdot \hat{B}_t, \quad 0 \le t \le T.
				$
				By \eqref{a1} and Lemma \ref{bengen}, the stochastic exponential 
				$$
				\mathcal{E}\big(\int_0^\cdot u(r, \omega, X_r^x)\, \mathrm{d}\hat{B}_r \big),
				$$
				where $
				u_{i} := \frac{\sigma_i}{\sigma_1^2+\cdots + \sigma_d^2}\big(b_{1} + b_{2}b_3\big)
				$	defines an equivalent probability measure $\mathbb{P}$ via
				$$
				\frac{\mathrm{d}\mathbb{Q}}{\mathrm{d}\mathbb{P}}
				:= \mathcal{E}\big(\int_0^T u(r, \omega, X_r^x)\, \mathrm{d}\hat{B}_r \big).
				$$
				Moreover, by Girsanov’s theorem, the process
				$$
				B_t := \hat{B}_t - \int_0^t u(r, \omega, X_r^x)\, \mathrm{d}r
				$$
				is a Brownian motion under $\mathbb{P}$. Consequently,
				\begin{align}
					\nonumber
					X_t^x 
					&= x + \int_0^t \sigma \cdot u(s, \omega, X_s^x)\, \mathrm{d}s + \sigma \cdot B_t, 
					\quad \mathbb{P}\text{-a.s.}, \; 0 \le t \le T,\\
					\label{eq:weak equation}
					&= x + \int_0^t b(s, \omega, X_s^x)\, \mathrm{d}s + \sigma \cdot B_t, 
					\quad \mathbb{P}\text{-a.s.}, \; 0 \le t \le T,
				\end{align}
				showing that $(X^x, B)$ is a weak solution to the SDE \eqref{eq:SDE sumprod} on $(\Omega, \mathcal{F}, \mathbb{P})$.
				
			\end{proof}
			
			In the following, we work on the stochastic basis 
			\[
			(\Omega, \mathcal{F}, \mathbb{P}, \{\mathcal{F}_t\}_{t \in [0,T]}),
			\]
			which supports the weak solution $(X^x, B)$ of \eqref{eq:SDE sumprod}, where $\{\mathcal{F}_t\}_{t \in [0,T]}$ denotes the filtration generated by $\{B_t\}_{t \in [0,T]}$, augmented by the $\mathbb{P}$-null sets.
			
			\textbf{Step 2}: In this step we show that $(X^{n}_{t} )_{n \geq 1}$ converges strongly to $\E\big[X_t|\mathcal{F}_t\big]$ in the space $L^2(\Omega,\mathcal{F},\mathbb{P})$. We start by showing that  for each $0 \leq t \leq T$ the sequence $(X^{n}_{t} )_{n \geq 1}$ converges weakly to $\E\big[X_t|\mathcal{F}_t\big]$ in the space $L^2(\Omega,\mathbb{P};\mathcal{F}_t)$.

			\begin{lemm}\label{lem:weak conv weak sol}
				Let $b = b_1 + b_2 \, b_3$ be as in Theorem~\ref{thm:sumprod} and let $(b_n)_{n\ge 1}$ given as in Subsection \ref{secmaldif}.  		
				Let $X^n_\cdot$ denote the strong solution of the SDE driven by the vector field $b_n$, and let $X_\cdot$ be a weak solution of the SDE associated with $b$, defined on the same probability space.  
				Then, for any continuous function of polynomial growth $h:\mathbb{R} \to \mathbb{R}$, the sequence $\{h(X^n_t)\}_{n \ge 1}$ is uniformly bounded in $L^2(\Omega, \mathbb{P};\mathcal{F}_t)$ and satisfies
				$$
				h(X^n_t) \rightharpoonup \E\big[h(X_t)\mid \mathcal{F}_t\big]
				\quad \text{weakly in } L^2(\Omega, \mathbb{P};\mathcal{F}_t)
				\quad \text{as } n \to \infty.
				$$
			\end{lemm}

			\begin{proof}
				
				Let 
				\begin{equation*}
					u_{i,n} := \frac{\sigma_i}{\sigma_1^2+\cdots + \sigma_d^2}\big(b_{1,n} + b_{2,n}b_3\big).
				\end{equation*}			
				We first show that the sequence $\big(h(X^{x,n}_{t})\big)_{n\ge 1}$ is bounded in 
				$L^2(\Omega,\mathbb{P};\mathcal{F}_t)$. Indeed, using the Girsanov transform, Hölder’s inequality, and the uniform boundedness of $u_{i,n}$, we obtain
				\begin{align}\label{eqweaklim1}
					\sup_{n}\E\big[|h(X^{x,n}_{t})|^2\big]
					\le&\sup_{n} \E\big[ e^{2\sum_{i=1}^d\int_0^Tu_{i,n}(r,x+\sigma\cdot B_r)\diffns B_r^i
						-2\sum_{i=1}^d\int_0^Tu_{i,n}^2(r,x+\sigma\cdot B_r)\diffns r}\big]^{\!\frac{1}{2}}\notag\\
					&\quad\times \E\big[e^{2\sum_{i=1}^d\int_0^Tu_{i,n}^2(r,x+\sigma\cdot B_r)\diffns r}\big]^{\frac{1}{4}}
					\E\big[|h(x+\sigma\cdot B_{t})|^4\big]^{\!\frac{1}{4}}\notag\\
					\le	& C\,\E\big[|h(x+\sigma\cdot B_{t})|^4\big]^{\frac{1}{4}}\notag\\
					=& C\big(\frac{1}{\sqrt{2\pi t\|\sigma \|^2}}
					\int_{\mathbb{R}}|h(x+z)|^4e^{-\frac{|z|^2}{2t\|\sigma \|^2}}\diffns z\big)^{\!\frac{1}{4}}\notag\\
					\le& C(p)\big(\frac{|x|^{4p}}{\sqrt{2\pi t\|\sigma \|^2}}
					\int_{\mathbb{R}}e^{-\frac{|z|^2}{2t\|\sigma \|^2}}\diffns z 
					+ \frac{1}{\sqrt{2\pi t\|\sigma \|^2}}
					\int_{\mathbb{R}}|z|^{4p}e^{-\frac{|z|^2}{2t\|\sigma \|^2}}\diffns z\big)^{\!\frac{1}{4}}\notag\\
					\le& C(p)\big(|x|^{4p}+\frac{(t\|\sigma\|)^{2p}}{\pi}\Gamma\big(2p+\tfrac{1}{2}\big)\big)^{\frac{1}{4}}<\infty.
				\end{align}

				Next, we show that $\big(h(X^{x,n}_{t})\big)_{n\ge 1}$ converges weakly to $\E\big[h(X^x_t)\mid\mathcal{F}_t\big]$ in $L^2(\Omega,\mathbb{P};\mathcal{F}_t)$.  
				Recall that the set
				\[
				\Big\{
				\mathcal{E}\Big(\!\int_0^t\dot\varphi_u\diffns B_u\!\Big), 
				:\; \varphi\in C_b^1([0,t];\mathbb{R}^d)
				\Big\}
				\]
				is dense in $L^2(\Omega,\mathbb{P};\mathcal{F}_t)$, where $C_b^1([0,t];\mathbb{R}^d)$ denotes the space of bounded continuously differentiable functions and $\dot\varphi$ their time derivative. 
				Hence, it suffices to show that for every such $\varphi$,
				\[
				\E\big[
				h(X^{x,n}_{t})
				\mathcal{E}\big(\!\int_0^t\dot\varphi_r\diffns B_r\big)
				\big]
				\to 
				\E\big[
				\E[h(X^x_t)\mid\mathcal{F}_t]
				\mathcal{E}\big(\int_0^t\dot\varphi_r\diffns B_r\big)
				\big].
				\]
				
				By the Cameron–Martin theorem (see, e.g., \cite{UsZa1}), for any measurable $h$,
				\begin{align}\label{eq:CM}
					\E\big[h(X^x_t)\mathcal{E}\big(\int_0^t\dot\varphi_u\diffns B_u\big)\big]
					= \int_\Omega h\big(X^x_t(\omega+ \varphi)\big)\diffns \mathbb{P}(\omega).
				\end{align}
				Let $\varphi\in C_b^1([0,T];\mathbb{R}^d)$ and define, for each $n$,
				\[
				\tilde X^{x,n}_t(\omega):= X^{x,n}_t(\omega+\varphi).
				\]
				Then $\tilde X^{x,n}$ satisfies the SDE
				\begin{equation}\label{eq:CM sde}
					\diffns\tilde X^{x,n}_t
					= \big(b_{1,n}(t,\tilde X^{x,n}_t)
					+ \tilde b_{2,n}(\tilde X^{x,n}_t)\tilde b_{3}(t,\omega)
					+ \sigma\dot\varphi_t\big)\diffns t
					+ \sigma\diffns B_t,
				\end{equation}
				where $\tilde b_3(t,\omega):= b_3(t,\omega+\varphi)$.  
				This follows by applying \eqref{eq:CM} to the defining equation of $X^{x,n}$ and using the identity $B_t(\omega+\varphi)=B_t(\omega)+\varphi_t$.
				Indeed let $\mathfrak{H}\in L^2(\Omega,\mathbb{P})$
				\begin{align*}
					&\E[\tilde X^{x,n}_t\mathfrak{H}]\notag\\
					= & \E\big[X^{x,n}_t\mathfrak{H}(\omega - \varphi)\mathcal{E}\big(\int_0^t\dot\varphi(u)\diffns B_u\big)\big]\\
					=	& \E\big[\big(x + \int_0^tb_{1,n}(u, X^{x,n}_u) + b_{2,n}( X^{x,n}_u)b_3(u,\omega)\diffns u + \sigma B_t \big)\mathfrak{H}(\omega-\varphi)\notag\\
					&\times {\mathcal E}\big(\int_0^t\dot\varphi_u\diffns B_u \big) \big] \\
					= 	&\E\big[\big(x + \int_0^t b_{1,n}(u, X^{x,n}_u(\omega+ \varphi))+ b_{2,n}(  X^{x,n}_u(\omega+\varphi))b_3(u,\omega+\varphi)\diffns u \notag\\
					& + \sigma B_t(\omega+\varphi)\big)\mathfrak{H}\big]\\
					=	& \E\big[\big(x + \int_0^t b_{1,n}(u, \tilde X^{x,n}_u(\omega))+  b_{2,n}(\tilde X^{x,n}_u(\omega))\tilde b_3(u,\omega)+ \sigma\dot\varphi\diffns u  + \sigma B_t(\omega)\big)\mathfrak{H}\big].
				\end{align*} 
				
				Similarly, since $X^x$ solves the limiting SDE under a probability $\mathbb{Q}\sim\mathbb{P}$ (see Lemma~\ref{lem:weak existence}), we define $\tilde X^x_t(\omega):=X^x_t(\omega+\varphi)$, which satisfies
				\begin{equation}
					\diffns \tilde X^{x}_t
					= \big(b_{1}(t,\tilde X^{x}_t)
					+ b_2(\tilde X^{x}_t)\tilde b_3(t,\omega)
					+ \sigma\dot\varphi_t\big)\diffns t
					+ \sigma\diffns B_t,\quad \mathbb{P}\text{-a.s.}
				\end{equation}
				Define
				\begin{equation}\label{eqtildeu1}
					\tilde u_{i,n}
					:= \frac{\sigma_i}{\sigma_1^2+\cdots+\sigma_d^2}
					\big(b_{1,n}+b_{2,n}\tilde b_3\big)
					=: b_{1,n}^{\sigma_i}+b_{2,n}^{\sigma_i}\tilde b_3,
					\quad
					\tilde u_i:= b_{1}^{\sigma_i}+b_{2}^{\sigma_i}\tilde b_3.
				\end{equation}
				By Girsanov’s theorem and properties of conditional expectation,
				\begin{align}\label{eqweaklim2}
					&\E\big[
					h(X^{x,n}_{t})\mathcal{E}\big(\int_0^t\dot\varphi_r\diffns B_r\big)
					-\E[h(X^x_t)\mid\mathcal{F}_t]\mathcal{E}\big(\int_0^t\dot\varphi_r\diffns B_r\big)
					\big]\notag\\
					&= \E\big[
					h(x+\sigma\cdot B_t)
					\big(
					\mathcal{E}\big(\int_0^t(\tilde u_n+\dot\varphi)\diffns B_r\big)
					-\mathcal{E}\big(\int_0^t(\tilde u+\dot\varphi)\diffns B_r\big)
					\big)
					\big].
				\end{align}

				Using the inequality $|e^a - e^b|\le |e^a + e^b||a - b|$, together with H\"older’s and Burkholder–Davis–Gundy inequalities, we find
				\begin{align}\label{eqweaklim3}
					&	\E\big[h(X^{x,n}_{t})\mathcal{E}\big(\int_0^t\dot\varphi_r\diffns B_r\big)-\E\big[h(X^x_t)|\mathcal{F}_t\big]\mathcal{E}\big(\int_0^t\dot\varphi_r\diffns B_r\big)\big]\notag\\
					\leq 	&C \E\big[h(x+\sigma\cdot B_{t})^2\big]^{\frac{1}{2}}E\big[\big(\mathcal{E}\big(\int_0^t\big\{\tilde u_n(r,x+\sigma \cdot B_r,\omega)+\dot\varphi_r\big\}\diffns B_r\big)\notag\\
					&+\mathcal{E}\big(\int_0^t\big\{\tilde u(r,x+\sigma \cdot B_r,\omega)+\dot\varphi_r\big\}\diffns B_r\big)\big)^4\big]^{\frac{1}{4}}\notag\\
					&\times\big\{\E\big[ \big(\int_0^t\big(\tilde u_n(r,x+\sigma \cdot B_r,\omega)- \tilde u(r,x+\sigma \cdot B_r,\omega)\big)\diffns B_r\big)^4\big]\notag\\
					&\quad +\E\big[\big(\int_0^t\|\tilde u_n(r,x+\sigma \cdot B_r,\omega)+\dot\varphi(t)\|^2	-\|\tilde u(r,x+\sigma \cdot B_r,\omega)+\dot\varphi_r\|^2\diffns r\big)^4\big] \big\}^{\frac{1}{4}}\notag\\
					=&I_1\times I_{2,n}\times (I_{3,n}+I_{4,n}).
				\end{align}
				The finiteness of $I_1$ follows from \eqref{eqweaklim1}.  
				Moreover, $I_{2,n}$ is uniformly bounded, since $\dot\varphi$, $b_{1,n}$, and $b_{2,n}$ are bounded, and $b_3$ satisfies a Novikov-type condition.  
				Finally, by the dominated convergence theorem, $I_{3,n}$ and $I_{4,n}$ both vanish as $n\to\infty$.  
				This completes the proof.
			\end{proof}
			
			\begin{lemm}[Strong $L^2$ convergence]
				\label{prop:convXn}
				Let $X^n_\cdot$ denote the strong solution of the SDE driven by the vector field $b_n$ (as given in Lemma \ref{lem:weak conv weak sol}), and let $X_\cdot$ be a weak solution of the SDE associated with $b$, defined on the same probability space.  
				Then, the sequence $\{X^n_t\}_{n \ge 1}$ satisfies
				$$
				X^n_t \rightharpoonup \E\big[X_t\mid \mathcal{F}_t\big]
				\quad \text{strongly in } L^2(\Omega,  \mathbb{P})
				\quad \text{as } n \to \infty.
				$$
			\end{lemm}
			
			\begin{proof}
				
				Observe that, by the compactness criterion, for each fixed $t$, there exists a subsequence $(X^{x,n_k}_t)_{k\geq 1}$ that converges strongly in $L^2(\Omega,\mathbb{P})$.  
				From Lemma~\ref{lem:weak conv weak sol}, setting $h(x)=x$ for $x\in\mathbb{R}$, it follows that $(X^{x,n_k}_t)_{k\geq 1}$ converges weakly to $\mathbb{E}\big[X^x_t\,\big|\,\mathcal{F}_t\big]$ in $L^2(\Omega,\mathbb{P})$.  
				Hence, by the uniqueness of the weak limit, there exists a subsequence $(n_k)$ such that $(X^{x,n_k}_t)_{k\geq 1}$ converges strongly to $\mathbb{E}\big[X^x_t\,\big|\,\mathcal{F}_t\big]$ in $L^2(\Omega,\mathbb{P})$.  
				By uniqueness of the limit, this convergence must hold for the entire sequence. 
				
				Indeed, suppose for contradiction that there exists $\varepsilon > 0$ such that a subsequence $(n_\ell)_{\ell\geq 1}$ satisfies
				$$
				\|X_t^{x,n_\ell} - \mathbb{E}[X^x_t\,|\,\mathcal{F}_t]\|_{L^2(\Omega,\mathbb{P})} \ge \varepsilon.
				$$
				By the compactness criterion, there exists a further subsequence $(n_{m})_{m\geq 1}$ of $(n_\ell)_{\ell\geq 1}$ such that 
				$$
				X_t^{x,n_{m}} \to \tilde{X}_t \quad \text{strongly in } L^2(\Omega,\mathbb{P}) \text{ as } m \to \infty.
				$$
				However, since $(X^{x,n_k}_t)_{k\geq 1}$ converges weakly to $\mathbb{E}\big[X^x_t\,|\,\mathcal{F}_t\big]$ in $L^2(\Omega,\mathbb{P})$, the uniqueness of the weak limit implies that 
				$$
				\tilde{X}_t = \mathbb{E}\big[X^x_t\,|\,\mathcal{F}_t\big].
				$$
				This contradicts the assumption that 
				$$
				\|X_t^{x,n_{m}} - \mathbb{E}[X^x_t\,|\,\mathcal{F}_t]\|_{L^2(\Omega,\mathbb{P})} \ge \varepsilon.
				$$
				Hence, the entire sequence $(X_t^{x,n})_{n\geq 1}$ converges strongly to $\mathbb{E}\big[X^x_t\,|\,\mathcal{F}_t\big]$ in $L^2(\Omega,\mathbb{P})$.

			\end{proof}

			\textbf{Step 3:} We show that $X_t$ is $\mathcal{F}_t$-measurable, for all $t\in[0,T]$ implying that \(X^x\) is a strong solution. Furthermore, we prove that this solution is pathwise unique.
			
			\begin{lemm}[Strong solution and pathwise uniqueness]
				\label{thm:adapted}
				For every $t\in[0,T]$, the weak solution $X_t$ to the SDE \eqref{eq:SDE sumprod} is $\mathcal{F}_t$-measurable, for every $t \in [0,T]$. Furthermore, the solution is pathwise unique.
			\end{lemm}
			\begin{proof}
				We first now show that $X_t$ is $\mathcal{F}_t$-measurable.  
				Let $h:\mathbb{R}\to\mathbb{R}$ be defined by $h(x)=x^2$.  
				By Lemma \ref{prop:convXn}, there exists a subsequence $(X^{n_k}_t)_{k\geq 0}$ such that 
				$$
				\lim_{k\to\infty} h(X^{n_k}_t) = h\big(\mathbb{E}[X_t|\mathcal{F}_t]\big),
				\quad \mathbb{P}\text{-a.s.}
				$$
				Moreover, by Lemma~\ref{lem:weak conv weak sol}, the sequence $(h(X^{n_k}_t))_{k}$ converges weakly in $L^2(\Omega,\mathcal{F},\mathbb{P})$ to 
				$\mathbb{E}[h(X_t)|\mathcal{F}_t]$.  
				By uniqueness of the limit in $L^2$, we obtain
				$$
				h\big(\mathbb{E}[X_t|\mathcal{F}_t]\big)
				= \mathbb{E}[h(X_t)|\mathcal{F}_t],
				\quad \mathbb{P}\text{-a.s.}
				$$
				That is,
				\begin{equation}\label{eqtrans1}
					\big(\mathbb{E}[X_t|\mathcal{F}_t]\big)^2 = \mathbb{E}[X_t^2|\mathcal{F}_t],
					\quad \mathbb{P}\text{-a.s.}
				\end{equation}
				Using repeatedly ~\eqref{eqtrans1} together with the tower property of conditional expectation
				\begin{align*}
					\mathbb{E}\big[(X_t - \mathbb{E}[X_t|\mathcal{F}_t])^2\big]
					&= \mathbb{E}[X_t^2]
					+ \mathbb{E}\big[(\mathbb{E}[X_t|\mathcal{F}_t])^2\big]
					- 2\,\mathbb{E}\big[X_t\,\mathbb{E}[X_t|\mathcal{F}_t]\big] \\
					&= \mathbb{E}[X_t^2]
					+ \mathbb{E}\big[\mathbb{E}[X_t^2|\mathcal{F}_t]\big]
					- 2\,\mathbb{E}\big[\mathbb{E}[X_t\,\mathbb{E}[X_t|\mathcal{F}_t]\,|\,\mathcal{F}_t]\big] \\
					&= 2\,\mathbb{E}[X_t^2]
					- 2\,\mathbb{E}\big[(\mathbb{E}[X_t|\mathcal{F}_t])^2\big]
					= 0.
				\end{align*}
				Therefore $X_t$ is $\mathcal{F}_t$-measurable.
				
				Let us now show that the solution is pathwise unique. Let $X_t^1$ and $ X_t^2$ be two solutions to the SDE \eqref{eq:SDE sumprod}. 	
				As shown on the course of the proof of Lemma \ref{lem:weak conv weak sol}, for every $\varphi \in C^1_b([0,t],\mathbb{R}^d)$, the process $\tilde X^i_t(\omega):= X^{i}_t(\omega +\varphi)$, $i=1,2$ satisfies the SDE
				\begin{equation}
					\label{eq:weaksol}
					\diffns \tilde X^{i}_t = (b_1(t,\tilde X^{i}_t)+b_{2}(\tilde X^{i}_t)\cdot\tilde b_3(t,\omega) + \sigma\dot\varphi_t)\diffns t + \sigma\diffns B_t,
				\end{equation}
				with $\tilde b_3(t,\omega) = b_3(t,\omega+\varphi)$.
				Since the function $\tilde b(t,\omega,x) =b_1(t,x)+ b_2(x) \tilde b_3(t,\omega)$ satisfies the Novikov condition, it follows from \cite[Proposition 5.3.6]{KaShr88}, that  $(\tilde X^1_t(\omega),B)$ and $(\tilde X^2_t(\omega),B)$ are two weak solutions to \eqref{eq:weaksol}.  Let $\mathbb{P}^i$ be the distribution of $(\tilde X^i_t(\omega),B),\, i=1,2$. Observe that by Assumption \ref{mainassum} and the boundedness of $\varphi$, we have $\mathbb{P}^i\big(\int_0^T|\tilde b(t,\tilde X^i_t,\omega)+\sigma \cdot\varphi_t|^2\diffns t<\infty\big)=1, i=1,2$ and thus  by \cite[Proposition 5.3.10]{KaShr88},  $(X^1_t(\omega+\varphi),B)$ and $(\tilde X^2_t(\omega+\varphi),B)$ have the same distribution.
				
				Hence, we have
				\begin{align*}
					\E\big[ X^1_t\mathcal{E}\big(\int_0^T\dot\varphi_u\diffns B_u\big)\big]=& \int_\Omega X_t^1(\omega +\varphi)\diffns \mathbb{P}^1(\omega) = \int_\Omega X_t^2(\omega +\varphi)\diffns \mathbb{P}^2(\omega)\\
					=& \E\big[X^2_t\mathcal{E}\big(\int_0^T\dot\varphi_u\diffns B_u\big)\big],
				\end{align*}
				which shows that $X^1_t = X^2_t$ for every $t$.
				Since the processes $X^1$ and $X^2$ have continuous paths, we conclude that they are indistinguishable.
			\end{proof}
			
			The proof of Theorem \ref{thm:sumprod} is then completed. In the next subsection we show that the solution to the SDE \eqref{eqconpb1} admits a Sobolev differentiable stochastic flow.

			\subsubsection{Proof of Theorem \ref{thm:sumprod1} and related result}\label{proofteosobdif}
			We start by showing uniform bound of the first variation of the approximating sequence of strong solution. Its analysis is similar to that of the Malliavin derivative considered above previously.

			\begin{proof}[Proof of Proposition \ref{pro:bound derivative}]
				The differentiability of the trajectories of $x \mapsto X^{s,x,n}_t$ follows from the seminal work \cite{Kun90}, from which we further obtain that $\partial_xX^{s,x,n}_u$ satisfies
				\begin{align}\label{eqexplifirstvar1}
					\partial_xX^{s,x,n}_t = &1 + \int_s^t\big\{\partial_xb_{1,n}(u, X^{s,x,n}_u) + b_{2,n}'( X^{s,x,n}_u)b_3(u,\omega)\big\}\partial_xX^{s,x,n}_u\diffns u\notag\\
					=& \exp\Big(\int_s^t\big\{\partial_xb_{1,n}(u, X^{s,x,n}_u) + b_{2,n}'( X^{s,x,n}_u)b_3(u,\omega)\big\}\diffns u\Big).
				\end{align}
				Thanks to Lemma \ref{lem:auxlemma}, there exists a positive constants $\tilde{C}(p,\sigma,\|b\|_\infty)$  such that 
				\begin{align*}
					\E\left[|\partial_x X^{s,x,n}_t|^p \right] &\le \tilde{C}(p,T,\sigma,\|b\|_\infty).
				\end{align*}
				Since the right side of the above bound does not depend on $x$ the required result follows. 
			\end{proof}

			\begin{proof}[Proof of Corollary \ref{corholdx}]
				Assume without loss of generality that $s\le t_1 <t_2$.
				Then, for every $n \in \mathbb{N}$, by applying the fundamental theorem of calculus, we have
				
				\begin{align*}
					X^{s,x_1,n}_{t_1}-X^{s,x_2,n}_{t_2}
					=&	X^{s,x_1,n}_{t_1}-X^{s,x_2,n}_{t_1}+	X^{s,x_2,n}_{t_1}-X^{s,x_2,n}_{t_2}\\
					=&\int_{y}^{x}\frac{\partial}{\partial z}X_{t_2}^{z,n}\diff z +\int_{t_1}^{t_2}b_{1, n}(u, X^{s,x_2,n}_u) + b_{2,n}(X^{s,x_2,n}_u)b_3(u, \omega)\diffns u\\
					&+ \sigma\cdot B_{t_1} - \sigma\cdot B_{t_2}\\
					=&\int_{y}^{x}\frac{\partial}{\partial z}X_{t_2}^{z,n}\diff z +\int_{t_1}^{t_2}b_{1, n}(u, X^{s_1,x_2,n}_u) + b_{2,n}(X^{s_1,x_2,n}_u)b_3(u, \omega)\diffns u\\
					& +\sigma\cdot (B_{t_1} -  B_{t_2}).
				\end{align*}
				Therefore, taking the $p$-power, the expectation and using H\"older inequality, the using the boundedness of $b_n$ and Proposition \ref{pro:bound derivative} give:
				\begin{align}\label{eq:lastintx}
					&	\E\big[\big|	X^{s,x_1,n}_{t_1}-X^{s,x_2,n}_{t_2}\big|^p\big]\notag\\
					=&|x-y|^{p-1}\int_{y}^{x}	\E\big[\big|\frac{\partial}{\partial z}X_{t_2}^{z,n}\big|^p\big]\diff z +\sigma\E\big[\big|B_{t_1} -  B_{t_2}\big|^p\big]\notag\\ &+(t_2-t_1)^{p-1}\int_{t_1}^{t_2}	\E\big[\big|b_{1, n}(u, X^{s,x_2,n}_u) + b_{2,n}(X^{s,x_2,n}_u)b_3(u, \omega)\big|^p\big]\diffns u\notag\\
					\leq &(|x-y|^{p}\sup_{z\in \R}\E\big[\big|\frac{\partial}{\partial z}X_{t_2}^{z,n}\big|^p\big]+C((t_2-t_1)^{p}+(t_1-t_2)^{p/2})\notag\\
					\leq &\mathcal{C}_p(\sigma,|| b||_\infty, T)(|x_1-x_2|^{p}+(t_2-t_1)^{p/2}).
				\end{align}
				Since $(X^{s, x_i, n}_{t_i})$ converges weakly to the unique solution $X^{s,x_i}_{t_i}$ of the SDE \eqref{eq:sde into} with drift $b$, (see Lemma \ref{lem:weak conv weak sol} and Theorem \ref{thm:adapted}) it follows by convexity and lower\\-semicontinuity of $K\mapsto \E[|K|^p]$ that, taking the limit in \eqref{eq:lastintx} yields the desired result.
			\end{proof}

			\begin{proof}[Proof of Proposition \ref{theosup1}]
				Without loss of generality, set $T=1$.	We start by applying Theorem \ref{LemGRR} for $K=[0,1], \mathcal{D}(t,s) = |t-s|^{\frac{\epsilon}{1+\epsilon}} , 0 <\epsilon< 1, \Psi(x) = x^{\frac{4(1+\epsilon)}{\epsilon}}, x \geq 0$
				$ f (t) = |X^x_t-X^y_t|, x, y \in \mathbb{R}$. Then, $\widehat\sigma(r) \geq r^{\frac{1+\epsilon}{\epsilon}}$, and we get that setting $s=0$ in \eqref{eqGRR1} and using triangle inequality
				\begin{align}\label{eqGRR2}
					|f(t)|\leq& 18 \int_0^{\frac{\mathcal{D}(t,0)}{2}}\Psi^{-1}\big(\frac{U}{\widehat\sigma^2(r)}\big)\diff r +|f(0)|
					=18 \int_0^{\frac{\mathcal{D}(t,0)}{2}}\Psi^{-1}\big(\frac{U}{\widehat\sigma^2(r)}\big)\diff r +|x-y|.
				\end{align}
				Choose $p>0$. Then 
				\begin{align*}
					|f(t)|^p\leq& C_p\Big\{\Big( \int_0^{\frac{\mathcal{D}(t,0)}{2}}\Psi^{-1}\big(\frac{U}{\widehat\sigma^2(r)}\big)\diff r\Big)^p +|x-y|^p\Big\}.
				\end{align*}	
				Taking the supremum on both sides gives
				\begin{align*}
					\sup_{0\leq t\leq 1}|f(t)|^p\leq& C_p\Big\{\Big( \int_0^{1}\Psi^{-1}\Big(\frac{U}{\widehat\sigma^2(r)}\Big)\diff r\Big)^p +|x-y|^p\Big\}\\
					\leq &  C_p\Big\{\Big( \int_0^{1}\Big(\frac{U}{\widehat\sigma^2(r)}\Big)^{\frac{\epsilon}{4(1+\epsilon)}}\diff r\Big)^p +|x-y|^p\Big\}\\
					\leq &  C_p\Big\{\Big( \int_0^{1}\big(\frac{1}{r^{\frac{2(1+\epsilon)}{\epsilon}}}\big)^{\frac{\epsilon}{4(1+\epsilon)}}\diff r\Big)^pU^{p\frac{\epsilon}{4(1+\epsilon)}} +|x-y|^p\Big\}\notag\\
					\leq & C_p\big(U^{p\frac{\epsilon}{4(1+\epsilon)}} +|x-y|^p\big).
				\end{align*}	
				Using the definition of $U$ in Theorem \ref{LemGRR}, the choice of $\Psi$, and choosing in particular $p\geq 5$ such that $p\frac{\epsilon}{4(1+\epsilon)}>1$, we have 
				\begin{align*}
					U^{p\frac{\epsilon}{4(1+\epsilon)}}\leq \int_0^1\int_0^1\Psi\Big(\Big(\frac{|f(t)-f(s)|}{\mathcal{D}(s,t)}\Big)\Big)^{p\frac{\epsilon}{4(1+\epsilon)}}\diff s \diff t\leq \int_0^1\int_0^1\Big(\frac{|f(t)-f(s)|}{\mathcal{D}(s,t)}\Big)^{p}\diff s \diff t.
				\end{align*}
				Thus
				\begin{align*}
					\mathbb{E}\Big[\sup_{0\leq t\leq 1}|f(t)|^p\Big]\leq&
					C_p\Big(\int_0^1\int_0^1\mathbb{E}\Big[\Big(\frac{|f(t)-f(s)|}{\mathcal{D}(s,t)}\Big)^{p}\Big]\diff s \diff t +|x-y|^p\Big)\\
					\leq&
					C_p\Big(\int_0^1\int_0^1\mathbb{E}\Big[\Big(\frac{|X^x_t-X^y_t-(X^x_s-X^y_s)|}{\mathcal{D}(s,t)}\Big)^{p}\Big]\diff s \diff t +|x-y|^p\Big)	.	 
				\end{align*}
				Let $(X^{x,n})_{n\geq 1}$ be the solution to the SDE \eqref{eqconpb1} with $b$ replaced by the approximating sequence $(b_n)_{n\geq 1}$. Then by Fatou's lemma we have 
				\begin{align*}
					\mathbb{E}\Big[\sup_{0\leq t\leq 1}|f(t)|^p\Big]
					\leq&
					C_p\Big(\liminf_{n\mapsto \infty}\int_0^1\int_0^1\mathbb{E}\Big[\Big(\frac{|X^{x,n}_t-X^{y,n}_t-(X^{x,n}_s-X^{y,n}_s)|}{\mathcal{D}(s,t)}\Big)^{p}\Big]\diff s \diff t \notag\\
					&+|x-y|^p\Big)		. 
				\end{align*}
				
				Let $x,y\in \R$ and assume without loss of generality that $x > y$. Then, by applying the fundamental theorem of calculus, we have
				\begin{align*}
					X^{x,n}_t-X^{y,n}_t
					=&\int_{y}^{x}\frac{\partial}{\partial z}X_t^{z,n}\diff z.
				\end{align*}
				Therefore
				\begin{align*}
					X^{x,n}_{t_2}-X^{y,n}_{t_2}-(X^{x,n}_{t_1}-X^{y,n}_{t_1})
					=&\int_{y}^{x}\Big(\frac{\partial}{\partial z}X_{t_2}^{z,n}-\frac{\partial}{\partial z}X_{t_1}^{z,n}\Big)\diff z.
				\end{align*}
				Dividing both sides by $\mathcal{D}$, taking the $p$-power, taking the expectation, and using H\"older inequality
				\begin{align*}
					&	\mathbb{E}\Big[	\Big(\frac{|X^{x,n}_{t_2}-X^{y,n}_{t_2}-(X^{x,n}_{t_1}-X^{y,n}_{t_1})|}{\mathcal{D}(t_1,t_2)}\Big)^p\Big]\\
					\leq& C_p\Big(|x-y|^{p-1}\int_{y}^{x}	\mathbb{E}\Big[\Big(\frac{\big|\frac{\partial}{\partial z}X_{t_2}^{z,n}-\frac{\partial}{\partial z}X_{t_1}^{z ,n}\big|}{\mathcal{D}(t_1,t_2)}\Big)^p\Big]\diff z\\
					\leq& C_p |x-y|^{p}\sup_{x\in \mathbb{R}}\mathbb{E}\Big[\Big(\frac{\big|\frac{\partial}{\partial x}X_{t_2}^{x,n}-\frac{\partial}{\partial x}X_{t_1}^{x ,n}\big|}{\mathcal{D}(t_1,t_2)}\Big)^p\Big].
				\end{align*}
				Then we get
				\begin{align}\label{eqfor30}
					&\mathbb{E}\big[\sup_{0\leq t\leq 1}|f(t)|^p\big]\notag\\
					\leq&
					C_{p}|x-y|^{p}\Big(\liminf_{n\mapsto \infty}\int_0^1\int_0^1\sup_{x\in \mathbb{R}}\mathbb{E}\Big[\Big(\frac{\big|\frac{\partial}{\partial x}X_{t_2}^{x,n}-\frac{\partial}{\partial x}X_{t_1}^{x ,n}\big|}{\mathcal{D}(t_1,t_2)}\Big)^p\Big]\diff t_1 \diff t_2 +1\Big).		 
				\end{align}

				Using \eqref{eqexplifirstvar1}, 
			for $t_1<t_2$, we have
			\begin{align*}
				\frac{\partial}{\partial x}X_{t_2}^{x,n}-\frac{\partial}{\partial x}X_{t_1}^{x,n} =&e^{\int_0^{t_2}\big\{\partial_x b_{1,n}(u,X^{x,n}_u)+ b_{2,n}^\prime(X^{x,n}_u) b_3(u,\omega)\big\} \diffns u}\\
				&-e^{\int_0^{t_1}\big\{\partial_x b_{1,n}(u,X^{x,n}_u)+ b_{2,n}^\prime(X^{x,n}_u) b_3(u,\omega)\big\} \diffns u},
			\end{align*}
			
			Taking the $p$-power and the expectation on both sides and using once more the inequality $|e^a-e^b|\leq |a-b||e^a+e^b|$, yields  
			\begin{align}\label{eqfor41} 
				&	\mathbb{E}\big[\big|	\frac{\partial}{\partial x}X_{t_2}^{x,n}-\frac{\partial}{\partial x}X_{t_1}^{x,n}\big|^p\big]\notag\\
				=&\mathbb{E}\big[\big|\int_{t_1}^{t_2}\big\{\partial_x b_{1,n}(u,X^{x,n}_u)+ b_{2,n}^\prime(X^{x,n}_u) b_3(u,\omega)\big\} \diffns u\big|^{2p}\big]^{1/2}\notag\\
				&\times	\mathbb{E}\big[\big|e^{\int_0^{t_2}\big\{\partial_x b_{1,n}(u,X^{x,n}_u)(u,X_u)+ b_{2,n}^\prime(X^{x,n}_u)\cdot b_3(u,\omega)\big\} \diffns u}\notag\\
				&	+e^{\int_0^{t_1}\big\{\partial_x b_{1,n}(u,X^{x,n}_u)+ b_{2,n}^\prime(X^{x,n}_u) b_3(u,\omega)\big\} \diffns u}\big|^{2p}\big]^{1/2}	\leq C(p,\|b\|_\infty)|t_2-t_1|^{p/2},
			\end{align}
			where the last equality follows by using \cite[Proposition 3.7 ]{MMNPZ13}, Proposition \ref{lem:auxlemma}  and some previous computations.	Since the right side of the inequality does not depend on $x$, it holds
			\begin{align*}
				\sup_{x\in \mathbb{R}}	\mathbb{E}\big[\big|	\frac{\partial}{\partial x}X_{t_2}^{x,n}-\frac{\partial}{\partial x}X_{t_1}^{x,n}\big|^p\big]
				\leq &C(p,\|b\|_\infty)|t_2-t_1|^{p/2}.
			\end{align*}

			Substituting this into \eqref{eqfor30}, we have 
			\begin{align*}
				\mathbb{E}\big[\sup_{0\leq t\leq 1}|f(t)|^p\big]
				\leq&
				C(p,\|b\|_\infty)|x-y|^{p}\Big(\int_0^1\int_0^1|t_2-t_1|^{p(\frac{1}{2}-\frac{\epsilon}{1+\epsilon})}\diff t_1 \diff t_2 +1\Big)	\notag\\
				\leq&	 C(p,\|b\|_\infty)|x-y|^{p}.
			\end{align*}	
			
			Thus for $p\geq 5$ with $p\frac{\epsilon}{4(1+\epsilon)}>1$, we have
			
			\begin{align}\label{eqfor42}
				\mathbb{E}\big[\sup_{0\leq t\leq 1}|X^x_t-X^y_t|^p\big]
				\leq	 C(p,\|b\|_\infty)|x-y|^{p}.
			\end{align}	
			Let $q>1$. Choose $p>1$ such that $qp\frac{\epsilon}{4(1+\epsilon)}>1$. Then
			\begin{align*}
				\mathbb{E}\big[\sup_{0\leq t\leq 1}|X^x_t-X^y_t|^q\big]\leq &	\mathbb{E}\big[\sup_{0\leq t\leq 1}|X^x_t-X^y_t|^{qp}\big]^{1/p}
				\leq	 C(q,d,\|b\|_\infty)|x-y|^{q}.
			\end{align*}	
			This concludes the proof 
		\end{proof}

		We conclude this section with the proof of the Sobolev differentiability.
		
		\begin{proof}[Proof of Theorem \ref{thm:sumprod1}]
			Let $p\ge2$.
			In other to show that $x\mapsto X^{s,x}$ is weakly differentiable, we start by showing that the sequence $(\partial_xX^{s,x,n})_n$ is bounded in $L^2(\Omega, L^p(\mathbb{R},w))$. This follows from Proposition \ref{pro:bound derivative} (see for e.g \cite{MNP2015}). Thus,\\ $(\partial_xX^{s,x,n}_t)$ admits a weakly converging subsequence $(\partial_xX^{s,x,n_k}_t)$ in $L^2(\Omega, L^p(\mathbb{R},w))$ to a limit $Y^{s,x}_t$. The remaining part of the proof follows as in \cite{MNP2015}. 
		\end{proof}
		We now turn to the proof of the maximum principle
		
		\subsection{Proof of Necessary and Sufficient Maximum Principle}
		We first start by giving the proof of the necessary maximum principle. 
		
		\subsubsection{Proof of of Theorem \ref{thm:necc}}\label{proofteonecmaxp} 
		\begin{lemm}
			\label{lem:conv.Xnn}
			We have the following bounds: 
			\begin{itemize}
				\item[(i)]For every sequence $(\alpha^n)_n$ in $\mathcal{A}$, it holds $\underset{n}{\sup} \E\big[\underset{t\in[0,T]}{\sup}|X^{n,\alpha^n}_t|^2 \big]<\infty$.
				\item[(ii)] For every $\alpha^1,\alpha^2 \in \mathcal{A}$ it holds that for all $p\geq 2$
				\begin{align*}
					&	\mathbb{E}\big[\sup_{0\leq t\leq 1}|X^{n,\alpha^1}_t-X^{\alpha^2}_t |^p\big]\\
					\leq&C(q,\|b\|_\infty)\Big\{	\E\big[\underset{0\leq t\leq 1}\sup|\alpha^1_s-\alpha^2_s|^{2p}\big]^{1/2}\\
					&	+ \Big(\int_0^1\frac{1}{\sqrt{2\pi s}}e^{\frac{|x|^2}{2s}}\int_{\mathbb{R}}\big|b_{1,n} (s,\sigma y)- b_1 (s,\sigma y)\big|^{2p}e^{-\frac{|y|^2}{4s}}\diffns y\diff s\Big)^{\frac{1}{2}}\notag\\
					&+\Big(\int_0^1\frac{1}{\sqrt{2\pi s}}e^{\frac{|x|^2}{2s}}\int_{\mathbb{R}}\big|b_{2,n} (\sigma y)-b_{2} (\sigma y)\big|^{4p}e^{-\frac{|y|^2}{4s}}\diffns y\diff s\Big)^{\frac{1}{4}}\Big\}.
				\end{align*}	
				\item[(iii)]Given $k \in \mathbb{N}$, for every sequence $(\alpha^n)_{n\geq 1}$ in $\mathcal{A}$ and $\alpha\in \mathcal{A}$ such that $\delta(\alpha^n,\alpha)\rightarrow 0$, it holds that 
				$$ \E\big[| X^{\alpha^n,k}_t - X^{\alpha,k}_t|^2 \big] \to 0.$$
			\end{itemize}
		\end{lemm}
		\begin{proof}
			The proof of $(i)$ is straightforward, thanks to Assumption \ref{mainassum20}. 	We use the Garsia-Rodemich-Rumsey theorem (see Theorem \ref{LemGRR}) to prove $(ii)$.	Without loss of generality, set $T=1$.	We start by applying Theorem \ref{LemGRR} for $K=[0,1], \mathcal{D}(t,s) = |t-s|^{\frac{\epsilon}{1+\epsilon}} , 0 <\epsilon< 1, \Psi(z) = z^{\frac{4(1+\epsilon)}{\epsilon}}, z \geq 0$
			$ f (t) = |X^{\alpha^1,n}_t-X^{\alpha^2}_t|, \alpha^1,\alpha^2 \in \mathcal{A}$ . Then, $\widehat\sigma(r) \geq r^{\frac{1+\epsilon}{\epsilon}}$, and we get that setting $s=0$ in \eqref{eqGRR1} and using triangle inequality
			\begin{align}\label{eqGRR20}
				|f(t)|\leq& 18 \int_0^{\frac{\mathcal{D}(t,0)}{2}}\Psi^{-1}\Big(\frac{U}{\widehat\sigma^2(r)}\Big)\diff r+|f(0)|= 18 \int_0^{\frac{\mathcal{D}(t,0)}{2}}\Psi^{-1}\Big(\frac{U}{\widehat\sigma^2(r)}\Big)\diff r .
			\end{align}
			Choose $p\geq 1$. Then 
			\begin{align*}
				|f(t)|^p\leq& C_p\Big( \int_0^{\frac{\mathcal{D}(t,0)}{2}}\Psi^{-1}\Big(\frac{U}{\widehat\sigma^2(r)}\Big)\diff r\Big)^p.
			\end{align*}	
			Taking the supremum on both sides gives
			\begin{align*}
				\sup_{0\leq t\leq 1}|f(t)|^p\leq& C_p\Big( \int_0^{1}\Psi^{-1}\Big(\frac{U}{\widehat\sigma^2(r)}\Big)\diff r\Big)^p
				\leq   C_p\Big( \int_0^{1}\Big(\frac{U}{\widehat\sigma^2(r)}\Big)^{\frac{\epsilon}{4(1+\epsilon)}}\diff r\Big)^p \\
				\leq &  C_p\Big( \int_0^{1}\Big(\frac{1}{r^{\frac{2(1+\epsilon)}{\epsilon}}}\Big)^{\frac{\epsilon}{4(1+\epsilon)}}\diff r\Big)^pU^{p\frac{\epsilon}{4(1+\epsilon)}}
				\leq   C_pU^{p\frac{\epsilon}{4(1+\epsilon)}} .
			\end{align*}	
			Using once more the definition of $U$ in Theorem \ref{LemGRR}, the choice of $\Psi$, and choosing in particular $p\geq 5$ such that $p\frac{\epsilon}{4(1+\epsilon)}>1$, we have 
			\begin{align*}
				U^{p\frac{\epsilon}{4(1+\epsilon)}}\leq \int_0^1\int_0^1\Psi\Big(\Big(\frac{|f(t)-f(s)|}{\mathcal{D}(s,t)}\Big)\Big)^{p\frac{\epsilon}{4(1+\epsilon)}}\diff s \diff t\leq \int_0^1\int_0^1\Big(\frac{|f(t)-f(s)|}{\mathcal{D}(s,t)}\Big)^{p}\diff s \diff t.
			\end{align*}
			Thus
			\begin{align}\label{eqsubGRR1}
				\mathbb{E}\Big[\sup_{0\leq t\leq 1}|f(t)|^p\Big]\leq&
				C_p\int_0^1\int_0^1\mathbb{E}\Big[\Big(\frac{|f(t)-f(s)|}{\mathcal{D}(s,t)}\Big)^{p}\Big]\diff s \diff t \notag \\
				\leq&
				C_p(\int_0^1\int_0^1\mathbb{E}\Big[\Big(\frac{|X^{\alpha^1,n}_t-X^{\alpha^2}_t-(X^{\alpha^1,n}_s-X^{\alpha^2}_s)|}{\mathcal{D}(s,t)}\Big)^{p}\Big]\diff s \diff t		 .
			\end{align}
			
			Adding and subtracting the same term and using the fundamental theorem of calculus, and solving first-order non-homogeneous linear ODE, we arrive at
			\begin{align*}
				&	X^{n,\alpha^1}_t-X^{\alpha^2}_t \\
				= &\int_0^t\big\{\int_0^1\partial_x b_{1,n} ( \Lambda_n(\lambda,u))\mathrm{d}\lambda+b_3(u,\alpha^1_u)\int_0^1b^\prime_{2,n}( \Lambda_n(\lambda,u)) \mathrm{d}\lambda\big\}(X^{n,\alpha^1}_u-X^{\alpha^2}_u)\mathrm{d}u\\
				&+ \int_0^t \big\{b_{1,n} (X^{\alpha^2}_u )-b_1(X^{\alpha^2}_u )\big\}\mathrm{d}u + \int_0^t b_3(u, \alpha^1_u)\big\{b_{2,n}(X^{\alpha^2}_u )-b_{2}(X^{\alpha^2}_u )\big\}\mathrm{d}u\\
				&+ \int_0^tb_{2}(X^{\alpha^2}_u )\big\{b_3(u, \alpha^1_u)-	b_3(u, \alpha^2_u )
				\big\}\diff u\\
				= &\int_0^t \exp\big(\int_u^t\big\{\int_0^1\partial_xb_{1,n} ( \Lambda_n(\lambda,r))\mathrm{d}\lambda+b_3(r,\alpha_1)\int_0^1b^\prime_{2,n}( \Lambda_n(\lambda,r)) \mathrm{d}\lambda\big\}\mathrm{d}r\big)\\
				&\times\big(\big\{b_{1,n} (X^{\alpha^2}_u )-b_1(X^{\alpha^2}_u )\big\} +  b_3(u, \alpha^1)\big\{b_{2,n}(X^{\alpha^2}_u )-b_{2}(X^{\alpha^2}_u )\big\}\\
				&+b_{2}(X^{\alpha^2}_u )\big\{b_3(u, \alpha^1_u)-	b_3(u, \alpha^2_u )
				\big\}\big)\diff u,
			\end{align*}
			where $\Lambda_n(\lambda,t)$ is the process given by $\Lambda_n(\lambda,t):= \lambda X^{n,\alpha^1}_t + (1 - \lambda)X^{\alpha^2}_t$.
			Therefore, 
			\begin{align*}
				&	X^{n,\alpha^1}_t-X^{\alpha^2}_t -(X^{n,\alpha^1}_s-X^{\alpha^2}_s )\\
				= &\int_s^t \exp\big(\int_u^t\big\{\int_0^1\partial_xb_{1,n} ( \Lambda_n(\lambda,r))\mathrm{d}\lambda+b_3(r,\alpha^1_r)\int_0^1b^\prime_{2,n}( \Lambda_n(\lambda,r)) \mathrm{d}\lambda\big\}\mathrm{d}r\big)\\
				&\times\big(\big\{b_{1,n} (X^{\alpha^2}_u )-b_1(X^{\alpha^2}_u )\big\} +  b_3(u, \alpha^1_u)\big\{b_{2,n}(X^{\alpha^2}_u )-b_{2}(X^{\alpha^2}_u )\big\}\\
				&+b_{2}(X^{\alpha^2}_u )\big\{b_3(u, \alpha^1_u)-	b_3(u, \alpha^2_u )
				\big\}\big)\diff u\\
				&	+\int_0^s\Big(\exp\big(\int_u^t\big\{\int_0^1\partial_x b_{1,n} ( \Lambda_n(\lambda,r))\mathrm{d}\lambda+b_3(r,\alpha^1_r)\int_0^1b^\prime_{2,n}( \Lambda_n(\lambda,r)) \mathrm{d}\lambda\big\}\mathrm{d}r\big)\\
				&-\exp\big(\int_u^s\big\{\int_0^1\partial_x b_{1,n} ( \Lambda_n(\lambda,r))\mathrm{d}\lambda+b_3(r,\alpha^1_r)\int_0^1b^\prime_{2,n}( \Lambda_n(\lambda,r)) \mathrm{d}\lambda\big\}\mathrm{d}r\big)\Big)\\
				&\times \big(\big\{b_{1,n} (X^{\alpha^2}_u )-b_1(X^{\alpha^2}_u )\big\} +  b_3(u, \alpha^1_u)\big\{b_{2,n}(X^{\alpha^2}_u )-b_{2}(X^{\alpha^2}_u )\big\}\\
				&+b_{2}(X^{\alpha^2}_u )\big\{b_3(u, \alpha^1_u)-	b_3(u, \alpha^2_u )
				\big\}\big)\diff u.
			\end{align*}
			Taking the absolute value, the $p$-power and  expectation on both sides and using the H\"older inequality and the inequality $|e^a-e^b|\leq |a-b||e^a+e^b|$ above, we have
			\begin{align*}
				&	\E\big[\big|X^{n,\alpha^1}_t-X^{\alpha^2}_t -(X^{n,\alpha^1}_s-X^{\alpha^2}_s )\big|^p\big]\\
				\leq  & C_p|t-s|^{p-1}\int_s^t \E\Big[\exp\big(p\int_u^t\big\{\int_0^1\partial_xb_{1,n} ( \Lambda_n(\lambda,r))\mathrm{d}\lambda\\
				&+b_3(r,\alpha^1)\int_0^1b^\prime_{2,n}( \Lambda_n(\lambda,r)) \mathrm{d}\lambda\big\}\mathrm{d}r\big)\times\big(\big|b_{1,n} (X^{\alpha^2}_u )-b_1(X^{\alpha^2}_u )\big|^p\\
				& +  \big|b_3(u, \alpha^1_u)\big|^p\big|b_{2,n}(X^{\alpha^2}_u )-b_{2}(X^{\alpha^2}_u )\big|^p+\big|b_{1}(X^{\alpha^2}_u )\big|^p\big|b_3(u, \alpha^1_u)-	b_3(u, \alpha^2_u )
				\big|^p\big)\Big]\diff u\\
				&	+s^{p-1}\int_0^s \E\Big[\big|\int_s^t\big\{\int_0^1\partial_xb_{1,n} ( \Lambda_n(\lambda,r))\mathrm{d}\lambda+b_3(r,\alpha^1_r)\int_0^1b^\prime_{2,n}( \Lambda_n(\lambda,r)) \mathrm{d}\lambda\big\}\mathrm{d}r\big|^p\\
				&\times\Big(\exp\big(\int_u^t\big\{\int_0^1\partial_xb_{1,n} ( \Lambda_n(\lambda,r))\mathrm{d}\lambda+b_3(r,\alpha^1_r)\int_0^1b^\prime_{2,n}( \Lambda_n(\lambda,r)) \mathrm{d}\lambda\big\}\mathrm{d}r\big)\\
				&+\exp\big(\int_u^s\big\{\int_0^1\partial_x\hat{b}_n( \Lambda_n(\lambda,r))\mathrm{d}\lambda+b_3(r,\alpha^1_r)\int_0^1b^\prime_{2,n}( \Lambda_n(\lambda,r)) \mathrm{d}\lambda\big\}\mathrm{d}r\big)\Big)^p\\
				&\times \big(\big|b_{1,n} (X^{\alpha^2}_u )-b_1(X^{\alpha^2}_u )\big|^p + \big|b_3(u, \alpha^1_u)\big|^p\big|b_{2,n}(X^{\alpha^2}_u )-b_{2}(X^{\alpha^2}_u )\big|^p\\
				&+\big|b_{2}(X^{\alpha^2}_u )\big|^p\big|b_3(u, \alpha^1_u)-	b_3(u, \alpha^2_u )
				\big|^p\big)\Big]\diff u
			\end{align*}
			\begin{align*}
				\leq &C_p|t-s|^{p-1}\int_s^t \E\Big[\exp\big(4p\int_u^t\big\{\int_0^1\partial_xb_{1,n} ( \Lambda_n(\lambda,r))\mathrm{d}\lambda\\
				&+b_3(r,\alpha^1_r)\int_0^1b^\prime_{2,n}( \Lambda_n(\lambda,r)) \mathrm{d}\lambda\big\}\mathrm{d}r\big)\Big]^{1/4}\times\Big(\E\big[\big|b_{1,n} (X^{\alpha^2}_u )-b_1(X^{\alpha^2}_u )\big|^{2p}\big]^{1/2} \\
				&+  \E\big[\big|b_3(u, \alpha^1_u)\big|^{2p}\big]^{1/4}\E\big[\big|b_{2,n}(X^{\alpha^2}_u )-b_{2}(X^{\alpha^2}_u )\big|^{2p}\big]^{1/2}\notag\\
				&+\E\big[\big|b_3(u, \alpha^1_u)-	b_3(u, \alpha^2_u )
				\big|^{2p}\big]^{1/2}\Big)\Big]\diff u+\int_0^s \E\Big[\big|\int_s^t\big\{\int_0^1\partial_xb_{1,n} ( \Lambda_n(\lambda,r))\mathrm{d}\lambda\\
				&	+b_3(r,\alpha^1_r)\int_0^1b^\prime_{2,n}( \Lambda_n(\lambda,r)) \mathrm{d}\lambda\big\}\mathrm{d}r\big|^{4p}\Big]^{1/4}\times C_ps^{p-1}\\
				&\times\Big(\E\big[\exp\big(4\int_u^t\big\{\int_0^1\partial_x b_{1,n} ( \Lambda_n(\lambda,r))\mathrm{d}\lambda+b_3(r,\alpha^1_r)\int_0^1b^\prime_{2,n}( \Lambda_n(\lambda,r)) \mathrm{d}\lambda\big\}\mathrm{d}r\big)\big]^{1/4}\\
				&+\E\big[\exp\big(4\int_u^s\big\{\int_0^1\partial_xb_{1,n} ( \Lambda_n(\lambda,r))\mathrm{d}\lambda+b_3(r,\alpha^1_r)\int_0^1b^\prime_{2,n}( \Lambda_n(\lambda,r)) \mathrm{d}\lambda\big\}\mathrm{d}r\big)\big]^{1/4}\Big)\\
				&\times \big(\E\big[\big|b_{1,n} (X^{\alpha^2}_u )-b_1(X^{\alpha^2}_u )\big|^{2p}\big]^{1/2}+\E\big[\big|b_3(u, \alpha^1_u)-	b_3(u, \alpha^2_u )
				\big|^{2p}\big]^{1/2}\\
				& + \E\big[\big|b_3(u, \alpha^1_u)\big|^{4p}\big]^{1/4}\E\big[\big|b_{2,n}(X^{\alpha^2}_u )-b_{2}(X^{\alpha^2}_u )\big|^{4}\big]^{1/4}\big)\diff u\\
				\leq 	&C_p|t-s|^{p-1}\Big\{\big(\int_s^t\E\big[\big|b_{1,n} (X^{\alpha^2}_u )-b_1(X^{\alpha_2}_u )\big|^{2p}\big]\diffns u\big)^{1/2} +\big(\int_s^t\E\big[\big|b_{2,n}(X^{\alpha^2}_u )\notag\\
				&-b_{2}(X^{\alpha^2}_u )\big|^{2p}\big]\diff u\big)^{1/2}+\big(\int_s^t\E\big[\big|b_3(u, \alpha^1_u)-	b_3(u, \alpha^2_u )
				\big|^{2p}\big]\diff u\big)^{1/2}\Big\}\\
				&	+C_ps^{p-1}|t-s|^{p/2}\Big\{\big(\int_0^s\E\big[\big|b_{1,n} (X^{\alpha^2}_u )-b_1(X^{\alpha^2}_u )\big|^{2p}\big]\diffns u\big)^{1/2} +\big(\int_0^s\E\big[\big|b_{2,n}(X^{\alpha^2}_u )-\\
				&-b_{2}(X^{\alpha^2}_u )\big|^{2p}\big]\diff u\big)^{1/2}+\big(\int_0^s\E\big[\big|b_3(u, \alpha^1_u)-	b_3(u, \alpha^2_u )
				\big|^{4p}\big]\diff u\big)^{1/4}\Big\}\\
				\leq &C_p|t-s|^{p/2}\Big\{\big(\int_0^1\E\big[\big|b_{1,n} (X^{\alpha^2}_u )-b_1(X^{\alpha^2}_u )\big|^{2p}\big]\diffns u\big)^{1/2} +\big(\int_0^1\E\big[\big|b_{2,n}(X^{\alpha^2}_u )\notag\\
				&-b_{2}(X^{\alpha^2}_u )\big|^{4p}\big]\diff u\big)^{1/4}+\big(\int_0^1\E\big[\big|b_3(u, \alpha^1_u)-	b_3(u, \alpha^2_u )
				\big|^{2p}\big]\diff u\big)^{1/2}\Big\},
			\end{align*}
			where we have use similar steps as in the proof of Lemma \ref{lem:auxlemma} to show the boundedness of the exponential, and Assumption \ref{mainassum20}.

			Thanks to the Lipschitz continuity of $b_3$, the last term on the right side is bounded
			\begin{equation}
				\label{eq:estim.alpha12}
				\E\Big[\int_0^1\big|b_3(s, \alpha^1_s)-	b_3(s, \alpha^2_s )
				\big|^{2p}\diff s \Big]^{\frac{1}{2}} \le  C\E\big[\underset{0\leq t\leq 1}\sup|\alpha_s^1-\alpha_s^2|^{2p}\big]^{1/2}.
			\end{equation}
			Using similar arguments as in \cite{BMBPD17, MenTan19}, we can show that
			\begin{align}
				\label{eq:estim.bnb}
				&	\E\Big[\int_0^{1}|b_2(X^{\alpha^2}_s) - b_{2,n}( X^{\alpha^2}_s)|^{4p}\mathrm{d}s\Big]^{\frac{1}{4}}\notag\\ 
				\le& C\Big(\int_0^1\frac{1}{\sqrt{2\pi s}}e^{\frac{|x|^2}{2s}}\int_{\mathbb{R}}\big|b_{2,n} (\sigma y)-b_{2} (\sigma y)\big|^{4p}e^{-\frac{|y|^2}{4s}}\diffns y\diff s\Big)^{\frac{1}{4}}.
			\end{align}
			and 
			\begin{align}
				\label{eq:estim.bnb1}
				&	\E\Big[\int_0^{1}|b_{1} (s,X^{\alpha^2}_s) - b_{1,n} (s, X^{\alpha^2}_s)|^{2p}\mathrm{d}s\Big]^{\frac{1}{2}}\notag\\ 
				\le& C\Big(\int_0^1\frac{1}{\sqrt{2\pi s}}e^{\frac{|x|^2}{2s}}\int_{\mathbb{R}}\big|b_{1,n}  (s,\sigma y)-b_1(s,\sigma y)\big|^{2p}e^{-\frac{|y|^2}{4s}}\diffns y\diff s\Big)^{\frac{1}{2}}.
			\end{align}

			Substituting this into \eqref{eqsubGRR1}, we have 
			\begin{align*}
				&\mathbb{E}\big[\sup_{0\leq t\leq 1}|f(t)|^p\big]\\
				\leq&
				C(p,\|b\|_\infty)\Big(\int_0^1\int_0^1|t_2-t_1|^{p(\frac{1}{2}-\frac{\epsilon}{1+\epsilon})}\diff t_1 \diff t_2\Big)\E\big[\underset{0\leq t\leq 1}\sup|\alpha_s^1-\alpha_s^2|^{2p}\big]^{1/2}	\notag\\
				&+ \Big(\int_0^1\frac{1}{\sqrt{2\pi s}}e^{\frac{|x|^2}{2s}}\int_{\mathbb{R}}\big|b_{1,n}  (s,\sigma y)-b_1 (s,\sigma y)\big|^{2p}e^{-\frac{|y|^2}{4s}}\diffns y\diff s\Big)^{\frac{1}{2}}\\
				&+\Big(\int_0^1\frac{1}{\sqrt{2\pi s}}e^{\frac{|x|^2}{2s}}\int_{\mathbb{R}}\big|b_{2,n} (\sigma y)-b_{2} (\sigma y)\big|^{4p}e^{-\frac{|y|^2}{4s}}\diffns y\diff s\Big)^{\frac{1}{4}}\\
				\leq&C(p,\|b\|_\infty)\Big\{	\E\big[\underset{0\leq t\leq 1}\sup|\alpha_s^1-\alpha_s^2|^{2p}\big]^{1/2}	\\
				&+ \Big(\int_0^1\frac{1}{\sqrt{2\pi s}}e^{\frac{|x|^2}{2s}}\int_{\mathbb{R}}\big|b_{1,n}  (s,\sigma y)-b_1 (s,\sigma y)\big|^{2p}e^{-\frac{|y|^2}{4s}}\diffns y\diff s\Big)^{\frac{1}{2}}\\
				&+\Big(\int_0^1\frac{1}{\sqrt{2\pi s}}e^{\frac{|x|^2}{2s}}\int_{\mathbb{R}}\big|b_{2,n} (\sigma y)-b_{2} (\sigma y)\big|^{4p}e^{-\frac{|y|^2}{4s}}\diffns y\diff s\Big)^{\frac{1}{4}}\Big\}.
			\end{align*}	
			
			Thus for $p\geq 5$ with $p\frac{\epsilon}{4(1+\epsilon)}>1$, we have
			
			\begin{align*}
				&	\mathbb{E}\big[\sup_{0\leq t\leq 1}|X^{n,\alpha^1}_t-X^{\alpha^2}_t |^p\big]\\
				\leq&C(p,\|b\|_\infty)\Big\{	\E\big[\underset{0\leq t\leq 1}\sup|\alpha_s^1-\alpha_s^2|^{2p}\big]^{1/2}\notag\\
				&	+ \Big(\int_0^1\frac{1}{\sqrt{2\pi s}}e^{\frac{|x|^2}{2s}}\int_{\mathbb{R}}\big|b_{1,n}  (s,\sigma y)-b_1 (s,\sigma y)\big|^{2p}e^{-\frac{|y|^2}{4s}}\diffns y\diff s\Big)^{\frac{1}{2}}\notag\\
				&+\Big(\int_0^1\frac{1}{\sqrt{2\pi s}}e^{\frac{|x|^2}{2s}}\int_{\mathbb{R}}\big|b_{2,n} (\sigma y)-b_{2} (\sigma y)\big|^{4p}e^{-\frac{|y|^2}{4s}}\diffns y\diff s\Big)^{\frac{1}{4}}\Big\}.
			\end{align*}	
			Let $q\geq2$. Choose $p\geq 2$ such that $qp\frac{\epsilon}{4(1+\epsilon)}\geq 2$. Then
			\begin{align*}
				&	\mathbb{E}\big[\sup_{0\leq t\leq 1}|X^{n,\alpha^1}_t-X^{\alpha^2}_t |^q\big]\\
				\leq &	\mathbb{E}\big[\sup_{0\leq t\leq 1}|X^{n,\alpha^1}_t-X^{\alpha^2}_t |^{qp}\big]^{1/p}\\
				\leq&C(q,\|b\|_\infty)\Big\{	\E\big[\underset{0\leq t\leq 1}\sup|\alpha_s^1-\alpha_s^2|^{2q}\big]^{1/2}	\\
				&+ \Big(\int_0^1\frac{1}{\sqrt{2\pi s}}e^{\frac{|x|^2}{2s}}\int_{\mathbb{R}}\big|b_{1,n}  (s,\sigma y)-b_1 (s,\sigma y)\big|^{2q}e^{-\frac{|y|^2}{4s}}\diffns y\diff s\Big)^{\frac{1}{2}}\notag\\
				&+\Big(\int_0^1\frac{1}{\sqrt{2\pi s}}e^{\frac{|x|^2}{2s}}\int_{\mathbb{R}}\big|b_{2,n} (\sigma y)-b_{2} (\sigma y)\big|^{4q}e^{-\frac{|y|^2}{4s}}\diffns y\diff s\Big)^{\frac{1}{4}}\Big\}.
			\end{align*}

			Since $b_k$ is Lipschitz continuous the convergence (iii) follows by classical arguments.
		\end{proof}
		The following theorem corresponds to \cite[Lemma 2.4]{MenTan22}
		\begin{lemm}
			\label{lem:J.continuous}
			Let $\alpha\in \mathcal{A}$ and let $\alpha^n$ be a sequence of admissible controls such that $\delta(\alpha^n,\alpha)\to 0$.
			Then
			\begin{itemize}
				\item[(i)] $| J_k(\alpha^n) - J_k(\alpha) | \to 0$ as $n\to \infty$  for every $k \in \mathbb{N}$ fixed. Precisely, the function $J_k:(\mathcal{A},\delta) \to \mathbb{R}$ is continuous.
				\item[(ii)] $|J_n(\alpha) - J(\alpha)| \le C\varepsilon_n$  for some $C>0$ with $\varepsilon_n\downarrow 0$.
			\end{itemize}
		\end{lemm}
		\begin{proof}
			Follows as in \cite{MenTan22}.
		\end{proof}
		The next result gives the stability of the first variation and the adjoint processes.
		
		\begin{lemm}
			\label{lem:conv.y.phi}
			Let $\alpha\in \mathcal{A}$ and $\alpha^n$ be a sequence of admissible controls such that $\delta(\alpha^n,\alpha)\to 0$.
			Then, the processes $X^{\alpha^n}_n$ and $X^{\alpha}$ admit Sobolev differentiable flows denoted $\Phi^{\alpha^n}_n$ and $\Phi^{\alpha}$, respectively and for every $0\le s\le t\le T$ it holds
			\begin{itemize}
				\item[(i)] $\mathbb{E}\big[|\Phi^{n,\alpha^n}_{s,t} - \Phi^\alpha_{s,t} |^2 \big] \to 0$ as $n\to \infty$,
				\item[(ii)] $\mathbb{E}\big[| Y^{n,\alpha^n}_t - Y^\alpha_t|^2 \big] \to 0$ as $n\to \infty$,
			\end{itemize}
			where $Y^\alpha$ is the adjoint process defined as
			\begin{equation*}
				Y^{\alpha}_t := \mathbb{E}\Big[\Phi^{\alpha}_{t,T}\partial_xg( X^{\alpha}_T) + \int_t^T\Phi^{\alpha}_{t,s} \partial_xf(s, X^{\alpha}_s, \alpha_s)\mathrm{d}s\mid \mathcal{F}_t \Big],
			\end{equation*}	
			and $Y^{\alpha^n}_n$ is defined similarly, with $(X^{\alpha},\alpha, \Phi^\alpha)$ replaced by  $(X^{\alpha^n}_n,\alpha^n, \Phi^{\alpha^n}_n)$.
		\end{lemm}
		\begin{proof} The proof follows as in \cite{MenTan22}. 
		\end{proof}
		
		\begin{proof}[Proof of Theorem \ref{thm:necc}]
			Let $\hat\alpha$ be an optimal control and $n\ge 1$ fixed.
			Thanks to the linear growth assumption on $f,g$ the function $J_n$ is bounded from above.
			By Lemma \ref{lem:J.continuous} the function $J_n$ is continuous on $(\mathcal{A},\delta)$ and there exists $\varepsilon_n$ such that 
			\begin{equation*}
				J(\hat\alpha) - J_n(\hat\alpha)\le \varepsilon_n \text{ and } J_n(\alpha) - J(\alpha) \le \varepsilon_n\quad \text{for all } \alpha \in \mathcal{A}.
			\end{equation*}
			That is, $J_n(\hat\alpha) \le \inf_{\alpha \in \mathcal{A}}J_n(\alpha) + 2\varepsilon_n$.
			Hence, using the Ekeland's variational principle (see \cite{Ekeland79}), there exists a control $\hat\alpha^n \in \mathcal{A}$ such that $\delta(\hat\alpha, \hat\alpha^n)\le (2\varepsilon_n)^{1/2}$ and 
			\begin{equation*}
				J_n(\hat\alpha^n) \le J_n(\alpha) + (2\varepsilon_n)^{1/2}\delta(\hat\alpha^n,\alpha)\quad \text{for all}\quad \alpha \in \mathcal{A}.
			\end{equation*}
			Thus define $J^\varepsilon_n(\alpha):= J_n(\alpha) + (2\varepsilon_n)^{1/2}\delta(\hat\alpha^n,\alpha)$, so that the control process $\hat\alpha^n$ is optimal for the problem with performance functional $J^\varepsilon_n$.
			
			Denote by $\beta \in \mathcal{A}$ an arbitrary control and let $\varepsilon\in[0,1]$ be a fixed constant. Let $\eta=\eta_n$ be defined by $\eta=\eta_n := \beta - \hat\alpha^n$, then by the convexity of  $\mathbb{A}$, we have  $\hat\alpha^n + \varepsilon\eta \in \mathcal{A}$.
			Thanks to the smoothness of $b_n$, $J_n$ is G\^ateau differentiable and its derivative in the direction $\eta$ is given by
			\begin{align*}
				\frac{\diffns}{\diffns\varepsilon}J_n(\hat \alpha + \varepsilon \eta)_{|_{\varepsilon = 0}}& = \mathbb{E}\Big[\int_0^T\partial_xf(t, X^{n,\hat\alpha^n}_t, \hat\alpha^n_t)\hat V^n_t + \partial_{\alpha}f(t, X^{n,\hat\alpha^n}_t, \hat\alpha^n_t)\eta_t\mathrm{d}t\\
				&\qquad + \partial_xg(X^{n,\hat\alpha^n}_T)\hat V^n_T \Big] ,
			\end{align*}
			where $V_n$ is the stochastic process solving the linear equation
			\begin{equation*}
				\diffns V^n(t) = \partial_xb_n(t, X^{n,\alpha}_t,\alpha_t)V^n_t\mathrm{d}t + \partial_\alpha b_n(t, X^{n,\alpha}_t,\alpha_t)\eta_t\mathrm{d}t,\quad V^n(0) = 0.
			\end{equation*}
			On the other hand using triangular inequality, we have
			\begin{equation*}
				\lim_{\varepsilon\downarrow 0}\frac{1}{\varepsilon}\big(\delta(\hat\alpha^n, \alpha + \varepsilon\eta) - \delta(\hat\alpha^n, \alpha) \big) \le \mathbb{E}\big[ \sup_{t \in [0,T]}|\eta_t|^4 \big]^{1/4}.
			\end{equation*}
			Therefore, $J^\varepsilon_n$ is also G\^ateau differentiable and since $\hat\alpha^n$ is optimal for $J^\varepsilon_n$, we have
			\begin{align*}
				0\le \frac{\mathrm{d}}{\mathrm{d}\varepsilon}J^\varepsilon_n(\hat\alpha^n + \varepsilon \eta)_{|_{\varepsilon = 0}} =& \frac{\mathrm{d}}{\mathrm{d}\varepsilon}J_n(\hat\alpha^n + \varepsilon \eta)_{|_{\varepsilon = 0}} + \lim_{\varepsilon\downarrow 0} (2\varepsilon_n)^{1/2}\frac{1}{\varepsilon}\delta(\hat\alpha^n,\hat\alpha^n + \varepsilon\eta)  \\
				\leq	&  \mathbb{E}\big[\int_0^T\partial_xf\big(t, X^{n,\hat\alpha^n}_{t}, \hat\alpha^n_{t} \big)\hat V^n_{t} + \partial_{\alpha}f\big( t, X^{n,\hat\alpha^n}_{t}, \hat\alpha^n_{t} \big)\eta_{t}\mathrm{d}t\\
				&+ \partial_xg(X^{n,\hat\alpha^n}_T)\hat V^n_T \big] + \big(2\varepsilon_n)^{1/2}(E[\sup_t|\eta_{t}|^{4}] \big)^{1/4}\\
				\leq	& \mathbb{E}\big[\int_0^T\partial_\alpha H_n\big(t, X^{n,\hat\alpha}_{t}, Y^{n,\hat\alpha^n}_{t}, \hat\alpha^n_{t} \big)\eta_{t}\mathrm{d}t \big] + \tilde C_C\varepsilon_n^{1/2},
			\end{align*}
			for a constant $\tilde C_C>0$ depending on the constant $C$ (introduced in the definition of $\mathcal{A}$).
			The inequality follows since $\hat\alpha^n\in \mathcal{A}$, and $H_n$ is the Hamiltonian of the problem with drift $b_n$ given by
			\begin{equation*}
				H_n(t,x,y,a) := f(t, x,a) + b_n(t,x,a)y
			\end{equation*}
			and $(Y^{n,\hat\alpha^n}, Z^{n,\hat\alpha^n})$ the adjoint processes satisfying 
			\begin{equation*}
				\mathrm{d}Y^{n,\hat\alpha^n}_{t} = -\partial_xH_n(t, X^{n,\hat\alpha}_{t}, Y^{n,\hat\alpha^n}_{t}, \hat\alpha^n_{t})\mathrm{d}t + Z^{n,\hat\alpha^n}_{t}\mathrm{d}B_{t}.
			\end{equation*}
			Next, we show that 
			\begin{align}\label{eqconvHam}
				\int_0^T\partial_\alpha H_n\big(t,X^{n,\hat\alpha}_{t}, Y^{n,\hat\alpha^n}_{t}, \hat\alpha^n_{t}\big)\eta_t\mathrm{d}t \text{ converges to }\int_0^T\partial_\alpha H\big(t, X^{\hat\alpha}_t, Y^{\hat\alpha}_t, \hat\alpha_t \big)(\beta_t-\hat \alpha_t)\mathrm{d}t .
			\end{align}
			Adding and subtracting several times gives
			\begin{align*}
				&\int_0^T\big\{	\partial_\alpha H_n\big(t,X^{n,\hat\alpha}_{t}, Y^{n,\hat\alpha^n}_{t}, \hat\alpha^n_{t} \big)\eta_{t}-	\partial_\alpha H\big(t, X^{\hat\alpha}_t, Y^{\hat\alpha}_t, \hat\alpha_t \big)(\beta_{t}-\hat\alpha_{t})	\big\}\mathrm{d}t \\
				=&\int_0^T\big\{\big(\partial_{\alpha}f(t,X^{n,\hat\alpha}_{t},  \hat\alpha^n_{t})- \partial_{\alpha}f(t, X^{\hat\alpha}_t,  \hat\alpha_t)\big)\eta(t)+\partial_{\alpha}f(t, X^{\hat\alpha}_t,  \hat\alpha_t)(\hat\alpha^n_t-\hat\alpha_t)\\
				&+ \big(b_{2,n}(X^{n,\hat\alpha^n}_t)-b_{2,n}(X^{\hat\alpha}_t)\big)\partial_{\alpha}b_{3}\big(t,  \hat\alpha^n_t \big)Y^{n,\hat\alpha^n}_t\eta_t+\big(b_{2,n}(X^{\hat\alpha}_t)\\
				&-b_{2}(X^{\hat\alpha}_t)\big)\partial_{\alpha}b_{3}\big(t,  \hat\alpha^n_t\big)Y^{n,\hat\alpha^n}_t\eta_t+b_{2}(X^{\hat\alpha}_t)\big(\partial_{\alpha}b_{3}\big(t,  \hat\alpha^n_t \big)-\partial_{\alpha}b_{3}\big(t,  \hat\alpha_t \big)\big)Y^{n,\hat\alpha^n}_t\eta_t\\
				&+b_{2}(X^{\hat\alpha}_t)\partial_{\alpha}b_{3}\big(t,  \hat\alpha_t \big)\big(Y^{n,\hat\alpha^n}_t-Y^{\hat\alpha}_t\big)\eta_t+b_{2}(X^{\hat\alpha}_t)\partial_{\alpha}b_{3}\big(t,  \hat\alpha_t \big)Y^{\hat\alpha}_t\big(\hat{\alpha}^n_t-\hat{\alpha}_t\big)\big\}\diffns t.
			\end{align*}
			Squaring and taking expectation on both sides, and using repeatedly H\"older inequality, the mean value theorem, the assumptions on the coefficients and the processes 		
			\begin{align*}
				&\mathbb{E}\big[\big(\int_0^T\big\{	\partial_\alpha H_n\big(t,X^{n,\hat\alpha}_{t}, Y^{n,\hat\alpha^n}_{t}, \hat\alpha^n_{t} \big)\eta_{t}-	\partial_\alpha H\big(t, X^{\hat\alpha}_{t}, Y^{\hat\alpha}_{t}, \hat\alpha_{t} \big)(\beta_{t}-\hat\alpha_{t})	\big\}\mathrm{d}t\big)^2\big] \\
				\leq &C\mathbb{E}\big[\int_0^T\big(|X^{n,\hat\alpha^n}_{t}-X^{\hat\alpha}_{t}|^2+|\hat\alpha^n_{t}-\hat\alpha_{t}|^2\big)|\eta(t)|^2\diffns t\big]\\
				&+\mathbb{E}\big[\int_0^T|\partial_{\alpha}f(t, X^{\hat\alpha}_{t}, \hat\alpha(t))|^2|\hat\alpha^n_{t}-\hat\alpha_{t}|^2\diffns t\big]\\
				&+ \mathbb{E}\big[\big(\int_0^T\int_0^1\partial_xb_{2,n}( \Lambda_n(\lambda,s)) \mathrm{d}\lambda(X^{n,\hat \alpha^n}_{t} - X^{\hat \alpha}_{t})\partial_{\alpha}b_{3}\big(t,  \hat\alpha^n_{t} \big)Y^{n,\hat\alpha^n}_{t}\eta_{t}\diffns t\big)^2\big]\\
				&+ \mathbb{E}\big[\big(\int_0^T\big(b_{2,n}(X^{\hat\alpha}_{t})-b_{2}(X^{\hat\alpha}_{t})\big)\partial_{\alpha}b_{3}\big(t,  \hat\alpha^n_{t} \big)Y^{n,\hat\alpha^n}_{t}\eta_{t}\diffns t\big)^2\big]\\
				&+\mathbb{E}\big[\int_0^T|b_{2}(X^{\hat\alpha}_{t})|^2|\hat\alpha^n_{t} -  \hat\alpha_{t} |^2|Y^{n,\hat\alpha^n}_{t}|^2|\eta_{t}|^2\diffns t\big]\\
				&+\mathbb{E}\big[\int_0^T|b_{2}(X^{\hat\alpha}_{t})|^2|\partial_{\alpha}b_{3}\big(t,  \hat\alpha_{t} \big)|^2\big|Y^{n,\hat\alpha^n}_{t}-Y^{\hat\alpha}_{t}\big|^2|\eta(t)|^2\diffns t\big]\\
				&+\mathbb{E}\big[\int_0^T|b_{2}(X^{\hat\alpha}_{t})|^2|\partial_{\alpha}b_{3}\big(t,  \hat\alpha_{t} \big)|^2|Y^{\hat\alpha}(t)|^2\big|\hat{\alpha}^n_{t}-\hat{\alpha}_{t}\big|^2\diffns t\big]\\
				\leq &C\int_0^T\big(\mathbb{E}\big[|X^{n,\hat\alpha^n}_{t}-X^{\hat\alpha}_{t}|^2\big]^{\frac{1}{4}}\mathbb{E}[|\eta(t)|^4]^{\frac{1}{2}}\mathbb{E}\big[|X^{n,\hat\alpha^n}_{t}|^6+|X^{\hat\alpha}_{t}|^6\big]^{\frac{1}{4}}\\
				&+\mathbb{E}\big[|\hat\alpha^n_{t}-\hat\alpha_{t}|^4\big]^{\frac{1}{2}}\mathbb{E}\big[|\eta_{t}|^4\big]^{\frac{1}{2}}\big)\diffns t+\int_0^T\mathbb{E}\big[|\partial_{\alpha}f(t, X^{\hat\alpha}_{t}, \hat\alpha_{t})|^4\big]^{\frac{1}{2}}\mathbb{E}\big[|\hat\alpha^n_{t}-\hat\alpha_{t}|^4\big]^{\frac{1}{2}}\diffns t\\
				&+ \big\{\mathbb{E}\big[\sup_{0\leq t\leq T}|X^{n,\hat \alpha^n}_t - X^{\hat \alpha}_{t}|^2\big]^{\frac{1}{8}}\mathbb{E}\big[\sup_{0\leq t\leq T}|X^{n,\hat \alpha^n}_{t}|^{14}+ |X^{\hat \alpha}_{t}|^{14}\big]^{\frac{1}{4}}\\
				&\times\mathbb{E}\big[\sup_{0\leq t\leq T}|Y^{n,\hat\alpha^n}_{t}|^{16}\big]^{\frac{1}{8}}\big\}\mathbb{E}\big[\sup_{0\leq t\leq T}|\eta_{t}|^{16}\big]^{\frac{1}{8}}\mathbb{E}\big[\big(\int_0^T\int_0^1|\partial_xb_{2,n}( \Lambda_n(\lambda,s)) |\mathrm{d}\lambda\diffns t\big)^4\big]^{\frac{1}{2}}\\
				&+ \int_0^T\mathbb{E}\big[|b_{2,n}(X^{\hat\alpha}_{t})-b_{2}(X^{\hat\alpha}_{t})|^2\big]\diffns t\int_0^T\mathbb{E}\big[|Y^{n,\hat\alpha^n}_{t}|^4\big]^{\frac{1}{2}}\mathbb{E}\big[|\eta(t)|^4\big]^{\frac{1}{2}}\diffns t\\
				&+\int_0^T\mathbb{E}\big[|\hat\alpha^n_{t} -  \hat\alpha_{t} |^4\big]^{\frac{1}{2}}\mathbb{E}\big[|Y^{n,\hat\alpha^n}_{t}|^8\big]^{\frac{1}{4}}\mathbb{E}\big[|\eta_{t}|^8\big]^{\frac{1}{4}}\diffns t\\
				&+\int_0^T\mathbb{E}\big[\big|Y^{n,\hat\alpha^n}_{t}-Y^{\hat\alpha}_{t}\big|^4\big]^{\frac{1}{2}}\mathbb{E}\big[|\eta_{t}|^4\big]^{\frac{1}{2}}\diffns t+\int_0^T\mathbb{E}\big[|Y^{\hat\alpha}_{t}|^4\big]^{\frac{1}{2}}\mathbb{E}\big[\big|\hat{\alpha}^n_{t}-\hat{\alpha}_{t}\big|^4\big]^{\frac{1}{2}}\diffns t.
			\end{align*}
			Using the $L^p$ boundedness of each term above, Lemmas \ref{lem:conv.Xnn} and  \ref{lem:conv.y.phi}, Proposition \ref{corsupclasder},  and similar steps as in the proof of Lemma \ref{lem:auxlemma}, it follows that the right side converges to $0$. Thus as $n$ goes to $\infty$, we have:
			\begin{align*}
				0
				\leq	& \mathbb{E}\big[\int_0^T\partial_\alpha H\big(t, X^{\hat\alpha}_{t}, Y^{\hat\alpha}_{t}, \hat\alpha_{t} \big)(\beta_{t}-\hat \alpha_{t})\mathrm{d}t \big],
			\end{align*}
			By standard arguments, we can thus conclude that
			\begin{equation*}
				\partial_\alpha H(t, X^{\hat\alpha}_{t}, Y^{\hat\alpha^n}_{t}, \hat\alpha_{t})\cdot (\beta - \hat\alpha(t)) \ge 0 \quad \mathbb{P}\otimes \mathrm{d}t \mathrm{-a.s}.
			\end{equation*}
			Thus \eqref{eq:nec.cond} is showed, the proof is completed.
		\end{proof}		
		\subsubsection{Proof of of Theorem \ref{thm:suff}}\label{proofteosufmaxp} 
		
		Let us now proceed with the proof of the sufficient condition for optimality.
		To ensure that the concavity of $H$, as assumed in Theorem \ref{thm:suff}, is preserved after approximation, we explicitly assume that the function 
		$b_n$ is defined using standard mollification.
		
		Therefore, $H_n(t,x,y,a):= f(t,x,a)+ b_n(t,x,a)y$ is a mollification of $H$ and thus remains concave. Let $\hat \alpha \in \mathcal{A}$ satisfy \eqref{eq:suff.con} and $\alpha'$ an arbitrary element of $\mathcal{A}$.
		Our goal is to show that $J(\hat \alpha) \ge J(\alpha')$.
		Let $n \in \mathbb{N}$ be arbitrarily chosen.
		By definition, we have
		\begin{align*}
			&J_n(\hat \alpha) - J_n(\alpha')\\
			& = \E\Big[g(X^{n,\hat \alpha}_T) - g(X^{n,\alpha'}_T) + \int_0^Tf(u, X^{n,\hat \alpha}_u, \hat \alpha_u) - f(u, X^{n,\alpha'}_u, \alpha'_u)\diff u  \Big]	\\
			&\ge \E\Big[\partial_xg(X^{n,\hata}_T)\big\{X^{\hata}_T -X^{n,\alpha'}_T\big\} + \int_0^T\big\{ b_n(u, X^{n,\alpha'}_u, \alpha'_u) - b_n(u, X^{n,\hata}_u,\hata_u)\big\} Y^{n,\hata}_u\diff u\\
			&\quad + \int_0^T H_n(u, X^{n,\hata}_u, Y^{n,\hata}_u, \hata_u) - H_n(u, X^{n,\alpha'}_u,Y^{n,\hata}_u, \alpha'_u)\diff u  \Big],
		\end{align*}	
		where we used the definition of $H_n$ and the fact that $g$ is concave.
		Since $Y^{n,\hata}$ satisfies
		\begin{equation*}
			Y^{n,\hata}_t = \E\big[\Phi^{n,\hata}_{t,T} \partial_xg( X^{n,\hata}_T) + \int_t^T\Phi^{n,\hata}_{t,u} \partial_xf(u, X^{n,\hata}_u, \hata_u)\mathrm{d}u\mid \mathcal{F}_t \big],
		\end{equation*}
		it follows by the martingale representation and the It\^o's formula that there is a square integrable progressive process $(Y^{\hata}_n,Z^{\hata}_n)$  such that $Y_n^{\hata}$ satisfies the (linear) equation 
		\begin{equation*}
			Y^{n,\hata}_t = \partial_xg(X^{n,\hata}_T) + \int_t^T\partial_xH_n(u, X^{n,\hata}_u, Y^{n,\hata}_u, \hata_u)\diff u - \int_t^TZ^{n,\hata}_u\diff B_u.
		\end{equation*}
		Recall that since $b_n$ is smooth, so is $H_n$.
		Therefore, by the It\^o's formula once again we have
		\begin{align*}
			Y^{n,\hata}_T\big\{X^{n,\hata}_T - X^{n,\alpha'}_T\big\}
			=&  \int_0^TY^{\hata}_n(u)\big\{b_n(u, X^{n,\hata}_u,\hata_u) - b_n(u, X^{n,\alpha'}_u,\alpha'_u) \big\}\diff u\\
			& - \int_0^T\big\{X^{n,\hata}_u - X^{n,\alpha'}_u \big\}\partial_xH_n(u, X^{n,\hata}_u, Y^{n,\hata}_u, \hata_u)\diff u\\
			& + \int_0^T\big\{X^{n,\hata}_u - X^{n,\alpha'}_u \big\} Z^{n,\hata}_u\diff B_u.
		\end{align*}
		Since the stochastic integral above is a local martingale, a standard localization argument allows to take expectation on both sides to get that
		\begin{align*}
			J_n(\hata) - J_n(\alpha') &\ge \E\big[- \int_0^T\big\{X^{n,\hata}_u - X^{n,\alpha'}_u \big\}\partial_xH_n(u, X^{n,\hata}_u, Y^{n,\hata}_u, \hata_u)\diff u \\
			&\quad + \int_0^T H_n(u, X_u^{n,\hata}, Y_u^{n,\hata}, \hata_u) - H_n(u, X^{n,\alpha'}_u,Y_u^{n,\hata}, \alpha'_u)\diff u   \big]\\
			&\ge \E\big[\int_0^T \partial_\alpha H_n(u, X^{n,\hata}_u, Y^{n,\hata}_u, \hata_u)\cdot(\hata_u - \alpha'_u)\diff u  \big],
		\end{align*}
		where the latter inequality follows by concavity of $H_n$. Coming back to the expression of interest $J(\hata) - J(\alpha')$, we have
		\begin{align*}
			J(\hata) - J(\alpha') & = J(\hata) - J_n(\hata) + J_n(\hata) - J_n(\alpha') + J_n(\alpha') - J(\alpha')\\
			&\ge J(\hata) - J_n(\hata) + \E\Big[\int_0^T \partial_\alpha H_n(u, X^{n,\hata}_u, Y^{n,\hata}_u, \hata_u)\cdot(\hata_u - \alpha'_u)\diff u  \Big]\\
			&\quad  + J_n(\alpha') - J(\alpha').
		\end{align*}
		Since $ b_{1,n}$ does not depend on $\alpha$, we have that
		\\		$\partial_\alpha H_n(u, X^{n,\hata}_u, Y^{n,\hata}_u, \hata_u) = \partial_\alpha b_{2,n}(X^{n,\hata}_u)b_{3}(u, \hata_u)Y^{n,\hata}_u + \partial_\alpha f(u, X^{n,\hata}_u,\hata_u)$.
		Thus, taking the limit as $n$ goes to infinity, it follows by Lemmas \ref{lem:conv.Xnn}, \ref{lem:J.continuous} and \ref{lem:conv.y.phi} that 
		\begin{align*}
			J(\hata) - J(\alpha') \ge \E\Big[\int_0^T \partial_\alpha H(u, X^{n,\hata}_u, Y^{n,\hata}_u, \hata_u)\cdot(\hata_u - \alpha'_u)\diff u  \Big].
		\end{align*}
		Since $\hata$ satisfies \eqref{eq:suff.con}, we therefore conclude that $J(\hata) \ge J(\alpha')$.

		
		\appendix

		\section{Proof of Lemma \ref{lem:auxlemma}}\label{profauxlem}
		
		Let $n \in \mathbb{N}$ and 	set
		\begin{equation}
			\label{eqgirs1 n}
			u_{i,n}=\frac{\sigma_i}{\sigma_1^2+\cdots + \sigma_d^2 }b_n.
		\end{equation}
		Since $b_{2,n}$ is the difference of two increasing functions, it follows that $|b_{2,n}^\prime|=\bar{b}_{2,n}^\prime+\tilde{b}_{2,n}^\prime$. Then using the following inequalities $|e^a-e^b|\leq |a-b||e^a+e^b|$ and $|ab|\leq \frac{\epsilon}{2}a^2+\frac{2}{\epsilon}b^2$, the boundedness of $b_3(u,\omega)$, the Girsanov theorem, there exists a constant $C$ that can change for one line to the other such that
		\begin{align*}
			&	\mathbb{E}\big[\big\{\exp\big(
			c\int_{t'}^t\big\{ \partial_xb_{1,n}(u,X^n_u)+b^\prime_{2,n}(X^n_u)\cdot b_3(u,\omega)\big\}\diffns u\big)-1\big\}^2\big]\notag\\
			\leq&\mathbb{E}\big[\big|
			c\int_{t'}^t\big\{\partial_xb_{1,n}(u,X^n_u)+b^\prime_{2,n}(X^n_u)\cdot b_3(u,\omega)\big\}\diffns u\big|^2\times\\
			&\big\{\exp\Big(
			c\int_{t'}^t\big\{\partial_xb_{1,n}(u,X^n_u)+b^\prime_{1,n}(X^n_u)\cdot b_3(u,\omega)\big\}\diffns u\big)+1\big\}^2\big]\\
			\leq &\mathbb{E}\big[ \big|
			c\int_{t'}^t\big\{ \partial_xb_{1,n}(u,X^n_u)+b^\prime_{2,n}(X^n_u)\cdot b_3(u,\omega)\big\}\diffns u\big|^2\big]^{1/2}\\
			&\times\mathbb{E}\big[\big\{\exp\big(
			c\int_{t'}^t\big\{ \partial_xb_{1,n}(u,X^n_u)+b^\prime_{2,n}(X^n_u) b_3(u,\omega)\big\}\diffns u\big)+1\big\}^2\big]^{1/2}\\
			\leq &C\mathbb{E}\big[ \big|
			c\int_{t'}^t\big\{ \partial_xb_{1,n}(u,X^n_u)+b^\prime_{2,n}(X^n_u) b_3(u,\omega)\big\}\diffns u\big|^2\big]^{1/2}\\
			&\times\big(\mathbb{E}\big[\big\{\exp\big(
			c\int_{t'}^t\big\{\partial_xb_{1,n}(u,X^n_u)+b^\prime_{2,n}(X^n_u) b_3(u,\omega)\big\}\diffns u\big)\big\}^2\big]^{1/2}+1\big)\\
			\leq&C\mathbb{E}[|M(\omega)|^4]^{1/4}\big(\mathbb{E}\big[ \big|
			c\int_{t'}^t\partial_xb_{1,n}(u,X^{n}_u)\diffns u\big|^2\big]^{1/2}+\mathbb{E}\big[ \big|
			c\int_{t'}^t\bar{b}_{2,n}^\prime(X^{n}_u)\diffns u\big|^4\big]^{1/4}\notag\\
			&+\mathbb{E}\big[ \big|
			c\int_{t'}^t\tilde{b}_{2,n}^\prime(X^{n}_u)\diffns u\big|^4\big]^{1/4}\big)\big(\E\big[\exp 2c\int_{t'}^t\partial_xb_{1,n}(u,X^{n}_u)\diffns u\big]^{1/4}\notag\\
			&\times\mathbb{E}\big[\exp\big(
			4cM(\omega)\int_{t'}^t\big(\bar{b}_{2,n}^\prime(X^{n}_u)+\tilde{b}_{2,n}^\prime(X^{n}_u)\big)\diffns u\big)\big]^{1/4}+1\big)\\
			\leq& C(|t-t'|^{1/2})\big(\mathbb{E}\big[\exp\big\{
			8cM(\omega)\int_{t'}^t\bar{b}_{2,n}^\prime(X^{n}_u)\diffns u\big\}\big]^{1/8}\\
			&\times\mathbb{E}\big[\exp\big\{
			8cM(\omega)\int_{t'}^t\tilde{b}_{2,n}^\prime(X^{n}_u)\diffns u\big\}\big]^{1/8}+1\big)\\
		\end{align*}
		\begin{align*}
			\leq& C(|t-t'|^{1/2})\big(\mathbb{E}\big[\exp\big\{
			\frac{256}{k}c^2M^2(\omega)\big\}\exp\big\{\frac{k}{4}\big|\int_{t'}^t\bar{b}_{2,n}^\prime(X^{n}_u)\diffns u\big|^2\big\}\big]^{1/8}\\
			&\times\mathbb{E}\big[\exp\big\{
			\frac{256}{k}c^2M^2(\omega)\big\}\exp\big\{\frac{k}{4}\big|\int_{t'}^t\tilde{b}_{2,n}^\prime(X^{n}_u)\diffns u\big|^2\big\}\big]^{1/8}+1\big)\\
			\leq& C(|t-t'|^{1/2})\big(\mathbb{E}\big[\exp\big\{ \frac{256}{k}c^2M^2(\omega)\big\}\big]^{\frac{1}{8}}\mathbb{E}\big[\exp\big\{ \frac{k}{2}\big|\int_{t'}^t\bar{b}_{2,n}^\prime(X^{n}_u)\diffns u\big|^2\big\}\big]^{\frac{1}{8}}\\
			&\times\mathbb{E}\big[\exp\big\{ \frac{256}{k}c^2M^2(\omega)\big\}\big]^{\frac{1}{8}}\mathbb{E}\big[\exp\big\{ \frac{k}{2}\big|\int_{t'}^t\tilde{b}_{2,n}^\prime(X^{n}_u)\diffns u\big|^2\big\}\big]^{\frac{1}{2}}+1\big)\\
			\leq&	 C(|t-t'|^{1/2})\big(\mathbb{E}\big[\mathcal{E}\big(\int_0^Tu_n(r,x+\sigma\cdot B_r,\omega)\diffns B_r\big)\exp\big\{\frac{k}{2}\big|\int_{t'}^t\bar{b}_{2,n}^\prime(x+\sigma\cdot B_u)\diffns u\big|^2\big\}\big]^{\frac{1}{8}}\\
			&\times \mathbb{E}\big[\mathcal{E}\big(\int_0^Tu_n(r,x+\sigma\cdot B_r,\omega)\diffns B_r\big)\exp\big\{ \frac{k}{2}\big|\int_{t'}^t\tilde{b}_{2,n}^\prime(x+\sigma\cdot B_u)\diffns u\big|^2\big\}\big]^{\frac{1}{8}}+1\big)\\
			\leq&	 C(|t-t'|^{1/2})\big\{\mathbb{E}\big[\mathcal{E}\big(\int_0^Tu_n(r,x+\sigma\cdot B_r,\omega)\diffns B_r\big)^2\big]^{\frac{1}{16}}\\
			&\times \mathbb{E}\big[\exp\big\{ k\big|\int_{t'}^t\bar{b}_{2,n}^\prime(x+\sigma\cdot B_u)\diffns u\big|^2\big\}\big]^{\frac{1}{16}}\mathbb{E}\big[\mathcal{E}\big(\int_0^Tu_n(r,x+\sigma\cdot B_r,\omega)\diffns B_r\big)^2\big]^{\frac{1}{16}}\\
			&\times \mathbb{E}\big[\exp\big\{ k\big|\int_{t'}^t\tilde{b}_{2,n}^\prime(x+\sigma\cdot B_u)\diffns u\big|^2\big\}\big]^{\frac{1}{16}}+1\big\}
			\leq	C(\|b_n\|_\infty)(|t-t'|^{1/2})\\
			&\leq 	C(|t-t'|^{1/2}),
		\end{align*}
		where we have used Proposition \ref{propSha1} and \cite[Proposition 3.7 ]{MMNPZ13} (see also \cite[Proposition 2.2]{Da07}) and the fact that $ \underset{n}{\sup}\mathbb{E}\big[\mathcal{E}\big(\int_0^Tu_n(r,x+\sigma\cdot B_r,\omega)\diffns B_r\big)^2\big]\leq C$. This concludes the proof.

		\section{Proof of Theorem \ref{Thmexplimalder}}\label{proofthmexplmal}
		We will only prove \eqref{ExpliciMalder}. This relies on fact that the set
		$$
		\Big\{h\otimes \mathcal{E}\Big(\int_0^1\dot{\varphi}(u)\mathrm{d}B(u)\Big):\varphi\in C^{1}_b(\mathbb{R}),h\in C^\infty_0(U)\Big\}
		$$
		spans a dense subspace in $L^2(U\times\Omega)$. Let $X^{x,n}$ be the solution of the SDE corresponding to $b_{n}$.
		Then, using \eqref{MalliavinDerivativeEquationr} and \eqref{eqLTf1}, we have  
		\begin{align} \label{qede2c}
			D^i_tX^{x,n}_s =& e^{-\int_t^s \int_{\mathbb{R}}b_n(u,z,\omega) L^{X^{x,n}}(\diffns u,\diffns z)}\notag\\
			&\times\Big(\int_t^sb_{2,n}(X^{x,n}_u) D^i_tb_3(u,\omega)e^{-\int_t^u\int_{\mathbb{R}} b_n(r,z,\omega) L^{X^{x,n}}(\diffns r,\diffns z)}\diffns u + \sigma_i \Big).
		\end{align}
		We know from Corollary \ref{maincor} that $(X^{x,n}_{t})_{n \geq 1}$ is relatively compact in $L^2(\Omega,P)$ and $\|D_sX^{x,n}_t\|_{L^2(P\otimes \diffns t)}$ is uniformly bounded in $n$. Thus $(D_sX^{x,n}_t)_{n \geq 1}$ converges weakly to $D_sX^{x}_t$ in $L^2(P;\mathbb{R})$, thanks to \cite[Lemma 1.2.3]{Nua06}. The proof is completed if we show that for every $\varphi \in C^1_b([0,T],\mathbb{R}^d)$ it holds
		\begin{align*}
			&\Big\{e^{-\int_t^s \int_{\mathbb{R}}b_n(u,z,\omega) L^{X^{x,n}}(\diffns u,\diffns z)}\Big(\int_t^sb_{2,n}(X^{x,n}_u) D^i_tb_3(u,\omega)e^{-\int_t^u\int_{\mathbb{R}} b_n(r,z,\omega) L^{X^{x,n}}(\diffns r,\diffns z)}\diffns u\\
			& + \sigma_i \Big)\mathcal{E}\Big(\int_0^T\dot \varphi_r\diffns B_r\Big)\Big\}_{n}
		\end{align*} 
		converges to 
		\begin{align*}
			&e^{-\int_t^s \int_{\mathbb{R}}b(u,z,\omega) L^{X^x}(\diffns u,\diffns z)}\Big(\int_t^sb_{2}(X_u) D^i_tb_3(u,\omega)e^{-\int_t^u\int_{\mathbb{R}} b(r,z,\omega) L^{X^x}(\diffns r,\diffns z)}\diffns u\\
			& + \sigma_i \Big)\mathcal{E}\Big(\int_0^T\dot\varphi_r\diffns B_r\Big)
		\end{align*}  
		in expectation. 
		We only show that the convergence in expectation for the first term. 
		Using Girsanov theorem and the Cameron-Martin theorem, we have
		\begin{align}
			L=	&	\Big|\E\Big[\mathcal{E}\Big(\int_0^T\dot\varphi_r\diffns B_r\Big)
			\Big\{e^{-\int_t^s\int_{\mathbb{R}} b_n(u,z,\omega) L^{X^{x,n}}(\diffns u,\diffns z)}\notag\\
			&\times \int_t^sb_{2,n}(X^{x,n}_u)\cdot D^i_tb_3(u,\omega)e^{-\int_t^u \int_{\mathbb{R}} b_n(r,z,\omega) L^{X^{x,n}}(\diffns r,\diffns z)}\diffns u\notag\\
			&	-e^{-\int_t^s \int_{\mathbb{R}} b(u,z,\omega) L^{X^x}(\diffns u,\diffns z)}\int_t^sb_{2}(X_u)\cdot D^i_tb_3(u,\omega)e^{-\int_t^u \int_{\mathbb{R}}b(r,z,\omega) L^{X^x}(\diffns r,\diffns z)}\diffns u\Big\}
			\Big]\Big|\notag\\
			\leq & \Big|\E\Big[\Big\{\mathcal{E}\Big(\int_0^T\Big\{\tilde u_n(r,\omega,B_r )+\dot\varphi_r\Big\}\diffns B_r\Big)-\mathcal{E}\Big(\int_0^T\Big\{\tilde u(r,\omega,B_r )+\dot\varphi_r\Big\}\diffns B_r\Big)\Big\}\notag\\
			&\times e^{-\int_t^s\int_{\mathbb{R}} b_n(u,z,\omega+\dot\varphi_u) L^{\|\sigma\|B^x_\sigma}(\diffns u,\diffns z)}
			\int_t^sb_{2,n}(B^x_u)\cdot D^i_tb_3(u,\omega+\dot\varphi_u)\notag\\
			&\times e^{-\int_t^u \int_{\mathbb{R}} b_n(r,z,\omega+\dot\varphi_r) L^{\|\sigma\|B^x_\sigma}(\diffns r,\diffns z)}\diffns u\Big]\Big|\notag \\
			&+\big|\E\big[\mathcal{E}\big(\int_0^T\big\{\tilde u(r,\omega,B_r )+\dot\varphi_r\big\}\diffns B_r\big)\big(e^{-\int_t^s \int_{\mathbb{R}} b(u,z,\omega+\dot\varphi_u) L^{\|\sigma\|B^x_\sigma}(\diffns u,\diffns z)}\notag\\
			&-e^{-\int_t^s\int_{\mathbb{R}} b_n(u,z,\omega+\dot\varphi_u) L^{\|\sigma\|B^x_\sigma}(\diffns u,\diffns z)}\big)\notag\\
			&\times \int_t^sb_{2}(B^x_u)\cdot D^i_tb_3(u,\omega+\dot\varphi_u)e^{-\int_t^u \int_{\mathbb{R}}b(r,z,\omega+\dot\varphi_r) L^{\|\sigma\|B^x_\sigma}(\diffns r,\diffns z)}\diffns u\big]\big|\notag\\
			&+\big|\E\big[e^{-\int_t^s\int_{\mathbb{R}} b_n(u,z,\omega+\dot\varphi_u) L^{\|\sigma\|B^x_\sigma}(\diffns u,\diffns z)}\notag\\
			&\times 
			\int_t^sD^i_tb_3(u,\omega+\dot\varphi_u)\big(b_{2,n}(B^x_u) -b_{2}(B^x_u)\big)e^{-\int_t^u \int_{\mathbb{R}} b_n(r,z,\omega+\dot\varphi_r) L^{\|\sigma\|B^x_\sigma}(\diffns r,\diffns z)}\diffns u	\big]\big|\notag\\
			&+\big|\E\big[e^{-\int_t^s\int_{\mathbb{R}} b_n(u,z,\omega+\dot\varphi_u) L^{\|\sigma\|B^x_\sigma}(\diffns u,\diffns z)}\int_t^sb_{2}(B^x_u)\cdot D^i_tb_3(u,\omega+\dot\varphi_u)\notag\\
			&\times\big(e^{-\int_t^u \int_{\mathbb{R}}b(r,z,\omega+\dot\varphi_r) L^{\|\sigma\|B^x_\sigma}(\diffns r,\diffns z)}-e^{-\int_t^u \int_{\mathbb{R}} b_n(r,z,\omega+\dot\varphi_r) L^{\|\sigma\|B^x_\sigma}(\diffns r,\diffns z)}\big)\diffns u	\big]\big|\notag\\
			=&I_{1,n}+I_{2,n}+I_{3,n}+I_{4,n}\notag
		\end{align}
		The convergence of $I_{1,n}$ follows as in the previous computations. We only show the convergence of $I_{2,n}$. Using H\"older inequality, the theorem hypothesis, we have
		\begin{align}
			I_{2,n}
			\leq & \E\big[\mathcal{E}\big(\int_0^T\big\{\tilde u(r,\omega,B_r )+\dot\varphi_r\big\}\diffns B_r\big)^{12}\big]^{\frac{1}{12}}\notag\\
			&\times\big[\big(e^{-\int_t^s \int_{\mathbb{R}} b(u,z,\omega+\dot\varphi_u) L^{\|\sigma\|B^x_\sigma}(\diffns u,\diffns z)}-e^{-\int_t^s\int_{\mathbb{R}} b_n(u,z,\omega+\dot\varphi_u) L^{\|\sigma\|B^x_\sigma}(\diffns u,\diffns z)}\big)^4\big]^{\frac{1}{4}}\notag\\
			&\times \E\big[\big(\int_t^sb_{2}(B^x_u)\cdot D^i_tb_3(u,\omega+\dot\varphi_u)e^{-\int_t^u \int_{\mathbb{R}}b(r,z,\omega+\dot\varphi_r) L^{\|\sigma\|B^x_\sigma}(\diffns r,\diffns z)}\diffns u\big)^{\frac{3}{2}}\big]^{\frac{2}{3}}\notag
		\end{align}
		\begin{align}
			\leq	& C\E\big[\big|e^{-\int_t^s \int_{\mathbb{R}} b(u,z,\omega+\dot\varphi_u) L^{\|\sigma\|B^x_\sigma}(\diffns u,\diffns z)}-e^{-\int_t^s\int_{\mathbb{R}} b_n(u,z,\omega+\dot\varphi_u) L^{\|\sigma\|B^x_\sigma}(\diffns u,\diffns z)}\big|\big]^{\frac{1}{8}}\notag\\
			&\times \E\big[\big(e^{-\int_t^s \int_{\mathbb{R}} b(u,z,\omega+\dot\varphi_u) L^{\|\sigma\|B^x_\sigma}(\diffns u,\diffns z)}+e^{-\int_t^s\int_{\mathbb{R}} b_n(u,z,\omega+\dot\varphi_u) L^{\|\sigma\|B^x_\sigma}(\diffns u,\diffns z)}\big)^7\big]^{\frac{1}{8}}\notag\\
			&\times \E\big[\big(\int_0^T \big(D^i_tb_3(u,\omega+\dot\varphi_u)\big)^2\diffns u\big)^{2}\big]^{\frac{1}{4}}	\notag\\
			&	\E\big[\big(\int_0^Te^{-2\int_t^u \int_{\mathbb{R}}b(r,z,\omega+\dot\varphi_r) L^{\|\sigma\|B^x_\sigma}(\diffns r,\diffns z)}\diffns u\big)^{\frac{6}{5}}\big]^{\frac{5}{12}}\notag\\
			\leq	&	C \E\big[\big|-\int_t^s \int_{\mathbb{R}} b(u,z,\omega+\dot\varphi_u) L^{\|\sigma\|B^x_\sigma}(\diffns u,\diffns z)\notag\\
			&+\int_t^s\int_{\mathbb{R}} b_n(u,z,\omega+\dot\varphi_u) L^{\|\sigma\|B^x_\sigma}(\diffns u,\diffns z)\big|^2\big]^{\frac{1}{16}}\notag\\
			&\times \E\big[\big|e^{-\int_t^s \int_{\mathbb{R}} b(u,z,\omega+\dot\varphi_u) L^{\|\sigma\|B^x_\sigma}(\diffns u,\diffns z)}+e^{-\int_t^s\int_{\mathbb{R}} b_n(u,z,\omega+\dot\varphi_u) L^{\|\sigma\|B^x_\sigma}(\diffns u,\diffns z)}\big|^2\big]^{\frac{1}{16}}\notag\\
			\leq	&	C \E\big[\big|\int_t^s \int_{\mathbb{R}} b(u,z,\omega+\dot\varphi_u) L^{\|\sigma\|B^x_\sigma}(\diffns u,\diffns z)\notag\\
			&-\int_t^s\int_{\mathbb{R}} b_n(u,z,\omega+\dot\varphi_u) L^{\|\sigma\|B^x_\sigma}(\diffns u,\diffns z)\big|^2\big]^{\frac{1}{16}}\notag\\
			&	+ \E\big[\big|\int_t^s \int_{\mathbb{R}} b_1(z)b_2(u,\omega+\dot\varphi_u) L^{\|\sigma\|B^x_\sigma}(\diffns u,\diffns z)\notag\\
			&-\int_t^s\int_{\mathbb{R}} b_{1,n}(z)b_2(u,\omega+\dot\varphi_u) L^{\|\sigma\|B^x_\sigma}(\diffns u,\diffns z)\big|^2\big]^{\frac{1}{16}}\Big)\notag		
		\end{align}
		where the last inequality follows from \eqref{eqLTf1} (with $X$ been replaced by the Brownian motion $B$) and similar reasoning as in Lemma \ref{lem:auxlemma}. Thanks to Lemma \ref{lemeqra1}, we have
		$$
		\int_t^s\int_{\mathbb{R}} b_n(u,z,\omega+\dot\varphi_u) L^{\|\sigma\|B^x}(\diffns u,\diffns z)= -\int_t^s\frac{\partial b_n}{\partial x}(u, \|\sigma\|B^x_u,\omega+\dot\varphi_u)\diffns u.
		$$
		One can show that the first term on the right side in the above equation converges to $0$ (see \cite{BMBPD17, MenTan19}). As for the second term, define $\bar b(t,\omega,x):=b_2(x)b_3(t,\omega)$ and $\check b(t,\omega,x):=b_3(t,\omega)\int_{-\infty}^xb_2(y)\diffns y=b_3(t,\omega)\check b_2(x)$ (and define $\check b_n, \bar b_n, \check b_{2,n}$ similarly). Using \cite[Theorem 5.2]{Ein2000}, we have 
		\begin{align*}
			&	\check b(s,\|\sigma\|B^x_s,\omega+\dot\varphi_s)\\
			=&\check  b(t,\|\sigma\|B^x_t, \omega+\dot\varphi_t)+\int_t^s\frac{\partial \check  b}{\partial u}(u, \|\sigma\|B^x_u,\omega+\dot\varphi_u)\diffns u \\
			&+\|\sigma\| \int_t^s\frac{\partial \check b}{\partial x}(u, \|\sigma\|B^x_u,\omega+\dot\varphi_u)\diffns B^{x,\sigma}_u+\frac{\|\sigma\|^2}{2}  \int_t^s\frac{\partial^2 \check b}{\partial x^2}(u, \|\sigma\|B^x_u,\omega+\dot\varphi_u)\diffns u\\
			=&\check b(t,\|\sigma\|B^x_t, \omega+\dot\varphi_t)+\int_t^s\frac{\partial \check b}{\partial u}(u, \|\sigma\|B^x_u,\omega+\dot\varphi_u)\diffns u \\
			&+\|\sigma\| \int_t^s\bar b(u, \|\sigma\|B^x_u,\omega+\dot\varphi_u)\diffns B^{x,\sigma}_u+\frac{\|\sigma\|^2}{2} \int_t^s\frac{\partial \bar b}{\partial x}(u, \|\sigma\|B^x_u,\omega+\dot\varphi_u)\diffns u
		\end{align*}
		Therefore
		\begin{align*}
			&	\int_t^s\int_{\mathbb{R}} \bar b(u,z,\omega+\dot\varphi_u) L^{\|\sigma\|B^x}(\diffns u,\diffns z)\\
			=&-	\|\sigma\|^2\int_t^s\frac{\partial \bar b}{\partial x}(u, \|\sigma\|B^x_u,\omega+\dot\varphi_u)\diffns u	\\
			=&2\big(-\check b(s,\|\sigma\|B^x_s,\omega+\dot\varphi_s)+\check b(t,\|\sigma\|B^x_t, \omega+\dot\varphi_t)+\int_t^s\frac{\partial\check b}{\partial u}(u, \|\sigma\|B^x_u,\omega+\dot\varphi_u)\diffns u \\
			&+\|\sigma\| \int_t^s\bar b(u, \|\sigma\|B^x_u,\omega+\dot\varphi_u)\diffns B^{x,\sigma}_u\big)
		\end{align*}	
		Substituting this above and using H\"older, Burkholder inequalities  and the definition of $\tilde b$ gives
		\begin{align*}
			I_{2,n}
			\leq & C \big\{\E\big[\big|\check b(s,\|\sigma\|B^x_s,\omega+\dot\varphi_s)-\check b_n(s,\|\sigma\|B^x_s,\omega+\dot\varphi_s)\big|^2\big]^{\frac{1}{16}}\\
			&+\E\big[\big|\check b(t,\|\sigma\|B^x_t, \omega+\dot\varphi_t)- \check b_n(t,\|\sigma\|B^x_t, \omega+\dot\varphi_t)\big|^2\big]^{\frac{1}{16}}\\
			&+\big(\int_t^s\E\big[\big|\frac{\partial \check b}{\partial u}(u, \|\sigma\|B^x_u,\omega+\dot\varphi_u)-\frac{\partial \check b_n}{\partial u}(u, \|\sigma\|B^x_u,\omega+\dot\varphi_u)\big|^2 \big]\diffns u \big)^{\frac{1}{16}} \\
			&+\|\sigma\|^{\frac{1}{8}} \big( \int_t^s\E\big[\big|\bar b(u, \|\sigma\|B^x_u,\omega+\dot\varphi_u)-\bar b_n(u, \|\sigma\|B^x_u,\omega+\dot\varphi_u)\big|^2\big]\diffns u\big)^{\frac{1}{16}}\big\}\\	\leq & C\Big\{
			\E[|M(\omega)|^4]^{\frac{1}{32}}\E\big[\big| \check b_2(\|\sigma\|B^x_s)-  \check b_{2,n}(\|\sigma\|B^x_s)\big|^4\big]^{\frac{1}{32}}\\
			&+\E[|M(\omega)|^4]^{\frac{1}{32}}\E\big[\big| \check b_2(\|\sigma\|B^x_t)-  \check b_{2,n}(\|\sigma\|B^x_t)\big|^4\big]^{\frac{1}{32}}\\
			&+\E[|M(\omega)|^4]^{\frac{1}{32}}\big(\int_t^s\E\big[\big|b_2(\|\sigma\|B^x_u)-b_{2,n}( \|\sigma\|B^x_u)\big|^4 \big]\diffns u \big)^{\frac{1}{32}} \\
			&+\|\sigma\|^2 \E[|M(\omega)|^4]^{\frac{1}{32}}\big(\int_t^s\E\big[\big|b_2(\|\sigma\|B^x_u)-b_{2,n}( \|\sigma\|B^x_u)\big|^4 \big]\diffns u \big)^{\frac{1}{32}}\Big\}\\
			\leq & C (\|\sigma\|^2)\Big\{
			\E[|M(\omega)|^4]^{\frac{1}{32}}\big\{\E\big[\big|  \check b_2(s,\|\sigma\|B^x_s)-  \check b_{2,n}(s,\|\sigma\|B^x_s)\big|^4\big]^{\frac{1}{32}}\\
			&+\E\big[\big|  b_2(\|\sigma\|B^x_t)-   b_{2,n}(\|\sigma\|B^x_t)\big|^4\big]^{\frac{1}{32}}\notag\\
			&+\big(\int_t^s\E\big[\big|b_2(\|\sigma\|B^x_u)-b_{2,n}( \|\sigma\|B^x_u)\big|^4 \big]\diffns u \big)^{\frac{1}{32}} \big\}\Big\}.
		\end{align*}
		Since $ b_2$ and $b_{2,n}$ are uniformly bounded, if follows that each of the above terms goes to zero as $n$ goes to $\infty$ (see for e.g., \cite{BMBPD17, MenTan19}). The proof is completed. 
		
		\section{Proposition \ref{prop:EisenbaumBV}}\label{proofeisen1} This is a modification of the proof of \cite[Theorem 5.1]{Ein2000}.
		Define
		\begin{align*}
			I_{\Delta}=\sum\limits_{\begin{subarray}{ll}
					0\leq i\leq n\\0\leq j\leq m
			\end{subarray}}A_1(y_i)A_2(s_j)(L^{X^x}(s_{j+1},y_{i+1})-L^{X^x}(s_{j},y_{i+1})-L^{X^x}(s_{j},y_{i+1})+L^{X^x}(s_{j},y_{i})).
		\end{align*}
		We will show that
		\begin{align}\label{EqEisenBV}
			\lim\limits_{\vert\Delta\vert\to0}I_{\Delta}=&A_1(b)\int_0^tA_2(s)\,\mathrm{d}_sL^{X^x}(s,b)-A_1(a)\int_0^tA_2(s)\,\mathrm{d}_sL^{X^x}(s,a)\notag\\
			&-\int_a^b\Big(\int_0^tA_2(s)\,\mathrm{d}_sL^{X^x}(s,y)\Big)\mathrm{d}A_1(y).
		\end{align}
		We develop this sum according to the rules of integration by parts and we obtain
		\begin{align*}
			I_{\Delta}=&\sum\limits_{\begin{subarray}{ll}
					0\leq i\leq n\\0\leq j\leq m
			\end{subarray}}A_1(y_i)A_2(s_j)(L^{X^x}(s_{j+1},y_{i+1})-L^{X^x}(s_{j},y_{i+1}))\\&+\sum\limits_{\begin{subarray}{ll}
					0\leq i\leq n\\0\leq j\leq m
			\end{subarray}}A_1(y_i)A_2(s_j)(L^{X^x}(s_{j},y_{i})-L^{X^x}(s_{j},y_{i+1}))\\
			&+\sum\limits_{\begin{subarray}{ll}
					0\leq i\leq n\\0\leq j\leq m
			\end{subarray}}A_1(y_{i+1})A_2(s_{j})L^{X^x}(s_{j+1},y_{i+1})-\sum\limits_{\begin{subarray}{ll}
					0\leq i\leq n\\0\leq j\leq m
			\end{subarray}}A_1(y_{i+1})A_2(s_{j})L^{X^x}(s_{j+1},y_{i+1})\\&+\sum\limits_{\begin{subarray}{ll}
					0\leq i\leq n\\0\leq j\leq m
			\end{subarray}}A_1(y_{i+1})A_2(s_{j})L^{X^x}(s_{j},y_{i+1})-\sum\limits_{\begin{subarray}{ll}
					0\leq i\leq n\\0\leq j\leq m
			\end{subarray}}A_1(y_{i+1})A_2(s_{j})L^{X^x}(s_{j},y_{i+1}).
		\end{align*}
		One can put different terms together so that
		\begin{align*}
			I_{\Delta}=&\sum\limits_{\begin{subarray}{ll}
					0\leq i\leq n\\0\leq j\leq m
			\end{subarray}}A_2(s_{j})\Big\{A_1(y_{i+1})L^{X^x}(s_{j+1},y_{i+1})-A_1(y_{i})L^{X^x}(s_{j+1},y_{i})\\&\qquad\qquad\qquad\qquad-A_1(y_{i+1})L^{X^x}(s_{j},y_{i+1})+A_1(y_{i})L^{X^x}(s_{j},y_{i})\Big\}\\&+\sum\limits_{\begin{subarray}{ll}
					0\leq i\leq n\\0\leq j\leq m
			\end{subarray}}A_2(s_{j})\Big\{A_1(y_{i})L^{X^x}(s_{j+1},y_{i+1})-A_1(y_{i+1})L^{X^x}(s_{j+1},y_{i+1})\\&\qquad\qquad\qquad\qquad-A_1(y_{i})L^{X^x}(s_{j},y_{i+1})+A_1(y_{i+1})L^{X^x}(s_{j},y_{i+1})\Big\}\\
			=&\sum\limits_{\begin{subarray}{ll}
					0\leq i\leq n\\0\leq j\leq m
			\end{subarray}}A_2(s_{j})\Big\{A_1(y_{i+1})(L^{X^x}(s_{j+1},y_{i+1})-L^{X^x}(s_{j},y_{i+1}))\\&\qquad\qquad\qquad\qquad-A_1(y_i)(L^{X^x}(s_{j+1},y_{i})-L^{X^x}(s_{j},y_{i}))\Big\}\\&+\sum\limits_{\begin{subarray}{ll}
					0\leq i\leq n\\0\leq j\leq m
			\end{subarray}}A_2(s_{j})\Big\{(A_1(y_{i+1})-A_1(y_i))(L^{X^x}(s_{j},y_{i+1})-L^{X^x}(s_{j+1},y_{i+1}))\Big\}\\&:=I_{\Delta}^{(1)}+I_{\Delta}^{(2)}.
		\end{align*} 
		Summing over $i$ in $I^{(1)}_{\Delta}$, we have
		\begin{align*}
			I^{(1)}_{\Delta}=\sum\limits_{0\leq j\leq m}A_2(s_j)\left\{A_1(b)(L^{X^x}(s_{j+1},b)-L^{X^x}(s_{j},b))-A_1(a)(L^{X^x}(s_{j+1},a)-L^{X^x}(s_{j},a))\right\}.
		\end{align*}
		As a consequence,
		\begin{align*}
			\lim\limits_{\vert\Delta\vert\to0}I^{(1)}_{\Delta}=A_1(b)\int_0^tA_2(s)\,\mathrm{d}_sL^{X^x}(s,b)-A_1(a)\int_0^tA_2(s)\,\mathrm{d}_sL^{X^x}(s,a).
		\end{align*}
		Moreover, $I_{\Delta}^{(2)}$ rewrites
		\begin{align*}
			I_{\Delta}^{(2)}=-\sum\limits_{0\leq i\leq n}\Big(\sum\limits_{0\leq j\leq m}A_2(s_j)(L^{X^x}(s_{j+1},y_{i+1})-L^{X^x}(s_{j},y_{i+1}))\Big)(A_1(y_{i+1})-A_1(y_i)).
		\end{align*}
		This implies that
		\begin{align*}
			\lim\limits_{\vert\Delta\vert\to0}I^{(2)}_{\Delta}=-\int_a^b\Big(\int_0^tA_2(s)\,\mathrm{d}_sL^{X^x}(s,y)\Big)\mathrm{d}A_1(y).
		\end{align*}
		This ends the proof of \ref{EqEisenBV}. Letting $a$ goes to $-\infty$, $b$ goes to $\infty$, we obtain \ref{EqEisenBV00}.

		\section{Garsia-Rodemich-Rumsey theorem}
		The following result is a version of Garsia-Rodemich-Rumsey theorem (see \cite{KwaRo04}).
		\begin{thm}[Garsia-Rodemich-Rumsey]\label{LemGRR}
			Let $K$ be a compact interval endowed with a metric $\mathcal{D}$. Define $\widehat\sigma(r):=
			\underset{x\in K}{\inf} \lambda(B(x,r))$, where $B(x,r):=\{y\in K : \mathcal{D}(x,y) \leq r\}$ denotes the ball of radius $r$ centred in
			$x \in K$ and  $\lambda$ is the Lebesgue measure. Let $\Psi: [0,\infty)\rightarrow [0,\infty)$ be a positive, increasing and convex function with $\Psi(0) = 0$ and denote by $\Psi^{-1}$ its inverse, which is a positive, increasing
			and concave function. Assume $f:[0,\infty)\rightarrow [0,\infty)$ is continuous on $(K,d)$ and let
			$$
			U=\int_{K\times K}\Psi\big(\frac{|f(t)-f(s)|}{\mathcal{D}(s,t)}\big)\diff s \diff t.
			$$
			Then 
			\begin{align}\label{eqGRR1}
				|f(t)-f(s)|\leq 18 \int_0^{\frac{\mathcal{D}(t,s)}{2}}\Psi^{-1}\big(\frac{U}{\widehat \sigma^2(r)}\big)\diff r.
			\end{align}
			
		\end{thm}
		\bibliographystyle{siamplain}
		\bibliography{Biblio+}
		

	\end{document}